\documentclass[11pt, reqno]{amsart}

\usepackage{cite}
\usepackage[english]{babel}
\usepackage[T1]{fontenc}
\usepackage{csquotes}
\usepackage{amsmath, amssymb, amsthm, mathrsfs}
\usepackage[foot]{amsaddr}
\usepackage{enumitem}
\usepackage{xcolor}
\definecolor{MyLinkColor}{rgb}{0,0,0.4}

\newcommand{\dx}[1]{\,\mathrm{d}#1}

\newcommand{\nHp}[1]{\mathrm{H}^{#1}}
\newcommand{\hHp}[2]{\widehat{\mathrm{H}}\vphantom{\mathrm{H}}^{#1}(#2)}
\newcommand{\rmhH}{\widehat{\mathrm{H}}\vphantom{\mathrm{H}}}
\newcommand{\e}{\mathrm{e}}
\newcommand{\PV}{\mathrm{PV}\!}
\newcommand{\Rel}{\mathrm{Re}\,}

\newcommand{\Aa}{\mathbb{A}}
\newcommand{\ve}{\varepsilon}
\newcommand{\vt}{\vartheta}

\newcommand{\BB}{\mathbb{B}}
\newcommand{\CC}{\mathbb{C}}
\newcommand{\DD}{\mathbb{D}}
\newcommand{\NN}{\mathbb{N}}
\newcommand{\RR}{\mathbb{R}}
\newcommand{\Ss}{\mathbb{S}}
\newcommand{\TT}{\mathbb{T}}
\newcommand{\ZZ}{\mathbb{Z}}

\newcommand{\bfa}{\mathbf{a}}
\newcommand{\bfb}{\mathbf{b}}
\newcommand{\bfc}{\mathbf{c}}

\newcommand{\bfz}{\mathbf{z}}

\newcommand{\mcB}{\mathcal{B}}
\newcommand{\mcC}{\mathcal{C}}
\newcommand{\mcD}{\mathcal{D}}

\newcommand{\mcL}{\mathcal{L}}
\newcommand{\mcO}{\mathcal{O}}
\newcommand{\mcP}{\mathcal{P}}
\newcommand{\mcQ}{\mathcal{Q}}

\newcommand{\mcU}{\mathcal{U}}
\newcommand{\mcV}{\mathcal{V}}
\newcommand{\mcW}{\mathcal{W}}

\newcommand{\mcZ}{\mathcal{Z}}
\newcommand{\rmd}{\mathrm{d}}
\newcommand{\rmC}{\mathrm{C}}

\newcommand{\rmH}{\mathrm{H}}
\newcommand{\rmL}{\mathrm{L}}

\newcommand{\rmW}{\mathrm{W}}

\newcommand{\scV}{\mathscr{V}}

\newcommand{\wt}{\widetilde}
\newcommand{\tnu}{{\tilde\nu}}

\newcommand{\tkappa}{{\tilde\kappa}}

\newcommand{\ts}[1]{t_{[#1]}}
\newcommand{\tss}{\ts{s}}
\newcommand{\T}[1]{T_{[#1]}}

\newcommand{\Tf}[2]{T_{[#1]}#2}
\newcommand{\dg}[2]{\delta_{[#1]}#2}

\newcommand{\Txsf}{\Tf{\xi,s}{f}}

\newcommand{\rot}{\mathrm{rot}}

\DeclareMathOperator*{\supp}{supp}
\DeclareMathOperator*{\vdiv}{div}
\DeclareMathOperator*{\curl}{curl}

\newcommand{\zd}{.\kern-\nulldelimiterspace}

\newtheorem{Theorem}{Theorem}[section]
\newtheorem{Proposition}[Theorem]{Proposition}
\newtheorem{Lemma}[Theorem]{Lemma}
\newtheorem{Corollary}[Theorem]{Corollary}
\theoremstyle{remark}

\numberwithin{equation}{section}
\usepackage[colorlinks=true,linkcolor=MyLinkColor,citecolor=MyLinkColor]{hyperref}

\begin{document}
\title[Well-posedness and Rayleigh--Taylor instability for the two-phase  Stokes flow]{Well-posedness and Rayleigh--Taylor instability of the two-phase periodic quasistationary Stokes flow}
\thanks{Partially supported by DFG Research Training Group~2339 ``Interfaces, Complex Structures, and Singular Limits in Continuum Mechanics - Analysis and Numerics''}
\author[Daniel B\"ohme]{Daniel B\"ohme}
\author[Bogdan-Vasile Matioc]{Bogdan-Vasile Matioc}
\address{Fakult\"at f\"ur Mathematik, Universit\"at Regensburg \\ D--93040 Regensburg, Deutschland}
\email{daniel.boehme@ur.de}
\email{bogdan.matioc@ur.de}

\begin{abstract}
 We study the two-phase, horizontally periodic, quasistationary Stokes flow in two dimensions driven by surface tension and gravity effects in the general context of fluids with (possibly) different viscosities and densities. The sharp interface which separates the fluids is assumed to be the graph of a periodic function. The mathematical model is then recast as a fully nonlinear and nonlocal evolution equation involving only the function parametrizing the interface. Our main results include well-posedness and a parabolic smoothing property, as well as a study of equilibrium solutions in subcritical Sobolev spaces. In particular, we establish the Rayleigh--Taylor instability of small, finger-shaped equilibria and prove that the stability properties of flat interfaces depend on the sign of a certain parameter. 
\end{abstract}


\maketitle

\pagestyle{myheadings}
\markboth{\sc{D. B\"ohme \& B.-V.~Matioc}}{\sc{Well-posedness and Rayleigh--Taylor instability for the two-phase  Stokes flow}}

\section{Introduction}

The subject of this paper is the study of the two-phase, horizontally periodic, quasistationary Stokes flow in two spatial dimensions. We assume that the flow is driven by the capillarity of the interface $\Gamma(t)$ separating the two immiscible fluid phases as well as gravity effects. This leads us to the following system of equations
	\begin{subequations}\label{eq:STOKES}
	\begin{equation}\label{eq:Stokes}
		\left.
		\arraycolsep=1.4pt
		\begin{array}{rclll}
			\mu^\pm\Delta v^\pm(t)-\nabla q^\pm(t)&=&0&\mbox{in $\Omega^\pm(t)$,}\\
			\vdiv v^\pm(t)&=&0&\mbox{in $\Omega^\pm(t)$,}\\{}
			[v(t)]&=&0&\mbox{on $\Gamma(t)$,}\\{}
			[T_\mu (v(t),q(t))]\tnu(t)&=&(\Theta x_2-\sigma\tkappa(t)) \tnu(t)&\mbox{on $\Gamma(t)$,}\\
			(v^\pm(t),q^\pm(t))(x)&\to&\Big(\pm\frac{c_{1,\Gamma(t)}}{\mu^\pm},\pm\frac{c_{2,\Gamma(t)}}{\mu^\pm},\pm c_{3,\Gamma(t)}\Big)&\begin{minipage}{3cm} 
				for $x_2\to\pm\infty$ uni-\\[-0.5ex]
				formly in $x_1\in \Ss$,
			\end{minipage}\\
			V_n(t)&=&v(t)\cdot\tnu(t)&\mbox{on $\Gamma(t)$}
		\end{array}\right\}
	\end{equation}
	for $t>0$. 
	Here, $\Omega^+(t)$ denotes the two fluid phase above the interface $\Gamma(t)$, while $\Omega^-(t)$ is the one below. In particular, we have $\Gamma(t)=\partial\Omega^\pm(t)$ and $\Ss\times\RR=\Gamma(t)\cup\Omega^+(t)\cup\Omega^-(t)$ for $t>0$. We denote by $\tnu=\tnu(t)$ the unit normal exterior to $\Omega^-(t)$ while $\tkappa=\tkappa(t)$ denotes the curvature of the interface~${\Gamma(t)}$. Moreover, we assume that there exists a periodic function $f(t):\Ss\to\RR$ such that~${\Gamma(t)}$ is given by the graph of this function, i.e.
	\begin{equation*}
		\Gamma(t)=\{x=(x_1,x_2)\in\Ss\times\RR : x_2=f(t,x_1)\},\qquad t>0.
	\end{equation*}
	Here, $\Ss:=\RR/2\pi\ZZ$ denotes the unit circle, functions defined on $\Ss$ being interpreted as  $2\pi$-periodic on~$\RR$.
	The system \eqref{eq:Stokes} is supplemented with the initial condition
	\begin{equation}\label{eq:ic}
		f(0)=f_0.
	\end{equation}
	 In \eqref{eq:Stokes}, ${v^\pm =v^\pm(t):\Omega^\pm(t)\to\RR^2}$ denotes the velocity field of the liquid located in the phase~${\Omega^\pm(t)}$, while~$q^\pm=q^\pm(t):\Omega^\pm(t)\to\RR$ is given by
	\begin{equation*}
		q^\pm(t,x)=p^\pm(t,x) +g\rho^\pm x_2,\qquad x=(x_1,x_2)\in\Omega^\pm(t),
	\end{equation*}
	where $p^\pm=p^\pm(t)$ is the pressure in $\Omega^\pm(t)$. The positive constants $\rho^\pm$ and $\mu^\pm$ describe the density and viscosity of the fluid located in the phase $\Omega^\pm(t)$, respectively. Furthermore, we define the stress tensor $T_{\mu^\pm}(v^\pm,q^\pm)$ by
	\begin{equation}\label{eq:stress}
		T_{\mu^\pm}(v^\pm,q^\pm):= -q^\pm I_2 +\mu^\pm \big(\nabla v^\pm +(\nabla v^\pm)^\top\big),\qquad (\nabla v^\pm)_{ij} := \partial_j v^\pm_i,\quad i,\,j=1,\, 2.
	\end{equation}
	We let $[v]$ and $[T_\mu(v,q)]$ be the jump of the velocity and the stress tensor across the interface, respectively; see the definition~\eqref{eq:defjump} below. 
	Then, we introduce the constant
	\begin{equation}\label{eq:Theta}
		\Theta:=-g[\rho]\in\RR,
	\end{equation}
	where $g\geq 0$ is the Earth's gravity. Furthermore,~${\sigma>0}$  is the surface tension coefficient, $V_n$ is the normal velocity of the interface $\Gamma(t)$, $x\cdot y$ is the Euclidean scalar product of two vectors $x,y\in\RR^n$, and $I_n$ is the identity matrix in $\RR^{n\times n}$.

	The (spatially) constant vector $c_{\Gamma(t)} := (c_{1,\Gamma(t)}, c_{2,\Gamma(t)}, c_{3,\Gamma(t)}) \in \RR^3$, which characterizes the far-field behavior of both the velocity and the pressure, is a priori unknown and, together with $v^\pm(t)$ and $q^\pm(t)$, constitutes part of the solution. However, it turns out that not only $v^\pm(t)$ and $q^\pm(t)$, but also $c_{\Gamma(t)}$ are uniquely determined at each time~$ t>0$ by the function~$f(t)$; see Theorem~\ref{Thm:Ftp}. We therefore sometimes refer simply to $f=f(t)$ as the solution to~\eqref{eq:STOKES}. In contrast to the last two components of $ c_{\Gamma(t)}$, the first component $c_{1,\Gamma(t)}$ depends in a subtle way on $f(t)$ and reflects the average vorticity in $\Ss \times \RR$ at time~$t>0$, that is,
	\begin{equation}\label{eq:limits}
		c_{1,\Gamma(t)} = -\frac{\mu^+ \mu^-}{2\pi(\mu^+ + \mu^-)} \, \PV \int_{\Ss \times \RR} \curl v(t) \, \rmd x, \qquad c_{2,\Gamma(t)} = 0, \qquad c_{3,\Gamma(t)} = -\frac{\Theta}{2} \langle f(t) \rangle,
	\end{equation}
	see \eqref{eq:curl}--\eqref{eq:intmean}. We note that, when the fluid viscosities are equal, the dependence of $c_{1,\Gamma(t)}$ on $f(t)$ becomes explicit; see~\cite[Eq.~(1.1c)]{Bohme.2024}.
	\end{subequations}

 Research on quasistationary Stokes flow has its origins in the study of single-phase fluids evolving in sufficiently smooth domains~$\Omega(t)\subset\RR^d$, $d\geq2$. 
 In~\cite{PG97}, well-posedness near smooth and strictly star-shaped configurations and exponential stability of balls were established. Alternative proofs of exponential stability, based on power series techniques, were later provided in dimensions $d\in\{2,3\}$ in~\cite{Fr02a, Fr02b}. 
 More recently, exponential stability of balls was also shown in the planar two-phase situation, where one fluid encloses the other and the interface dynamics are governed by surface tension~\cite{Ch24}. In a different direction, the quasistationary Stokes equations arise as the singular zero–Reynolds number limit of the Navier–Stokes system; see~\cite{S99, Solo99}.

	The two-phase version of the problem in bounded geometries, including possible phase transitions and in arbitrary dimension $d\geq2$, is treated systematically in the monograph~\cite{PS16} while stabilization of balls using a feedback operator acting on the interface was studied in~\cite{Co24}.

	The non-periodic counterpart of~\eqref{eq:STOKES} with gravity effects  excluded and general viscosities was studied in~\cite{Matioc.2021, MP2022}, where local well-posedness was obtained in $\rmH^r(\RR)$ for exponents $r\in(3/2,2)$ arbitrarily close to the critical value $r=3/2$, cf.~\cite[Remark~1.2]{MP2022}. 
	The singular limit of the two-phase problem, when one of the viscosities tends to zero, was studied in~\cite{MP2023} and it was shown therein that it converges to the one-phase problem.
	
	When surface tension is neglected ($\sigma=0$), the problem reduces to an ODE formulation, yielding local~\cite{GGS25a} and global well-posedness~\cite{GGS25} in the setting of H\"older space. Importantly, the case $\sigma=0$ coincides with the transport Stokes system, a model for sedimentation of rigid particles in viscous fluids~\cite{H18, Me19, MS22, Le22, Gr23}.	
	
	A related line of research concerns the Peskin problem for an elastic string immersed in a viscous fluid, 
	which shares the same equations for the dynamics in the bulk as~\eqref{eq:STOKES} (with $g=0$). Recent results include~\cite{LT19, MRS19, CS24, CN23, GMS23, GGS23, EH25, EKMS25}.

 This paper extends the study~\cite{Bohme.2024}, which analyzed problem~\eqref{eq:STOKES} under the assumption of equal viscosities, to the more general setting of possibly distinct viscosities. In addition, we provide a comprehensive study of equilibrium solutions to~\eqref{eq:STOKES}. While in the non-periodic case~\cite{MP2022} neither gravity effects could be incorporated, nor the stability of equilibrium solutions was addressed, we still pursue the same strategy as therein for solving~\eqref{eq:STOKES}. Specifically, at each time step $t>0$, we determine the velocity, the pressure, and the constant $c_\Gamma$ from~\eqref{eq:Stokes} assuming that $f(t)\in \rmH^3(\Ss)$. For equal viscosities, this is done by constructing~$(v^\pm,q^\pm)$ via the hydrodynamical single-layer potential with gravity and capillary forces as density, c.f.~\cite{Bohme.2024,Ladyzhenskaya.1963}. 
	When the viscosities differ, this is not applicable since  $[\mu v]\neq0$ on $\Gamma$. 
	To handle this, we use the hydrodynamic double-layer potential, where the density solves a linear, singular integral equation of the second kind depending on the velocity field obtained from the single-layer potential; see~\eqref{eq:Db=V}. In Theorem~\ref{Thm:Ftp}, we prove that the fixed-time problem associated with~\eqref{eq:Stokes} admits a unique solution, given precisely by the sum of the single-layer and double-layer potentials; see~\eqref{eq:sol}. A substantial distinction between the periodic and non-periodic cases lies in the far-field boundary condition~\eqref{eq:Stokes}$_5$. In the non-periodic case, the velocity converges to zero at infinity; in contrast, in the periodic setting, the velocity profile becomes asymptotically horizontal as $x_2\to\pm\infty$, with its magnitude determined by the average vorticity in~$\Ss\times\RR$, but the profiles~$+\infty$ and~$-\infty$ have opposite signs. Similarly, for $x_2\to\pm\infty$, the pressure differs from the hydrostatic pressure by a constant depending on~$f$ taking opposite signs at $\pm\infty$. It is noteworthy that the constant describing the pressure deviation is influenced solely by gravity effects,  whereas capillary effects primarily determine the constant characterizing the far-field velocity. When the fluid viscosities differ, however, this velocity constant may also be affected by gravity in a non-trivial way; see~\eqref{eq:Ftp_c}.
	
	Using the kinematic boundary condition~\eqref{eq:Stokes}$_6$, we then reformulate~\eqref{eq:STOKES} as a single nonlocal, fully nonlinear evolution equation for $f$, the corresponding evolution operator being well-defined for $f\in \rmH^r(\Ss)$ with $r>3/2$. Exploiting both new mapping properties and those established in~\cite{Bohme.2024} for the underlying singular integral operators, we prove that the resulting evolution equation is parabolic in the sense of Proposition~\ref{Prop:Psi_Gen}. Applying the abstract parabolic theory from~\cite{Lunardi.1995}, we finally obtain local well-posedness  of~\eqref{eq:STOKES} together with a parabolic smoothing property. 
	
	The results are summarized in the following first main theorem.

	\begin{Theorem}\label{Thm:Main}
	Let $\Theta \in \RR$, $\sigma, \mu^+, \mu^- \in (0,\infty)$, $r \in (3/2, 2)$, and $f_0 \in \rmH^r(\Ss)$. Then, the following hold:
	\begin{enumerate}[label=\textup{(\roman*)}, leftmargin=*, itemsep=1ex]
    \item \textup{(Well-posedness)} 
    	There exists a unique maximal solution $(f, v^\pm, q^\pm,c_\Gamma)$ to~\eqref{eq:STOKES}  with
    	\begin{equation*}
    		f := f(\cdot, f_0) \in \rmC\big( [0, T_+), \rmH^r(\Ss) \big)\cap \rmC^1\big( [0, T_+), \rmH^{r-1}(\Ss) \big),
    	\end{equation*}
    	and such that for each $t \in (0, T_+)$ we have  $  f(t) \in \rmH^3(\Ss)$ and
    	\begin{equation*}
        	v^\pm(t) \in \rmC^2\big(\Omega^\pm(t),\RR^2 \big) \cap \rmC^1 \big( \overline{\Omega^\pm(t)}, \RR^2 \big),\qquad q^\pm(t) \in \rmC^1\big(\Omega^\pm(t)\big) \cap \rmC\big(\overline{\Omega^\pm(t)}\big),
    	\end{equation*}
    	where $T_+ = T_+(f_0) \in (0, \infty]$. 
	\item\textup{(Parabolic smoothing)}  
    	The function $f$ is smooth on $(0, T_+) \times \Ss$, that is,
   		\begin{equation*}
   			\big[(t,\xi) \mapsto f(t,\xi)\big] \in \rmC^\infty\big( (0, T_+) \times \Ss, \RR \big).
   		\end{equation*}
    \item \textup{(Global existence)}  
    	The solution is global, that is $T_+(f_0) = \infty$, provided that for each $T>0$,
    	\begin{equation*}
    		\sup_{t \in [0, T] \cap [0, T_+(f_0))}\lVert f(t) \rVert_{\nHp{r}} < \infty.
    	\end{equation*}
	\end{enumerate}
	\end{Theorem}

	Before presenting our main result on the existence and stability properties of equilibrium solutions to \eqref{eq:STOKES}; see Theorem~\ref{Thm:Stability}, we note that the set of solutions to \eqref{eq:STOKES} is invariant under horizontal and vertical translations. Indeed,  if $(f,v^\pm,q^\pm, c_\Gamma)$ is a solution to \eqref{eq:STOKES}, then $(\tilde f, \tilde{v}^\pm,\tilde{q}^\pm, c_{\tilde \Gamma})$ with
\begin{subequations}\label{eq:sol_trans}
	\begin{equation}\label{eq:sol_trans1}
		\tilde f(t,\xi)= f(t,\xi-a)+c,\qquad \tilde v^\pm(t,x)=v^\pm(t,x-(a,c)),\qquad\tilde q^\pm(t,x)=q^\pm(t,x-(a,c))\mp\frac{c\Theta}{2}
	\end{equation}
	and
	\begin{equation}\label{eq:sol_trans2}
	c_{\tilde \Gamma}=c_\Gamma-(0,0,c\Theta/2)
	\end{equation}
	\end{subequations}
	is also a solution to \eqref{eq:STOKES}. Furthermore, the integral mean $\langle f\rangle$ is preserved by the flow since, by~\eqref{eq:Stokes}$_2$,~\eqref{eq:Stokes}$_6$,  and~\eqref{eq:limits}
	\begin{equation}\label{eq:pres}
		\frac{{\rmd}\langle f\rangle}{{\rmd}t}(t)=\int_{\Gamma(t)}v(t)\cdot\tnu(t)\dx{\sigma}=\int_{\Omega^\pm(t)}\vdiv v^\pm(t)\dx{x}=0,\qquad t\in(0,T_+).
	\end{equation}
	The properties~\eqref{eq:sol_trans}–\eqref{eq:pres} will  be considered when analyzing the stability of equilibrium solutions. In particular, we focus on equilibrium solutions with zero integral mean and examine their stability under perturbations that also have zero integral mean. To this end, we define 
	\begin{equation*}
		\hHp{s}{\Ss}:=\{f\in\rmH^s(\Ss):\langle f\rangle =0\},\qquad s\geq 0.
	\end{equation*}
	Surprisingly, the set of equilibrium solutions to the Stokes flow~\eqref{eq:STOKES} coincides with that of the periodic Muskat problem with surface tension; see Theorem~\ref{Thm:SolBif} and ~\cite{Escher.2011, Ehrnstrom.2013, Sa23, Matioc.2020}. 
	Equilibrium solutions to \eqref{eq:STOKES} with zero integral mean are characterized by the fact that $f\in\rmC^\infty(\Ss)$ solves the capillarity equation
	\begin{equation}\label{eq:cap_eq}
		\kappa(f) + \lambda f = 0, \qquad \text{where}~ \lambda :=\frac{g[\rho]}{\sigma},
	\end{equation}
	the complete picture of the equilibrium solutions to \eqref{eq:STOKES} being provided in Section~\ref{Sec:6}.

	In Theorem~\ref{Thm:Stability}, we show that the trivial equilibrium $f=0$ is exponentially stable if $g[\rho]/\sigma<1$, which happens either when the fluid below is denser than the fluid above, or when $[\rho]> 0$ is smaller than~$\sigma/g$. If, however, $g[\rho]/\sigma>1$, the equilibrium $f=0$ is unstable in the sense that there exists a neighborhood of zero in $\hHp{r}{\Ss}$  and solutions starting arbitrarily close to $0$ that eventually leave this neighborhood. This reflects the fact that gravity amplifies small disturbances  when $\rho^+ >\rho^-$, a phenomenon known as the Rayleigh-–Taylor instability; see \cite{LR82, T50}. Similar instabilities have also been extensively studied in related two-phase problems, including the Muskat problem \cite{Escher.2011, Matioc.2020}, the Verigin problem \cite{PSW19}, and the full Navier–Stokes equations \cite{JTW16, PS10, GT11, TW12, W22}.

	Moreover, if $g[\rho]/\sigma$ exceeds a threshold value $\lambda_*\in(0,1)$; see \eqref{eq:lambdastar}, the Stokes flow~\eqref{eq:STOKES} may also possess finger-shaped equilibria which are located on global bifurcation branches (of solutions to \eqref{eq:cap_eq}). 
	These describe configurations where the heavier fluid positioned above intrudes into the lighter fluid below, forming characteristic finger-like patterns. 
	In Theorem~\ref{Thm:Stability} we prove, by applying the  exchange of stability principle due to Crandall and Rabinowitz \cite{Crandall.1973},  that  the small finger-shaped equilibrium solutions, 
	which have a cosine-shaped profile, also feature the Rayleigh--Taylor instability property. 
	The stability properties of finger-shaped equilibria with large slopes (which become unbounded as one approaches the end of the global bifurcation branches) are not addressed by Theorem~\ref{Thm:Stability}, 
and this remains an open problem.

	\begin{Theorem}[Exponential stability/Rayleigh--Taylor instability]\label{Thm:Stability}
	Let $\Theta\in\RR$,  $\sigma,\, \mu^+,\, \mu^-\in(0,\infty)$, and $r\in(3/2,2)$ be given.
	\begin{enumerate}[label=\textup{(\roman*)}]
	\item\textup{(Stability properties of the trivial solution $f=0$)}
		 Let 
		 \begin{equation}\label{eq:theta0}
			\vartheta_0:=\frac{\sigma+\Theta}{2(\mu^+ +\mu^-)}\mathbf{1}_{[0,\infty)}(\sigma-\Theta)+\frac{\sqrt{\sigma\Theta}}{\mu^+ +\mu^-}\mathbf{1}_{(0,\infty)}(\Theta-\sigma).
		\end{equation}
		\begin{enumerate}[label=\textup{(\alph*)}]
		\item \textup{(Exponential stability)}	
			Assume~$\sigma+\Theta>0$ and choose~$\vartheta\in (0, \vartheta_0)$. Then, there exist constants~$\delta>0$ and~$M>0$, such that for any~${f_0\in\hHp{r}{\Ss}}$ satisfying $\lVert f_0\rVert_{\rmH^r}<\delta,$ the solution to \eqref{eq:STOKES} exists globally, i.e. $T_+(f_0)=\infty$ and
			\begin{equation*}
				\lVert f(t)\rVert_{\rmH^r}+\bigg\lVert \frac{\rmd f}{\rmd t}(t)\bigg\rVert_{\rmH^{r-1}}\leq M e^{-\vartheta t}\lVert f_0\rVert_{\rmH^r}\qquad \text{for all}~t\geq 0.
			\end{equation*}			
		\item \textup{(Rayleigh--Taylor instability)}
			If $\sigma+\Theta<0$, then the zero solution is unstable. \smallskip
		\end{enumerate}
	\item \textup{(Rayleigh--Taylor instability of small finger-shaped equilibria)} 
		For any $\ell\in\NN$, there exists a smooth bifurcation curve $(\lambda_\ell,f_\ell):(-\delta_\ell,\delta_\ell)\to\RR\times \hHp{r}{\Ss}$, $\delta_\ell >0$, such that
		\begin{equation*}
			\left\{
			\begin{array}{lll}
				\lambda_\ell(s):=\ell^2 -\dfrac{3\ell^4}{8}s^2 +\mcO(s^4)\quad\text{in}~\RR,\\[1ex]
				f_\ell(s):= s\cos(\ell\cdot)+\mcO(s^2)\quad\text{in}~\hHp{r}{\Ss},
			\end{array}\right.
			\qquad\text{for}~s\to 0,
		\end{equation*}
		with $f_\ell(s)$ being an even equilibrium to \eqref{eq:STOKES} if $\Theta=-\sigma\lambda_\ell(s)$. 
		The equilibrium $f_\ell(s)$ is unstable if $0\leq |s|<\delta_\ell$ and $\delta_\ell$ is sufficiently small.
	\end{enumerate}
	\end{Theorem}
	
		Lastly, we fix some notation which will be used throughout this paper.
\subsection*{Notation} 
	Let $E_1, \ldots, E_n, E, F$ be Banach spaces, with $n \in \mathbb{N}$. We write $\mathcal{L}^n(E_1, \ldots, E_n, F)$ for the Banach space of bounded $n$-linear maps from $\prod_{i=1}^n E_i$ into $F$. In the special case when~${E_i = E}$ for all $1 \leq i \leq n$, we abbreviate this by $\mathcal{L}^n(E, F)$, and let $\mathcal{L}_{\mathrm{sym}}^n(E, F)$ denote the subspace of symmetric operators. Given an open set $\mathcal{U} \subset E$, we write $\rmC^{1-}(\mathcal{U}, F)$ for the space of locally Lipschitz continuous mappings, and $\rmC^{\infty}(\mathcal{U}, F)$ for the space of smooth mappings from~$\mathcal{U}$ to $F$.

	For a function $w=(w_1,w_2):\mathcal{O}\to\RR^2$, where $\mathcal{O}$ is an open subset in $\Ss\times\RR$, we set 
	\begin{equation}\label{eq:curl}
	 \curl w:=\partial_1w_2-\partial_2w_1,
	\end{equation}
	and, given $f:\Ss\to\RR$, 
	\begin{equation}\label{eq:intmean}
		\langle f\rangle:=\frac{1}{2\pi}\int_{-\pi}^\pi f(\xi)\, \rmd\xi
	\end{equation}
	is the integral mean of $f$.

 \subsection*{Outline}
	 In Section~\ref{Sec:2} and Appendix~\ref{Sec:B}, we prove the unique solvability of a stationary two-phase Stokes problem with prescribed velocity jump across the interface, using the hydrodynamic double-layer potential, and also treat a corresponding homogeneous problem. Sections~\ref{Sec:3}--\ref{Sec:4} (based on results from  Appendix~\ref{Sec:A}) analyze the resolvent set of the double-layer potential operator in~$\rmL^2(\Ss)$ and in higher-order Sobolev spaces. In Section~\ref{Sec:5}, we solve the fixed-time problem corresponding to~\eqref{eq:Stokes}, reformulate~\eqref{eq:STOKES} as an evolution equation for $f$, and establish its parabolic character, leading to the proof of Theorem~\ref{Thm:Main}. Finally, in Section~\ref{Sec:6}, we identify and study the equilibrium solutions of~\eqref{eq:STOKES} and prove Theorem~\ref{Thm:Stability}.

\section{Layer potentials and the unique solvability of a boundary value problem with transmission boundary conditions}\label{Sec:2}

	In the first part of this section, we consider the following boundary value problem with boundary conditions of transmission type
	\begin{equation}\label{eq:Ftp_b}
		\left.
		\arraycolsep=1.4pt
		\begin{array}{rclll}
		\Delta v^\pm-\nabla q^\pm&=&0&\mbox{in $\Omega^\pm$,}\\
		\vdiv v^\pm&=&0&\mbox{in $\Omega^\pm$,}\\{}
		[v]&=&\beta\circ\Xi^{-1}&\mbox{on $\Gamma$,}\\{}
		[T_1 (v,q)]\tnu&=&0&\mbox{on $\Gamma$,}\\
		(v^\pm,q^\pm)(x)&\to&\pm( 	c_1,c_2,0)&\begin{minipage}{3cm} 
			for $x_2\to\pm\infty$ uni-\\[-0.5ex]
			formly in $x_1\in \Ss$,
		\end{minipage}
		\end{array}\right\}
	\end{equation}
	where $f\in\rmH^{3}(\Ss)^2$,  $\beta\in\rmH^{2}(\Ss)^2$, and $c_1,\, c_2\in\RR$, and show that problem \eqref{eq:Ftp_b} has a classical solution in the sense of \eqref{eq:clsol} (which is unique) only for a particular choice of the constants $c_1,\, c_2$. Moreover, the solution admits an explicit integral representation; see Theorem~\ref{Thm:Ftp_b} below. In the second part of this section, we investigate in Proposition~\ref{Prop:Ftp_0} the unique solvability of a homogeneous boundary value problem associated with~\eqref{eq:Stokes}$_1$--\eqref{eq:Stokes}$_5$. These results play a central role in the proof of Theorem~\ref{Thm:Ftp}, where we show, in the context of~\eqref{eq:Stokes}, that the function  parametrizing the interface between the fluids determines, at each moment, the velocity and pressure in both fluid layers.

	We begin by introducing some notation and conventions that will be used throughout the paper. In \eqref{eq:Ftp_b} we set
	\begin{equation*}
		\Omega^\pm:=\{x=(x_1 ,x_2)\in\Ss\times\RR: x_2\gtrless f(x_1)\},\qquad \Gamma:= \{(\xi, f(\xi))\in\Ss\times\RR: \xi\in\Ss\},
	\end{equation*}
	and  $\Xi:= \Xi_f: \Ss\to\Gamma$ is the $\rmC^{5/2}$-diffeormorphism  $\Xi:= (\text{id}_\Ss, f)$. We further define
	\begin{equation}\label{eq:nutauomega}
	 	\omega:=\omega(f):= (1+f'^2)^{1/2},\quad	\nu:=(\nu_1,\nu_2):=\nu(f) :=\omega^{-1}(-f', 1)^\top, \quad \tau:=\tau(f) :=\omega^{-1}(1,f')^\top.
	\end{equation}

	Given a function $w:(\Ss\times\RR)\setminus\Gamma\to\RR^k$, $k\in\NN$, we set $w^\pm:=w|_{\Omega^\pm}$ and  denote by
	\begin{equation}\label{eq:deflimits}
		\{w\}^\pm(x_0):= \lim_{\Omega^{\pm} \ni x\to x_0}w(x),\qquad x_0\in\Gamma,
	\end{equation}
	the one-sided limits of  $w$ in $x_0\in\Gamma$, provided that these limits exist. 
	Moreover,
	\begin{equation}\label{eq:defjump}
		[w](x_0):= \{w\}^+(x_0) -\{w\}^-(x_0)  
	\end{equation}
	is the jump of $w$ across $\Gamma$ in $x_0\in\Gamma$.
	
	Vice versa, we associate to functions $w^\pm:\Omega^\pm\to\RR^k$, $k\in\NN$, the mapping  $w:={\bf 1}_{\Omega^+}w^++{\bf 1}_{\Omega^-}w^-$, which is   defined a.e. in $\Ss\times\RR$.  

	By a classical solution to \eqref{eq:Ftp_b} (and to the other transmission problems considered later on) we mean a pair $(v,q)\in X_f$, where 
	\begin{equation}\label{eq:clsol}
		X_f:=\Bigg\{(v,q):(\Ss\times\RR)\setminus\Gamma\to\RR^2\times\RR\,\Bigg|\,
		\begin{aligned}
			&v^{\pm}\in \rmC^2(\Omega^\pm,\RR^2)\cap \rmC^1(\overline{\Omega^\pm},\RR^2)\\
			&q^\pm \in \rmC^1(\Omega^\pm)\cap \rmC(\overline{\Omega^\pm})
		\end{aligned}
		\Bigg\}.
	\end{equation}

	We further  recall  that the $x_1$-periodic Stokeslet~$(\mcU,\mcP)$,  with~${\mcU=(\mcU_j^k)_{1\leq  j,k\leq 2}}$ and~$\mcP=(\mcP^1,\mcP^2)^\top$,  is given by 
	\begin{equation}\label{eq:UP_by_z}
		\mcU(x)=\frac{1}{8\pi}
		\begin{pmatrix}
			z_0(x)+x_2 z_2(x) & -x_2 z_1(x)\\
			-x_2 z_1(x) & z_0(x) -x_2 z_2(x)
		\end{pmatrix},\qquad
		\mcP=-\frac{1}{4\pi}(z_1,z_2)^\top
	\end{equation}
	 for  $x=(x_1,x_2)\in(\Ss\times\RR)\setminus\{0\}$; see \cite{Bohme.2024, GGS25a}, where, using the shorthand notation
	\begin{equation}\label{eq:notation1}
		t_{[x_1]}:=\tan\!\left(\frac{x_1}{2}\right),\quad x_1\in\RR\setminus(\pi+2\pi\mathbb{Z}),\qquad\text{and}\qquad T_{[x_2]}:=\tanh\!\left(\frac{x_2}{2}\right),\quad x_2\in\RR,
	\end{equation}
	we set
	\begin{equation}\label{eq:defz}	
	\begin{aligned}
		&z_0(x):=\ln\bigg(\frac{\ts{x_1}^2+\T{x_2}^2}{(1+\ts{x_1}^2)(1-\T{x_2}^2)}\bigg),\qquad
		z_1(x):=\frac{\ts{x_1}(1-\T{x_2}^2)}{\ts{x_1}^2+\T{x_2}^2},\qquad
		z_2(x):=\frac{(1+\ts{x_1}^2)\T{x_2}}{\ts{x_1}^2+\T{x_2}^2},\\
		&z_3(x):=x_2z_5(x),\qquad
		z_5(x):=\frac{(1+\ts{x_1}^2)(1-\T{x_2}^2)(\ts{x_1}^2-\T{x_2}^2)}{2(\ts{x_1}^2+\T{x_2}^2)^2},\\
		&z_4(x):=x_2z_6(x),\qquad z_6(x):=\frac{\ts{x_1}\T{x_2}(1+\ts{x_1}^2)(1-\T{x_2}^2)}{2(\ts{x_1}^2+\T{x_2}^2)^2}.
	\end{aligned}
	\end{equation}
	 Noticing that,
	\begin{equation}\label{eq:grad_z_i}
		\nabla z_0=(z_1,z_2)^\top,\qquad \nabla z_1=(-z_5,-2z_6)^\top,\qquad \nabla z_2=(-2z_6,z_5)^\top,
	\end{equation}
	the partial derivatives of $(\mcU,\mcP)$ are  given by 
	\begin{equation}\label{eq:derivU}
	\begin{aligned}
		&\partial_1 \mcU=\frac{1}{8\pi}
		\begin{pmatrix}
			z_1 - 2 z_4 &  z_3\\
			 z_3 & z_1 + 2 z_4
		\end{pmatrix},
		&&\quad\partial_1 \mcP=\frac{1}{4\pi}(z_5,2z_6)^\top ,
		\\
		&\partial_2 \mcU=\frac{1}{8\pi}
		\begin{pmatrix}
			2z_2 +  z_3 & -z_1 + 2 z_4\\
			-z_1+ 2 z_4 & - z_3
		\end{pmatrix},
		&&\quad \partial_2 \mcP=\frac{1}{4\pi}(2z_6,-z_5)^\top.
	\end{aligned}
	\end{equation}
	Since $(\mcU^k,\mcP^k)$ with $\mcU^k:=(\mcU^k_1,\mcU^k_2)^\top$, $k=1,2$, solve the  homogeneous Stokes equations
	 \begin{equation}\label{eq:hom_Stokes}
 	 	\left.
 	 	\arraycolsep=1.4pt
 	 	\begin{array}{rcl}
 	 		\Delta U -\nabla P &=& 0,\\[.5ex] 	
 	 		\vdiv U &=&0
 	 	\end{array}\right\}
 	 	\qquad \text{in $(\Ss\times\RR)\setminus\{0\}$};
 	 \end{equation}
	 see \cite[Eq.~(2.9)]{Bohme.2024}, straightforward computations show that~$(\mcW^{i,k},\mcQ^{i,k}):(\Ss\times\RR)\setminus\{0\}\to\RR^2\times\RR$ with  
	\begin{equation}\label{eq:defWQ}
		\mcW^{i,k}_j:=-\mcP^k \delta_{ij}+\partial_i \mcU_j^k+\partial_j \mcU_i^k,\quad j=1,2,\qquad \mcQ^{i,k}:=2\partial_i\mcP^k,
	\end{equation}
	are  solutions to \eqref{eq:hom_Stokes} for  $i,k=1,2$. Setting $\mcW_j:=(\mcW^{i,k}_j)_{1\leq i,k\leq 2}$ and $\mcQ:=(\mcQ^{i,k})_{1\leq i,k\leq 2}$, we infer from \eqref{eq:UP_by_z} and \eqref{eq:derivU} that
	\begin{equation}\label{eq:WQexp}
		\mcW_1=\frac{1}{4\pi}
 		\begin{pmatrix}
 			2z_1-2 z_4 & z_2+z_3\\
			z_2+z_3 & 2 z_4
		\end{pmatrix},\quad
 		\mcW_2=\frac{1}{4\pi}
 		\begin{pmatrix}
 			z_2+z_3 & 2 z_4\\
 			2 z_4 & z_2-z_3
 		\end{pmatrix},\quad
 		\mcQ=\frac{1}{2\pi}
 		\begin{pmatrix}
 			z_5 & 2z_6\\
 			2z_6 & -z_5
 		\end{pmatrix}.
	\end{equation}
 	To solve problem \eqref{eq:Ftp_b}, we introduce the hydrodynamic double-layer potential   $(v,q)=(v,q)[\beta]$, with~${v=(v_1,v_2)^\top}$, generated by the density $\beta\circ\Xi^{-1}$, cf. \cite{Ladyzhenskaya.1963}, defined for $ x\in\Omega^\pm$ by
 	\begin{equation}\label{eq:defwq}
 	\begin{aligned}
 		v_j(x):=v_j[\beta](x)&:=\int_{-\pi}^\pi \mcW_j^{i,k}(r)\nu_i(s)\beta_k(s)\omega(s)\,\rmd s,\qquad j=1,2,\\
 		q(x):=q[\beta](x)&:=\int_{-\pi}^\pi \mcQ^{i,k}(r)\nu_i(s)\beta_k(s)\omega(s)\,\rmd s,
 	\end{aligned}
 	\end{equation}
 	with
 	\begin{equation*}
 		r=r(x,s):=x-(s,f(s)),\qquad x\in (\Ss\times\RR)\setminus\Gamma,\quad s\in\Ss.
 	\end{equation*}
 	Defining  for $\varphi\in\rmL^2(\Ss)$ the  integral operators $Z_i(f)[\varphi]:(\Ss\times\RR)\setminus\Gamma\to\RR,$ $1\leq i\leq 4$, by 
 	\begin{equation}\label{eq:defZ}
 		Z_i(f)[\varphi](x):=\frac{1}{2\pi}\int_{-\pi}^\pi z_i(r)\varphi(s)\,\rmd s;
 	\end{equation}
	see \cite[Lemma~C.1]{Bohme.2024}, and  using the shorthand $Z_i :=Z_i(f)$, we have, using obvious matrix notation,
 	\begin{equation}\label{eq:w_by_Z}
 	v=v[\beta]=\frac{1}{2}
 	\begin{pmatrix}
 		Z_2+Z_3& 2Z_4\\
 		2Z_4&Z_2-Z_3
 	\end{pmatrix}
 	\begin{bmatrix}
	 	\beta_1\\
	 	\beta_2
 	\end{bmatrix}
 	-\frac{1}{2}
 	\begin{pmatrix}
 		2Z_1-2Z_4& Z_2+Z_3\\
 		Z_2+Z_3&2Z_4
 	\end{pmatrix}
 	\begin{bmatrix}
 		f'\beta_1\\
 		f'\beta_2
 	\end{bmatrix}.
 	\end{equation}
	Since \eqref{eq:derivU} and \eqref{eq:defWQ} imply that 
	\[
	\mcQ^{i,k}(r)\nu_i(s)\beta_k(s)\omega(s)= 2\partial_s\big(\mcP^1(r)\big)\beta_2(s)-2\partial_s\big(\mcP^2(r)\big)\beta_1(s),\qquad s\in\Ss, 
	\]
	 we obtain from~\eqref{eq:UP_by_z} and integration by parts in \eqref{eq:defwq} the formula
 	\begin{equation}\label{eq:q_by_Z}
 		q=q[\beta]=Z_1(f)[\beta_2']-Z_2(f)[\beta_1'].
 	\end{equation}
 	The pair $(v,q)$ defined in~\eqref{eq:defwq} is the solution  to the boundary value problem~\eqref{eq:Ftp_b}, as stated in the following result.

 	\begin{Theorem}\label{Thm:Ftp_b}
 	Given $f\in\rmH^3(\Ss)$ and $\beta\in\rmH^2(\Ss)^2$,  problem~\eqref{eq:Ftp_b} has a solution~${(v,q)\in X_f}$ if and only if the constants $c_1,\, c_2$ in \eqref{eq:Ftp_b}$_5$ are given by
	\begin{equation}\label{eq:c1c2}
		c_1 = \frac{\langle \beta_1-f'\beta_2\rangle}{2}\qquad\text{and}\qquad c_2= \frac{\langle \beta_2-f'\beta_1\rangle}{2}.
	\end{equation}
	Assuming \eqref{eq:c1c2}, the solution $(v,q)$  is unique and is given by \eqref{eq:defwq} (or \eqref{eq:w_by_Z}--\eqref{eq:q_by_Z}).	
		
	Moreover, we have
	\begin{equation}\label{eq:stress_inf}
		T_1 (v^\pm,q^\pm)(x)\to0 \qquad\text{for $x_2\to\pm\infty$ uniformly in $x_1\in \Ss$.}
	\end{equation}	 
 	\end{Theorem}

 	\begin{proof}
 	That our solutions satisfy the equations \eqref{eq:Ftp_b}$_{1-2}$ and are unique follows as in the proof of \cite[Theorem~2.2]{Bohme.2024}. The boundary conditions \eqref{eq:Ftp_b}$_{3-4}$ as well as the far-field conditions \eqref{eq:Ftp_b}$_5$ and~\eqref{eq:stress_inf} are established in Lemma~\ref{Lem:vq_bd} and Lemma~\ref{Lem:vq_ff}.
 	\end{proof}

 	To establish the unique solvability of  the boundary value problem~\eqref{eq:Stokes}$_1$--\eqref{eq:Stokes}$_5$ (at a fixed time~$t>0$), we  consider  the  homogeneous boundary value problem
 	\begin{equation}\label{eq:Ftp_0}
		\left.
		\arraycolsep=1.4pt
		\begin{array}{rclll}
		\Delta v^\pm-\nabla q^\pm&=&0&\mbox{in $\Omega^\pm$,}\\
		\vdiv v^\pm&=&0&\mbox{in $\Omega^\pm$,}\\{}
		\mu^- v^+ - \mu^+ v^-&=&0&\mbox{on $\Gamma$,}\\{}
		[T_1 (v,q)]\tnu&=&0&\mbox{on $\Gamma$,}\\
		(\mu^\mp v^\pm,q^\pm)(x)&\to&\pm c& \begin{minipage}{3cm} 
			for $x_2\to\pm\infty$ uni-\\[-0.5ex]
			formly in $x_1\in \Ss$,
		\end{minipage}
		\end{array}\right\}
 	\end{equation}
 	with $c\in\RR^3$, and prove that  \eqref{eq:Ftp_0} 	has a  solution $(v, q)\in X_f$  if and only if $c=0$, the solution  being moreover unique. Note that \eqref{eq:Ftp_0} and \eqref{eq:Stokes}$_1$--\eqref{eq:Stokes}$_5$  (with the right side of \eqref{eq:Stokes}$_4$ replaced by $0$) are equivalent problems.

 	\begin{Proposition}\label{Prop:Ftp_0}
	Given $f\in\rmC^1(\Ss)$ and $c\in\RR^3$, the transmission boundary value problem~\eqref{eq:Ftp_0} has a  solution $(v, q)\in X_f$  if and only if $c=0$. In this case, the solution is unique (and trivial).
	\end{Proposition}

	\begin{proof}
	Assume that  \eqref{eq:Ftp_0} has a solution $(v,q)\in X_f$ for some $c=(c_1,c_2,c_3)\in\RR^3$. Then, in view of~\eqref{eq:Ftp_0}$_3$, we have 
	\begin{equation*}
		\psi:=\mu^-v^+{\bf 1}_{\Omega^+} +\mu^+v^-{\bf 1}_{\Omega^-}\in \rmC(\Ss\times\RR)^2\cap \rmH^{1}_{\rm loc}(\Ss\times\RR)^2.
	\end{equation*}
	Given $\phi\in\rmC(\Ss\times\RR)^2\cap \rmH^{1}(\Ss\times\RR)^2 $ with compact support, we compute, using  Stokes' theorem and~\eqref{eq:Ftp_0}$_2$, that   
	\begin{equation*}
	\begin{aligned}
		& \int_{\Omega^\pm} \phi^\pm \cdot\big(\Delta v^\pm -\nabla q^\pm\big)\, \rmd x \pm \int_{\Gamma} {\phi^\pm}^\top \big(-q^\pm I_2+\nabla v^\pm+(\nabla v^\pm)^\top\big) \tnu \, \rmd \sigma\\
  		&=-\frac{1}{2}\int_{\Omega^\pm} \big(\nabla \phi^\pm +(\nabla \phi^\pm)^\top\big ):\big(\nabla v^\pm +(\nabla v^\pm)^\top\big )\, \rmd x+\int_{\Omega^\pm}   q^\pm\vdiv\phi^\pm\, \rmd x,
	\end{aligned}
	\end{equation*}
	where $A:B$ is the Frobenius inner product of matrices $A,\, B\in\RR^{2\times 2}$, with $|A|:=\sqrt{A:A}$ being the associated norm. Summing up, we obtain in view of \eqref{eq:Ftp_0}$_1$ and \eqref{eq:Ftp_0}$_4$, using also the continuity of $\phi$, that
	\begin{equation}\label{eq:test} 
		\int_{\Ss\times\RR}\bigg(\frac{1}{\mu^-}{\bf 1}_{ \Omega^+} +\frac{1}{\mu^+}{\bf 1}_{ \Omega^-}\bigg)  \big(\nabla \psi +(\nabla \psi)^\top\big ):\big(\nabla \phi +(\nabla \phi)^\top\big )\, \rmd x = 2\int_{\Ss\times\RR}  q\vdiv\phi\, \rmd x.
	\end{equation}
 
	Let $u\in \rmC^{\infty}(\RR,[0,1])$ be an even function with $u(t)=1$ for $|t|\leq 1$ and $u(t)=0$ for $|t|\geq 2$. For each integer $n> \lVert f\rVert_\infty$, we define the even function $u_n\in \rmC^{\infty}(\RR,[0,1])$  by $u_n(t)=u(t-n+1)$ for~$t\in[n,n+1]$,    $u_n(t)=1$ for $|t|\leq n$, and~$u_n(t)=0$ for~$|t|\geq n+1$. Then, setting
	\begin{equation*}
		\phi_n(x):=u_n(x_2)\psi(x),\qquad x\in\Ss\times\RR,
	\end{equation*} 
	we conclude that $\phi_n\in \rmC(\Ss\times\RR)^2\cap \rmH^1(\Ss\times\RR)^2$ has compact support. 

	Using $\phi_n$ as a test function in \eqref{eq:test} and noticing by~\eqref{eq:Ftp_0}$_2$ that  for a.e. $x\in\Ss\times\RR$
	\begin{equation*}
	\begin{aligned}
		\vdiv\phi_n (x)&=u_n'(x_2)\psi_2(x) ,\\
		\nabla \phi_n +(\nabla \phi_n)^\top(x)&=u_n(x_2)\big(\nabla \psi +(\nabla \psi)^\top\big)(x)+ u_n'(x_2)
		\begin{pmatrix}
			0&\psi_1\\
			\psi_1&2\psi_2
		\end{pmatrix} (x), 
	\end{aligned}
	\end{equation*}
	we arrive at
	\begin{equation}\label{eq:testres} 
		\int_{\Ss\times\RR} {\bf 1}_{\{|x_2|< n\}} \big| \nabla \psi +(\nabla \psi)^\top\big|^2\, \rmd x\leq  \max\{\mu^+,\mu^-\}R_n,
	\end{equation}
	where $R_n:=R_n^I+R_n^{II}$ and
	\begin{equation*}
	\begin{aligned}
		R_n^I&:=2\int_{\Ss\times\RR}   u_n'(x_2)(q\psi_2)(x)\, \rmd x,\\
		R_n^{II}&:=-\int_{\Ss\times\RR}\bigg(\frac{1}{\mu^-}{\bf 1}_{\Omega^+} +\frac{1}{\mu^+}{\bf 1}_{\Omega^-}\bigg)u_n'(x_2)  \big(\nabla \psi +(\nabla \psi)^\top\big ):
		\begin{pmatrix}
 			0&\psi_1\\
			\psi_1&2\psi_2
		\end{pmatrix} (x)\, \rmd x.
	\end{aligned}
	\end{equation*}
	Since $(v,q)\in X_f$, it is straightforward to infer from \eqref{eq:Ftp_0}$_2$ and \eqref{eq:Ftp_0}$_5$  that $(R_n)$ is a bounded (actually convergent) sequence, implying that 
	\begin{equation}\label{eq:sg_psi_L2}
		\nabla\psi+(\nabla \psi)^\top\in\rmL^2(\Ss\times\RR,\RR^{2\times 2}).
	\end{equation}
	Moreover, it holds that
	\begin{equation}\label{eq:korn}
		\int_{\Ss\times\RR}|\nabla \phi_n|^2\,\rmd x+\int_{\Ss\times\RR}|\vdiv \phi_n |^2\,\rmd x=\frac{1}{2}\int_{\Ss\times\RR}|\nabla \phi_n +(\nabla \phi_n)^\top|^2\,\rmd x.
	\end{equation}
	Since \eqref{eq:sg_psi_L2} ensures that the right side of \eqref{eq:korn} defines a bounded sequence in~$\RR$, we deduce that $\nabla\psi\in\rmL^2(\Ss\times\RR,\RR^{2\times 2})$. Exploiting this property, the oddness of $u'$, and \eqref{eq:Ftp_0}$_5$ we easily obtain that $R_n\to0$ for $n\to\infty$. Consequently, we may infer from \eqref{eq:testres} that $\nabla \psi+(\nabla \psi)^\top=0$ in~$\rmL^2(\Ss\times\RR,\RR^{2\times 2})$. In view of this property, the  identity \eqref{eq:korn} may be reformulated as 
	\begin{equation*}
		\int_{\Ss\times\RR}u_n^2(x_2)|\nabla \psi|^2(x)\,\rmd x+\int_{\Ss\times\RR}(u_n^2)'(x_2) (\psi \cdot\partial_2\psi)(x)\,\rmd x= 0,
	\end{equation*}
	the second term on the left side converging, due to \eqref{eq:Ftp_0}$_5$, $(v,q)\in X_f$, and~$\nabla\psi\in\rmL^2(\Ss\times\RR,\RR^{2\times 2})$, to~$0$ as~$n\to\infty$. This in turn implies that $\nabla\psi=0$ and together with \eqref{eq:Ftp_0}$_5$ we conclude   that~$\psi=0$ and~$(c_1,c_2)=0$. Since \eqref{eq:Ftp_0}$_1$ and \eqref{eq:Ftp_0}$_5$ ensure that $q=0$ and $c_3=0$, this completes the proof. 
	\end{proof}

	\section{The \texorpdfstring{$\rmL^2$}{L2}--resolvent of the double-layer potential operator}\label{Sec:3}	

	In this section, we introduce the double-layer potential operator $\DD(f)$ associated with the Stokes operator 
	and the periodic graph $\Gamma=\{(\xi,f(\xi))\,:\,\xi\in\Ss\}$, where $f\in \rmC^1(\Ss)$ is an arbitrary function. 
	We prove that $\DD(f)\in \mcL \big(\rmL^2(\Ss)^2\big)$, and subsequently analyze its resolvent set; see Theorem~\ref{Thm:D(f)_Res_2}. 
	This result is used in Section~\ref{Sec:4} where, under stronger regularity assumptions on $f$,  we further investigate these properties within different $\rmC^*$-algebraic frameworks. 
	
	The analysis in this part uses properties of a particular class of (singular) integral operators introduced in \cite{Bohme.2024}, which we now recall. Given $m,\,n,\,p,\,q\in\NN_0$ satisfying $p\leq n+q+1$ and Lipschitz continuous mappings~${\bfa=(a_1,\dots,a_m):\RR\to\RR^m}$,~${\bfb=(b_1,\dots,b_n):\RR\to\RR^n,}$~${\bfc=(c_1,\dots,c_q):\RR\to\RR^q}$, we define the integral operators
	\begin{equation}\label{eq:Bnmpq}
		B_{n,m}^{p,q}(\bfa\vert \bfb)[\bfc,\varphi](\xi):= \frac{1}{2\pi}\PV\int_{-\pi}^{\pi}\frac{\prod\limits_{i=1}^{n}\frac{\T{\xi,s}b_i}{\tss}\prod\limits_{i=1}^{q}\frac{\dg{\xi,s}{c_i}/2}{\tss}}{\prod\limits_{i=1}^{m}\Big[1+\Big(\frac{\T{\xi,s}a_i}{\tss}\Big)^2\Big]}\frac{\varphi(\xi-s)}{\tss}\tss^p\,\rmd s
	\end{equation}
	and
	\begin{equation}\label{eq:Cnm}
		C_{n,m}(\bfa)[\bfb,\varphi](\xi):=\frac{1}{\pi}\PV\int_{-\pi}^\pi \frac{\prod\limits_{i=1}^{n}\frac{\dg{\xi,s}{b_i}}{s}}{\prod\limits_{i=1}^{m}\Big[1+\Big(\frac{\dg{\xi,s}{a_i}}{s}\Big)^2\Big]}\frac{\varphi(\xi-s)}{s}\,\rmd s,
	\end{equation}
 	where $\varphi\in\rmL^2(\Ss)$ and $\xi\in\RR$. The operators $B_{n,m}^{p,q}$ are the building blocks of the double-layer potential operator $\DD(f)$; see \eqref{eq:B_by_Bnmpq} and~\eqref{eq:D(f)}, while the operators $C_{n,m}$ are used in its analysis. We adopt the shorthand notation introduced in \eqref{eq:notation1} together with 
	\begin{equation}\label{eq:notation2} 
		\dg{\xi,s}{f}:= f(\xi)-f(\xi-s)\qquad\text{and}\qquad\T{\xi,s}f:=\tanh\bigg(\frac{\dg{\xi,s}{f}}{2}\bigg),\qquad \xi,\, s\in\RR.
	\end{equation}
	
	The mappings $B_{n,m}^{p,q}(\bfa\vert \bfb)[\bfc,\varphi]$ and $C_{n,m}(\bfa)[\bfb,\varphi]$ are $ 2\pi $-periodic, provided that $\bfa$, $\bfb$, and $\bfc$ are themselves $2\pi$-periodic. Moreover, the periodic Hilbert transform $H$; see, e.g., \cite{Torchinsky.2004, Butzer.1971}, appears in a natural way through the identity
	\begin{equation}\label{eq:HT}
 		B_{0,0}^{0,0}[\varphi](\xi)=\frac{1}{2\pi}\PV\int_{-\pi}^\pi\frac{\varphi(\xi-s)}{t_{[s]}}\,\rmd s= H[\varphi](\xi),\qquad \xi\in\RR.
	\end{equation}
	Whenever  all coordinate functions of $\bfa$, $\bfb$, and~$\bfc$ are  identical to a given function $f\in \rmW^{1,\infty}(\Ss)$, we set for simplicity
	\begin{equation}\label{eq:B(f)}
  		B_{n,m}^{p,q}(f):=B_{n,m}^{p,q}(f,\ldots, f\vert f,\ldots,f)[f,\ldots,f,\cdot],\qquad 0\leq p\leq n+q+1,
 	\end{equation}
	and
  	\begin{equation}\label{eq:Cnm0}
		C_{n,m}^0(f):=C_{n,m}(f,\ldots,f)[f,\ldots, f,\cdot].
	\end{equation}
 	To shorten notation, we introduce the operators
	\begin{equation}\label{eq:B_by_Bnmpq}
	\begin{aligned}
		B_1(f)&:=B_{0,1}^{0,0}(f)-B_{2,1}^{2,0}(f),\\
		B_2(f)&:=B_{1,1}^{0,0}(f)+B_{1,1}^{2,0}(f),\\
		B_3(f)&:=B_{0,2}^{0,1}(f)+B_{0,2}^{2,1}(f)-B_{2,2}^{0,1}(f)-2B_{2,2}^{2,1}(f)-B_{2,2}^{4,1}(f)+B_{4,2}^{2,1}(f)+B_{4,2}^{4,1}(f),\\	
		B_4(f)&:=B_{1,2}^{0,1}(f)+B_{1,2}^{2,1}(f)-B_{3,2}^{2,1}(f)-B_{3,2}^{4,1}(f),\\
		B_5(f)&:=2\big(B_{0,1}^{1,1}(f)-B_{2,1}^{3,1}(f)\big),\\
		B_6(f)&:=2\big(B_{1,1}^{1,1}(f)+B_{1,1}^{3,1}(f)\big),
	\end{aligned}
	\end{equation}	
	which  build the double-layer potential operator $\DD(f)$; see \eqref{eq:D(f)} below. Moreover, related to these operators,
	 we define for $f\in \rmW^{1,\infty}(\Ss)$ and $\varphi\in\rmL^2(\Ss)$  the function
	\begin{equation}\label{eq:B0}
		B_0(f)[\varphi](\xi):=\frac{1}{2\pi}\int_{-\pi}^\pi\ln\bigg(\frac{\tss^2+(\Txsf)^2}{(1+\tss^2)(1-(\Txsf)^2)}\bigg)\varphi(\xi-s)\,\rmd s,\qquad \xi\in\Ss.
	\end{equation}
	 
	The following result recalls the~$\rmL^2(\Ss)$-boundedness of the operators $B_{n,m}^{p,q}(\bfa\vert \bfb)[\bfc,\cdot]$.
	\begin{Lemma}\label{Lem:Bnmpq_L2_L2}
	Given $n,\,m,\,p,\, q\in\NN_0$ with $p\leq n+q+1$ and~$(\bfa, \bfb)\in\rmW^{1,\infty}(\Ss)^{m+n}$, there exists a constant $C>0$ that depends only on $n,\,m,\,p,\,q$,  and $\lVert(\bfa',\bfb')\rVert_\infty$ such that for all $\bfc\in\rmW^{1,\infty}(\Ss)^{q}$ and~$\varphi\in\rmL^2(\Ss)$ we have 
	\begin{equation}\label{eq:Bnmpq_L2_L2}
		\big\lVert B_{n,m}^{p,q}(\bfa\vert \bfb)[\bfc,\varphi]\big\rVert_2\leq C\lVert\varphi\rVert_2 \prod_{i=1}^{q}\lVert c_i'\rVert_\infty.
	\end{equation}
	Moreover, $B_{n,m}^{p,q}\in\rmC^{1-}\big(\rmW^{1,\infty}(\Ss)^{m+n},\mcL^q_{\textup{sym}}\big(\rmW^{1,\infty}(\Ss),\mcL\big(\rmL^2(\Ss)\big)\big)\big)$.
	\end{Lemma}

	\begin{proof}
		See \cite[Lemma~A.2, Lemma~A.3, and Lemma~A.5]{Bohme.2024}.
	\end{proof}
	
	We point out that the $\rmL^2$-adjoint $(B_{n,m}^{p,q}(\bfa\vert\bfb)[\bfc,\cdot])^*$ of $B_{n,m}^{p,q}(\bfa\vert\bfb)[\bfc,\cdot]$ is given by the relation
	\begin{equation}\label{eq:Bnmpq-ad}
		(B_{n,m}^{p,q}(\bfa\vert\bfb)[\bfc,\cdot])^*=(-1)^{p+1}B_{n,m}^{p,q}(\bfa\vert\bfb)[\bfc,\cdot].
	\end{equation}
	
	Given $f\in\rmW^{1,\infty}(\Ss)$ and $\beta\in\rmL^2(\Ss)^2$, we introduce the double-layer potential operator $\DD(f)$ by
	\begin{equation}\label{eq:D(f)}
		\DD(f)[\beta]:=-\frac{1}{2}
		\begin{pmatrix}
			B_2+B_3&2B_4\\
			2B_4&B_2-B_3
		\end{pmatrix}\begin{bmatrix}
		\beta_1\\
		\beta_2
		\end{bmatrix}
		+\frac{1}{2}\begin{pmatrix}
			2B_1-2B_4&B_2+B_3\\
			B_2+B_3&2B_4
		\end{pmatrix}\begin{bmatrix}
		f'\beta_1\\
		f'\beta_2
		\end{bmatrix},
	\end{equation}
	where $B_i:=B_i(f)$, $1\leq i\leq 4$, are defined in \eqref{eq:B_by_Bnmpq}. Using \eqref{eq:Bnmpq-ad} it can be easily verified that the~$\rmL^2$-adjoint of~$\DD(f)$ is given by
	\begin{equation}\label{eq:D(f)-ad}	
			\DD(f)^*[\beta]:=\frac{1}{2}
		\begin{pmatrix}
			B_2+B_3&2B_4\\
			2B_4&B_2-B_3
		\end{pmatrix}\begin{bmatrix}
		\beta_1\\
		\beta_2
		\end{bmatrix}
		-\frac{f'}{2}\begin{pmatrix}
			2B_1-2B_4&B_2+B_3\\
			B_2+B_3&2B_4
		\end{pmatrix}\begin{bmatrix}
		\beta_1\\
		\beta_2
		\end{bmatrix}.
	\end{equation}
 	Moreover, we introduce the singular integral operators $\BB_1(f),\,\BB_2(f)\in\mcL\big(\rmL^2(\Ss)\big)$, which  also appear in the analysis of the periodic Muskat problem;  see~\cite{Matioc.2020}, by setting for $\varphi\in\rmL^2(\Ss)$
 	\begin{equation}\label{eq:BB_i}
 		\BB_1(f)[\varphi]:=-B_1(f)[f'\varphi]+B_2(f)[\varphi]\qquad\text{and}\qquad \BB_2(f)[\varphi]:=B_1(f)[\varphi]+B_2(f)[f'\varphi].
 	\end{equation}
 	
 	A key result in the analysis of~\eqref{eq:STOKES} is the following invertibility result. 
 	\begin{Theorem}\label{Thm:D(f)_Res_2}
	Given $\delta\in(0,1)$, there exists a constant $C=C(\delta)>0$ with the property that for all~$\lambda\in\RR$ with~$|\lambda|\geq 1/2+\delta$ and $f\in\rmC^1(\Ss)$ with $\lVert f'\rVert_\infty\leq 1/\delta$ we have
	\begin{equation}\label{eq:D(f)_Res_2}
		\big\lVert\big(\lambda-\DD(f)^*\big)[\beta]\big\rVert_2\geq C\lVert\beta\rVert_2,\qquad \beta\in\rmL^2(\Ss)^2.
	\end{equation}
	Moreover,  $\lambda-\DD(f)^*,\, \lambda-\DD(f)\in\mcL\big(\rmL^2(\Ss)^2\big)$ are isomorphisms  for all~$\lambda\in\RR$ satisfying~$|\lambda|> 1/2$  and all~${f\in\rmC^1(\Ss)}$.
	\end{Theorem}

	The proof of Theorem~\ref{Thm:D(f)_Res_2} is deferred to the end of the section, as it requires some preparatory work.  One of the key ingredients is the following result.

 	\begin{Lemma}\label{Lem:Df_1}
 	Given $K>0$, there exists a constant $C=C(K)>0$ with the property that for all~${\beta\in\rmL^2(\Ss)^2}$,~${\lambda\in[-K,K]}$, and $f\in\rmC^1(\Ss)$ satisfying $\lVert f'\rVert_\infty\leq K$ we have
 	\begin{equation}\label{eq:Df_1}
 	\begin{aligned}
 		C\big\lVert\big(\lambda-\DD(f)^*\big)[\beta]\big\rVert_2 \lVert\beta\rVert_2&\geq m(\lambda)\lVert\omega^{-1}\beta\cdot\tau\rVert_2^2+ \langle\beta_1\rangle^2\\
 		&\quad+\big\lVert\big(\lambda-\tfrac{1}{2}\BB_1(f)\big)[\omega^{-1}\beta\cdot\nu]-\tfrac{1}{2}\BB_2(f)[\omega^{-1}\beta\cdot\tau]\big\rVert_2^2,
 	\end{aligned}
 	\end{equation}
 	where $m(\lambda):=\max\big\{\big(\lambda+\frac{1}{2}\big)\big(\lambda-\frac{3}{2}\big),\big(\lambda-\frac{1}{2}\big)\big(\lambda+\frac{3}{2}\big)\big\}$.
 	\end{Lemma}

 	\begin{proof}
 	 Recalling~\eqref{eq:B_by_Bnmpq} and Lemma~\ref{Lem:Bnmpq_L2_L2}, it suffices to establish \eqref{eq:Df_1} for $f\in\rmC^\infty(\Ss)$ with $\lVert f'\rVert_\infty\leq K$ and~$\beta\in\rmC^\infty(\Ss)^2$. To proceed, we define the hydrodynamic single-layer potential $(u,p)=(u,p)[\beta]$, generated by the density $\beta=(\beta_1,\beta_2)^\top$, by setting for  $x\in(\Ss\times\RR)\setminus\Gamma$
	\begin{equation*}
	\begin{aligned}
		u(x)&:=\int_{-\pi}^{\pi}\mcU^k(x-(s,f(s)))\beta_k(s)\,\rmd s,\\
 	 	p(x)&:= \int_{-\pi}^{\pi}\mcP^k(x-(s,f(s)))\beta_k(s)\,\rmd s,\\
	\end{aligned}
	\end{equation*}
	with $\mcU,\,\mcP$ defined in \eqref{eq:UP_by_z}. Arguing as in the proof of \cite[Theorem~2.2]{Bohme.2024}, we may deduce that the mappings~${u^\pm\in \rmC^\infty(\Omega^\pm,\RR^2)}$, $p^\pm\in\rmC^\infty(\Omega^\pm)$ have extensions~${u^\pm\in\rmC^1(\overline{\Omega^\pm},\RR^2)}$, $p^\pm\in\rmC(\overline{\Omega^\pm})$, being solutions to the homogeneous Stokes system 
	\begin{equation}\label{eq:hom_Stokes_up}
 	\left.
 	\arraycolsep=1.4pt
 	\begin{array}{rcl}
		\Delta u^\pm -\nabla p^\pm &=& 0,\\
 	 	\vdiv u^\pm &=&0
 	\end{array}\right\}
 	\qquad \text{in}~\Omega^\pm.
 	\end{equation}
	Using \eqref{eq:UP_by_z} and \eqref{eq:derivU} together with the formula~\eqref{eq:Z_pm}, we have
 	\begin{equation}\label{eq:du_p_BV}
 	\begin{aligned}
		\{\partial_1 u\}^\pm\circ\Xi&=\frac{1}{4}
		\begin{pmatrix}
 	 		B_1-2B_4&B_3\\
 	 		B_3&B_1+2B_4
 	 	\end{pmatrix}
 	 	\begin{bmatrix}
			\beta_1\\
			\beta_2
		\end{bmatrix}
		\pm\frac{\nu_1(\beta\cdot\tau)\tau}{2\omega},\\
		\{\partial_2 u\}^\pm\circ\Xi&=\frac{1}{4}
 	 	\begin{pmatrix}
 	 		2B_2+B_3&-B_1+2B_4\\
 	 		-B_1+2B_4&-B_3
 	 	\end{pmatrix}
 	 	\begin{bmatrix}
			\beta_1\\
			\beta_2
		\end{bmatrix}
		\pm\frac{\nu_2(\beta\cdot\tau)\tau}{2\omega},\\
		\{p\}^\pm\circ\Xi&=-\frac{1}{2}\big(B_1[\beta_1]+B_2[\beta_2]\big) \mp \frac{\beta\cdot\nu}{2\omega},
	\end{aligned}
	\end{equation}
	with $\nu=(\nu_1,\nu_2),\,\tau$, $\omega$  defined in \eqref{eq:nutauomega}	and	 $B_i:=B_i(f)$, $1\leq i\leq 4$, defined in~\eqref{eq:B_by_Bnmpq}. 
	We also note that if we discard the jump terms in \eqref{eq:du_p_BV} and replace for $1\leq i\leq4$ the operator~$B_i(f)$ by~$Z_i(f)$,  then~\eqref{eq:du_p_BV}  provides a formula for $\partial_i u $, $i=1,\,2$,  and $p$ in $(\Ss\times\RR)\setminus\Gamma$. Moreover, letting~$\TT(f)\in\mcL(\rmL^2(\Ss)^2)$  be the integral operator in the formula \eqref{eq:du_p_BV} for $\{\partial_2 u\}^\pm\circ\Xi$, that is
	\begin{equation*}
		\TT(f)[\beta]:=\frac{1}{4}
 	 	\begin{pmatrix}
			2B_2+B_3&-B_1+2B_4\\
 	 		-B_1+2B_4&-B_3
 	 	\end{pmatrix}
 	 	\begin{bmatrix}
			\beta_1\\
			\beta_2
		\end{bmatrix},
 	 \end{equation*}
	then, due to~\eqref{eq:Bnmpq-ad}, we deduce that $\TT(f)^*=-\TT(f)$, and therefore
	\begin{equation}\label{eq:T_SkA}
		\langle \TT(f)[\beta] \,|\,\beta\rangle_2=0,\qquad \beta\in\rmL^2(\Ss)^2,
	\end{equation}
	where $\langle\cdot \,|\,\cdot\rangle_2$  denotes here and below the $\rmL^2(\Ss)^2$-inner product. Using \eqref{eq:du_p_BV}, we find for the normal stress at the interface that
	\begin{equation}\label{eq:normalStress}
		\omega\big(\{T_1(u,p)\}^\pm\circ\Xi\big)\nu=\bigg(\pm\frac{1}{2}+\DD(f)^*\bigg)[\beta].
	\end{equation}
	Since $(u^\pm,p^\pm)$ are solutions to~\eqref{eq:hom_Stokes_up}, a straightforward computation leads us to the identity
	\begin{equation}\label{eq:divT}
		\vdiv \big(T_1(u,p)\partial_2u\big)=\frac{1}{4} \partial_2 |p I_2+T_1(u,p)|^2 \qquad\text{in $(\Ss\times\RR)\setminus\Gamma$.}
	\end{equation}
	Recalling~\eqref{eq:Zi_lim}, we infer from~\eqref{eq:du_p_BV} (and the discussion following it) that for $x_2\to\pm\infty$ we  have
	\begin{equation}\label{eq:Td2u_lim}
		\big(T_1(u^\pm,p^\pm)\partial_2u^\pm\big)(x)\to\frac{\langle\beta_1\rangle}{4}
		\begin{pmatrix}
			\langle\beta_2\rangle\\
			\langle\beta_1\rangle
		\end{pmatrix}
		\qquad\text{and}\qquad |p^\pm I_2+T_1(u^\pm,p^\pm )|^2(x)\to\frac{\langle\beta_1\rangle^2}{2}
	\end{equation}		 
	uniformly with respect to $x_1\in\Ss$. Integrating the relation~\eqref{eq:divT} over
	\begin{equation}\label{eq:Omega_n}
		\Omega^\pm_n:=\Omega^\pm\cap \{x=(x_1,x_2)\in\Ss\times\RR\,:\, |x_2|<n \},\qquad n>\lVert f\rVert_\infty,
	\end{equation}
	and using Stokes' theorem, we obtain in virtue of \eqref{eq:Td2u_lim}, after letting $n\to\infty$,   
	\begin{equation}\label{eq:T_Gauss}
		4\big\langle \omega\big(\{T_1(u,p)\}^\pm\circ\Xi\big)\nu \,|\, \{\partial_2u\}^\pm\circ\Xi \big\rangle_2= \int_{-\pi}^\pi\big|\big(\{p\}^\pm\circ\Xi I_2+\{T_1(u,p )\}^\pm\circ\Xi\big)\big|^2\,\rmd s + \pi\langle\beta_1\rangle^2 .
	\end{equation}
		
	The relations~\eqref{eq:du_p_BV}--\eqref{eq:normalStress} together with the property that $\TT(f)\in\mcL\big(\rmL^2(\Ss)^2\big)$ ensure that there exists a constant $C=C(K)>0$ such that the left side of \eqref{eq:T_Gauss} satisfies 
	\begin{equation*}
	\begin{aligned}
		4\big\langle \omega\big(\{T_1(u,p)\}^\pm\circ\Xi\big)\nu \,|\, \{\partial_2u\}^\pm\circ\Xi \big\rangle_2
		&=4\bigg\langle \bigg(\pm\frac{1}{2}+\DD(f)^*\bigg)[\beta] \,\bigg|\, \TT(f)[\beta]\pm\frac{\nu_2(\beta\cdot\tau)\tau}{2\omega} \bigg\rangle_2\\
		&=4\bigg\langle\bigg(\lambda\pm\frac{1}{2}\bigg)\beta-\big(\lambda-\DD(f)^*\big)[\beta]\,\bigg|\,\TT(f)[\beta]\pm\frac{\nu_2(\beta\cdot\tau)\tau}{2\omega}\bigg\rangle_2\\
		&\leq C\big\lVert\big(\lambda-\DD(f)^*\big)[\beta]\big\rVert_2 \lVert\beta\rVert_2 \pm 2\bigg(\lambda\pm\frac{1}{2}\bigg)\lVert\omega^{-1}\beta\cdot\tau\rVert_2^2.
	\end{aligned}
	\end{equation*}
		
	Concerning the integral term on the right side of \eqref{eq:T_Gauss}, we note that 
	\begin{equation*}
	\begin{aligned}
		I&:=\int_{-\pi}^\pi\big|\{p\}^\pm\circ\Xi I_2+\{T_1(u,p )\}^\pm\circ\Xi\big|^2\,\rmd s \geq\int_{-\pi}^\pi\big|\big(\{p\}^\pm\circ\Xi\big)\nu+\big(\{T_1(u,p )\}^\pm\circ\Xi\big)\nu\big|^2\,\rmd s\\
		&=\bigg\lVert\big(\{p\}^\pm\circ\Xi\big)\nu+\frac{1}{\omega}\bigg(\pm\frac{1}{2}+\DD(f)^*\bigg)[\beta]\bigg\rVert_2^2.
	\end{aligned}
	\end{equation*}				
	Since $\beta=(\beta\cdot\tau)\tau+(\beta\cdot\nu)\nu$ and observing that
	\begin{equation*}
		B_1(f)[\beta_1]+B_2(f)[\beta_2]=\BB_1(f)[\omega^{-1}\beta\cdot\nu]+\BB_2(f)[\omega^{-1}\beta\cdot\tau], 
	\end{equation*}
	H\"{o}lder's inequality, the uniform bounds $|\lambda|\leq K$ and~${\lVert\BB_i(f)\rVert_{ \mcL(\rmL^2(\Ss))}\leq C(K)}$,~${i=1,2}$ (see \eqref{eq:B_by_Bnmpq} and Lemma~\ref{Lem:Bnmpq_L2_L2}), and the formulas~\eqref{eq:du_p_BV} and \eqref{eq:normalStress} lead us now to
	\begin{equation*}
	\begin{aligned}
		I&\geq\Big\lVert \big(\{p\}^\pm\circ\Xi\big)\nu+\frac{1}{\omega}\Big(\lambda\pm\frac{1}{2}\Big)\beta-\frac{1}{\omega}\big(\lambda-\DD(f)^*\big)[\beta] \Big\rVert_2^2\\
		&=\Big\lVert \Big(\frac{1}{2}\big(\BB_1(f)[\omega^{-1}\beta\cdot\nu]+\BB_2(f)[\omega^{-1}\beta\cdot\tau]\big)\pm \frac{\beta\cdot\nu}{2\omega}\Big)\nu-\frac{1}{\omega}\Big(\lambda\pm\frac{1}{2}\Big)\beta+\frac{1}{\omega}\big(\lambda-\DD(f)^*\big)[\beta] \Big\rVert_2^2\\
		&\geq\Big\lVert \Big(\frac{1}{2}\big(\BB_1(f)[\omega^{-1}\beta\cdot\nu]+\BB_2(f)[\omega^{-1}\beta\cdot\tau]\big) \pm \frac{\beta\cdot\nu}{2\omega}\Big)\nu-\frac{1}{\omega}\Big(\lambda\pm\frac{1}{2}\Big)\beta \Big\rVert_2^2\\	
		&\qquad+\Big\lVert\frac{1}{\omega}\big(\lambda-\DD(f)^*\big)[\beta]\Big\rVert_2^2-C(K)\big\lVert\big(\lambda-\DD(f)^*\big)[\beta]\big\rVert_2\lVert\beta\rVert_2\\
		&\geq \Big\lVert -\frac{1}{\omega}\Big(\lambda\pm\frac{1}{2}\Big)(\beta\cdot\tau)\tau+\Big(\frac{1}{2}\BB_1(f)-\lambda\Big)[\omega^{-1}\beta\cdot\nu]\nu+\frac{1}{2}\BB_2(f)[\omega^{-1}\beta\cdot\tau]\nu\Big\rVert_2^2\\
		&\qquad-C\big\lVert\big(\lambda-\DD(f)^*\big)[\beta]\big\rVert_2\lVert\beta\rVert_2\\
		&= \Big(\lambda\pm\frac{1}{2}\Big)^2\lVert\omega^{-1}\beta\cdot\tau\rVert_2^2+\Big\lVert\Big(\frac{1}{2}\BB_1(f)-\lambda\Big)[\omega^{-1}\beta\cdot\nu]+\frac{1}{2}\BB_2(f)[\omega^{-1}\beta\cdot\tau]\Big\rVert_2^2\\
		&\qquad\quad-C\big\lVert\big(\lambda-\DD(f)^*\big)[\beta]\big\rVert_2\lVert\beta\rVert_2.
	\end{aligned}
	\end{equation*}
	The desired estimate~\eqref{eq:Df_1} now follows directly from the inequalities previously derived for both sides of~\eqref{eq:T_Gauss}.
	\end{proof}

	Recalling~\eqref{eq:B_by_Bnmpq},~\eqref{eq:Bnmpq-ad}, and~\eqref{eq:BB_i}, we compute that the $\rmL^2$-adjoint $\Aa(f):=\BB_1(f)^* \in\mcL\big(\rmL^2(\Ss)\big)$ satisfies the formula
	\begin{equation}\label{eq:defAA}
		\Aa(f) = f' B_1(f) - B_2(f).
	\end{equation}
	Related to $\Aa(f)$, we introduce
	\begin{equation}\label{eq:BB(f)}
		\BB(f):=B_1(f)+f'B_2(f)\in \mcL\big(\rmL^2(\Ss)\big)
	\end{equation} 
	and establish, as an additional ingredient in the proof of Theorem~\ref{Thm:D(f)_Res_2}, the following result.

	\begin{Lemma}\label{Lem:A(f)_Res}
	Given $K>0$, there exists a constant $C=C(K)>0$ with the property  that for all~$\lambda\in\RR$ with~$|\lambda|\geq 1$ and $f\in\rmC^1(\Ss)$ with $\lVert f'\rVert_\infty\leq K$ we have
	\begin{equation}\label{eq:A(f)_Res}
		\lVert \varphi\rVert_2\leq C\lVert(\lambda-\Aa(f))[\varphi]\rVert_2,\qquad \varphi\in\rmL^2(\Ss).
	\end{equation}
	Moreover, $\lambda-\Aa(f)\in{\mcL\big(\rmL^2(\Ss)\big)}$ is an isomorphism for all $\lambda\in\RR$ with $|\lambda|\geq 1$ and ~$f\in\rmC^1(\Ss)$.
	\end{Lemma}

	\begin{proof}
	Thanks to~\eqref{eq:B_by_Bnmpq}, \eqref{eq:defAA}, and Lemma~\ref{Lem:Bnmpq_L2_L2}, it is sufficient to establish~\eqref{eq:A(f)_Res} for~$f\in \rmC^\infty(\Ss)$ satisfying~${\lVert f'\rVert_\infty \leq K}$ and $\varphi \in\rmC^\infty(\Ss)$. We begin by introducing the vector field $V=(V_1,V_2)^\top$ with
	\begin{equation*}
		V := \big(-Z_2(f)[\varphi],\, Z_1(f)[\varphi]\big)^\top : (\Ss \times \RR) \setminus \Gamma \to \RR^2,
	\end{equation*}
	and recall~\eqref{eq:reg_Zi}–\eqref{eq:Z_pm} and~\eqref{eq:Zi_lim}. Then, $V^\pm := V|_{\Omega^\pm}$ belong   to $\rmC^\infty(\Omega^\pm) \cap \rmC(\overline{\Omega^\pm})$ and
	\begin{equation*}
		\vdiv V^\pm = \curl V^\pm = 0 \qquad \text{in } \Omega^\pm
	\end{equation*}
	by~\eqref{eq:curl}. Hence, we have
	\begin{equation}\label{eq:divV_lim}
		\vdiv
		\begin{pmatrix}
			2V_1^\pm V_2^\pm \\
			(V_2^\pm)^2 - (V_1^\pm)^2
		\end{pmatrix}
		= 0 \quad \text{in $\Omega^\pm$}\qquad\text{and}\qquad
		\begin{pmatrix}
			2V_1 V_2 \\
			(V_2)^2 - (V_1)^2
		\end{pmatrix}
		\underset{|x_2|\to\infty}\to
		\begin{pmatrix}
			0 \\
			-\langle\varphi\rangle^2
		\end{pmatrix},
	\end{equation}
	while
	\begin{equation}\label{eq:V_BV}
		\{V\}^\pm \circ \Xi = \big(-B_2(f)[\varphi],\, B_1(f)[\varphi]\big)^\top \mp \frac{\tau}{\omega} \varphi \qquad \text{on $\Ss$}.
	\end{equation}
	Integrating the first identity in \eqref{eq:divV_lim} over $\Omega^\pm_n$ with	$n>\lVert f\rVert_\infty$ (recalling~\eqref{eq:Omega_n}), it follows from  Stokes' theorem, \eqref{eq:divV_lim}, and~\eqref{eq:V_BV}, after letting $n\to\infty$, that
	\begin{equation*}
		\int_{-\pi}^\pi\omega\nu\cdot
		\begin{pmatrix}
			2\{V_1\}^\pm \{V_2\}^\pm \\
			\{(V_2^\pm)\}^2 - (\{V_1\}^\pm)^2
		\end{pmatrix}
		\circ\Xi\,\rmd s=-2\pi\langle\varphi\rangle^2.
	\end{equation*}
	Equivalently, noticing  from \eqref{eq:defAA}--\eqref{eq:BB(f)} and \eqref{eq:V_BV} that 
	\begin{equation*}
		\{V\}^\pm \circ \Xi=\frac{1}{\omega}\big((\mp 1+\Aa(f))[\varphi]\tau +\BB(f)[\varphi]\nu\big),
	\end{equation*}
	we have 
	\begin{equation}\label{eq:BFF2}
		\int_{-\pi}^{\pi}\frac{1}{\omega^2}\big[|\BB(f)[\varphi]|^2+2f'\BB(f)[\varphi](\mp1+\Aa(f))[\varphi]-|(\mp 1+\Aa(f))[\varphi]|^2\big]\,\rmd s=-2\pi\langle\varphi\rangle^2.
	\end{equation}		
	Using Young's inequality, we find a constant $C(K)>0$ such that
	\begin{equation*}
		\lVert (\mp1+\Aa(f))[\varphi]\rVert_2\leq C\big(\lVert \BB(f)[\varphi]\rVert_2+|\langle\varphi\rangle|\big),
	\end{equation*}		 
	hence
	\begin{equation}\label{eq:B(f)-1}
		\lVert\varphi\rVert_2 =\frac{1}{2}\lVert (1+\Aa(f))[\varphi]-(-1+\Aa(f))[\varphi]\rVert_2\leq C\big(\lVert \BB(f)[\varphi]\rVert_2+|\langle\varphi\rangle|\big).
	\end{equation}
	Noticing  for  $|\lambda|\geq 1$ that $(\mp 1+\Aa(f))[\varphi] =(\lambda\mp1)\varphi-(\lambda-\Aa(f))[\varphi]$, we infer from \eqref{eq:BFF2},  after plugging in this expression and adding the equation with sign $-$ multiplied by  $\lambda+ 1$   to the equation with sign $+$ multiplied by $1-\lambda$, that
	\begin{equation*}
		\int_{-\pi}^{\pi}\frac{1}{\omega^2}\big[(\lambda^2-1)|\varphi|^2+|\BB(f)[\varphi]|^2-2f'\BB(f)[\varphi](\lambda-\Aa(f))[\varphi]-|(\lambda-\Aa(f))[\varphi]|^2\big]\,\rmd s=-2\pi\langle\varphi\rangle^2.
	\end{equation*}
	With Young's inequality, we again find a constant $C(K)>0$ such that
	\begin{equation*}
		(\lambda^2-1)\lVert\varphi\rVert_2+\lVert\BB(f)[\varphi]\rVert_2+|\langle\varphi\rangle|\leq C\lVert(\lambda-\Aa(f))[\varphi]\rVert_2.
	\end{equation*}
	The desired claim~\eqref{eq:A(f)_Res} is now a straightforward consequence of this estimate and \eqref{eq:B(f)-1}.
		
	Finally, since $\Aa(f)\in\mcL\big(\rmL^2(\Ss)\big)$, its spectrum is compact and the method of continuity; see, e.g.,~\cite[Proposition~I.1.1.1]{Amann.1995} yields the isomorphism property.
	\end{proof}

	We are now in a position to prove Theorem~\ref{Thm:D(f)_Res_2}.
	
	\begin{proof}[Proof of Theorem~\ref{Thm:D(f)_Res_2}]
	Assuming that the claim is false, we find sequences $(\lambda_k)\subset\RR$, $(f_k)\subset\rmC^1(\Ss)$, and $(\beta_k)\subset\rmL^2(\Ss)^2$ satisfying	for all $k\in\NN$
	\begin{equation*}
		|\lambda_k|\geq \frac{1}{2}+\delta,\qquad \lVert f'_k\rVert_\infty\leq\frac{1}{\delta},\qquad \lVert\beta_k\rVert_2=1,
	\end{equation*}
	and such that
	\begin{equation}\label{eq:lk-Dfk}
		(\lambda_k-\DD(f_k)^*)[\beta_k]\to 0\qquad\text{in}~\rmL^2(\Ss)^2.
	\end{equation}
	Recalling~\eqref{eq:nutauomega}, we define $\nu_k := \nu(f_k)$, $\tau_k := \tau(f_k)$, and $\omega_k := \omega(f_k)$ and note from \eqref{eq:B_by_Bnmpq}, \eqref{eq:D(f)-ad}, and Lemma~\ref{Lem:Bnmpq_L2_L2} that there exists a constant $M > 0$ such that $\lVert\DD(f_k)^*\rVert_{\mcL(\rmL^2(\Ss))} \leq M$ for all  $k \in \NN$. This bound together with~\eqref{eq:lk-Dfk} implies  the boundedness of the sequence $(\lambda_k)$ in $\RR$. Observing that the constant $m=m(\lambda)$ in Lemma~\ref{Lem:Df_1} satisfies $m(\lambda_k)\geq \delta(2+\delta)>0 $ for all~$k\in\NN$, we deduce from~\eqref{eq:Df_1} that
	\begin{equation}\label{eq:w_b_tau_0}
		\omega_k^{-1}\beta_k\cdot\tau_k\to 0 \qquad\text{and}\qquad \big(\lambda_k-\tfrac{1}{2}\BB_1(f_k)\big)[\omega_k^{-1}\beta_k\cdot\nu_k]-\tfrac{1}{2}\BB_2(f_k)[\omega_k^{-1}\beta_k\cdot\tau_k]\to 0 
	\end{equation}
	in $\rmL^2(\Ss)$. In view of the boundedness of the sequence $(\BB_2(f_k))_k$  in $\mcL\big(\rmL^2(\Ss)\big)$, which is a consequence of \eqref{eq:B_by_Bnmpq}, \eqref{eq:BB_i}, and Lemma~\ref{Lem:Bnmpq_L2_L2}, we infer from \eqref{eq:w_b_tau_0} that
	\begin{equation*}
		\big(\lambda_k-\tfrac{1}{2}\BB_1(f_k)\big)[\omega_k^{-1}\beta_k\cdot\nu_k]\to 0\qquad \text{in}~\rmL^2(\Ss).
	\end{equation*}
	Since $\Aa(f_k)=\BB_1(f_k)^*$, $k\in\NN$, Lemma~\ref{Lem:A(f)_Res} implies in turn that 
	\begin{equation}\label{eq:w_b_nu_0}
		\omega_k^{-1}\beta_k\cdot\nu_k \to 0\qquad \text{in}~\rmL^2(\Ss).
	\end{equation}
	Combining \eqref{eq:w_b_tau_0} and \eqref{eq:w_b_nu_0}, we have $\beta_k\to0$ in $\rmL^2(\Ss)^2$, which contradicts the fact that~$\lVert\beta_k\rVert_2=1$ for all $k\in\NN$. This proves \eqref{eq:D(f)_Res_2}.
		
	The isomorphism property of  $\lambda-\DD(f)^*\in\mcL\big(\rmL^2(\Ss)^2\big)$ with $|\lambda|>1/2$ and $f\in\rmC^1(\Ss)$ is now a straightforward consequence of \eqref{eq:D(f)_Res_2} and of the compactness of the spectrum of $\DD(f)^*\in\mcL\big(\rmL^2(\Ss)^2\big)$, the invertibility of its adjoint $\lambda-\DD(f)$  following at once.	
	\end{proof}
	
 \section{The resolvent of the  double-layer potential operator in higher order Sobolev spaces}\label{Sec:4}

 	In this section, we assume that $f \in \rmH^{r}(\Ss)$ with $r \in (3/2, 2)$, and later that $f \in \rmH^{3}(\Ss)$, and we show that the resolvent set of $\DD(f)$, viewed as an operator in either $\mcL\big(\rmH^{r-1}(\Ss)^2\big)$ or $\mcL\big(\rmH^{2}(\Ss)^2\big)$, contains all real numbers~$\lambda$ such that~${|\lambda| > 1/2}$; see Theorem~\ref{Thm:D(f)_Res_Hr-1} and Theorem~\ref{Thm:D(f)_Res_H2}. These results are key arguments in the proof of Theorem~\ref{Thm:Ftp} where we prove in the context of~\eqref{eq:STOKES} that the interface between the fluids determines at each fixed positive time instant the velocity and pressure fields in both fluid layers.
 	
\subsection*{Invertibility in \texorpdfstring{$\mcL\big(\rmH^{r-1}(\Ss)\big)^2$}{L(Hr-1)}}
 	To start, we fix~$r\in(3/2,2)$ and  recall from~\cite[Corollary~A.8]{Bohme.2024} that	for $n,\,m,\,p,\,q\in\NN_0$ with $p\leq n+q+1$  we have 
 	\begin{equation}\label{eq:reg_B}
	\begin{aligned}
		&\big[f\mapsto B_{n,m}^{0,q}(f)\big]\in\rmC^\infty\big(\rmH^r(\Ss), \mcL\big(\rmH^{r-1}(\Ss)\big)\big),\\
		&\big[f\mapsto B_0(f)\big],\,\big[f\mapsto B_{n,m}^{p,q}(f)\big]\in\rmC^\infty\big(\rmH^r(\Ss), \mcL\big(\rmH^{r-1}(\Ss), \rmH^r(\Ss)\big)\big),\qquad   1\leq p\leq n+q+1.
	\end{aligned}
	\end{equation}
	The definition~\eqref{eq:D(f)} of the  double-layer potential operator $\DD$ together with \eqref{eq:B_by_Bnmpq} and \eqref{eq:reg_B} implies now
	 \begin{equation}\label{eq:D(f)_smooth}
	 	\big[f\mapsto \DD(f)\big]\in\rmC^\infty \big(\rmH^r(\Ss),\mcL\big(\rmH^{r-1}(\Ss)^2\big)\big),
	\end{equation}
	so that in particular $\DD(f)\in\mcL\big(\rmH^{r-1}(\Ss)^2\big)$ for all $f\in\rmH^r(\Ss)$. The arguments in the proof of the main result of this section; see Theorem~\ref{Thm:D(f)_Res_Hr-1}, rely on mapping properties for the operators $B_{n,m}^{p,q}$ established in Lemma~\ref{Lem:Bnmpq_L2_Hr-1} and on the interpolation relation
 	\begin{equation}\label{eq:interpolation}
		\big[\rmH^{s_0}(\Ss),\rmH^{s_1}(\Ss)\big]_{\theta}=\rmH^{(1-\theta)s_0+\theta s_1}(\Ss),\qquad \theta\in (0,1),\quad -\infty<s_0 \leq s_1<\infty,  
	\end{equation}
	where~$[\cdot,\cdot]_\theta$ is the complex interpolation functor of exponent $\theta$.

 	\begin{Lemma}\label{Lem:Bnmpq_L2_Hr-1}
 	Let $n,\,m,\,p,\, q\in\NN_0$ and~$(\bfa, \bfb)\in\rmH^r(\Ss)^{m+n}$ be given. 
 	\begin{enumerate}
 	\item[{\rm (i)}] Let $p\leq n+q+1$. Then, there exists a constant $C>0$ that depends only on $n,\,m,\,p,\,q,\,r$,  and $\lVert(\bfa,\bfb)\rVert_{\rmH^r}$ such that for all $\bfc\in\rmH^r(\Ss)^{q}$ and~$\varphi\in\rmH^{r-1}(\Ss)$ we have 
		\begin{equation}\label{eq:Bnmpq_Hr-1}
			\big\lVert B_{n,m}^{p,q}(\bfa\vert \bfb)[\bfc,\varphi]\big\rVert_{\rmH^{r-1}}\leq C\lVert\varphi\rVert_{\rmH^{r-1}} \prod_{i=1}^{q}\lVert c_i\rVert_{\rmH^r}.
		\end{equation}
	\item[{\rm (ii)}] Let $q\geq 1$ and $p\leq n+q+1$. Then, there exists a constant~${C>0}$ that depends only  on~${n,\,m,\,p,\,q}$, and~${\lVert(\bfa,\bfb)\rVert_{\rmH^r}}$ such that for all~${\varphi\in\rmH^{r-1}(\Ss)}$ and $\bfc\in\rmH^1(\Ss)\times \rmH^r(\Ss)^{q-1}$ we have
		\begin{equation}\label{eq:Bnmpq_L2_Hr-1_c}
			\big\lVert B_{n,m}^{p,q}(\bfa\vert\bfb)[\bfc,\varphi]\big\rVert_2\leq C\lVert c_1'\rVert_2\lVert\varphi\rVert_{\rmH^{r-1}}\prod_{i=2}^q \lVert c_i\rVert_{\rmH^r}.
		\end{equation}
		Moreover,~${B_{n,m}^{p,q}\in\rmC^{1-}\big(\rmH^r(\Ss)^{m+n},\mcL^{q+1}\big(\rmH^1(\Ss),\rmH^r(\Ss)^{q-1},\rmH^{r-1}(\Ss);\rmL^2(\Ss)\big)\big)}$.
	\item[{\rm (iii)}] Let $d,\,\tilde{d}\in \rmH^{1}(\Ss)$ and $p\leq n+q+2$. Then, there exists a constant~${C>0}$ that depends only on~$n,\,m,\,p,\,q$,~${\lVert(\bfa,\bfb)\rVert_{\rmH^r}}$, and~$\lVert (d,\tilde{d})\rVert_{\rmH^1}$ such that for all~${\varphi\in\rmH^{r-1}(\Ss)}$ and $\bfc\in\rmH^r(\Ss)$ we have
		\begin{equation}\label{eq:Bnmpq_L2_Hr-1_b}
			\big\lVert B_{n+1,m}^{p,q}(\bfa\vert(\bfb,d))[\bfc,\varphi]-B_{n+1,m}^{p,q}(\bfa\vert(\bfb,\tilde{d}))[\bfc,\varphi]\big\rVert_2\leq C\lVert (d-\tilde{d})'\rVert_2\lVert\varphi\rVert_{\rmH^{r-1}}\prod_{i=1}^q \lVert c_i\rVert_{\rmH^{r}}.
		\end{equation}
	\item[{\rm (iv)}] There exists a constant $C>0$ that depends only on $n,\,m$, and $\lVert\bfa\rVert_{\rmH^r}$ with the property that for all~${\varphi\in\rmH^{r-1}(\Ss)}$   we have
		\begin{equation}\label{eq:Cnmpq_Hr-1}
			\big\lVert C_{n,m}(\bfa)[\bfb,\varphi]\big\rVert_{\rmH^{r-1}}\leq C\lVert\varphi\rVert_{\rmH^{r-1}} \prod_{i=1}^{q}\lVert b_i\rVert_{\rmH^r}.
		\end{equation}
 	\end{enumerate}
 	\end{Lemma}

	\begin{proof}
	The estimate \eqref{eq:Bnmpq_Hr-1} is proven in \cite[Lemma~A.4 and~A.6]{Bohme.2024}, while the  assertions (ii)--(iii) are shown in Appendix~\ref{Sec:A}. Finally, the estimate~\eqref{eq:Cnmpq_Hr-1} is established in \cite[Lemma A.1~(iii)]{Bohme.2024} and \cite{DBThesis}.
	\end{proof}

 	We now establish with Theorem~\ref{Thm:D(f)_Res_Hr-1} a counterpart of Theorem~\ref{Thm:D(f)_Res_2} in~$\rmH^{r-1}(\Ss)^2$.

 	\begin{Theorem}\label{Thm:D(f)_Res_Hr-1}
 	Let $\delta\in (0,1)$ be given. Then, there exists a positive constant~${C=C(\delta,r)}$, such that for all $\lambda\in\RR$ with $|\lambda|\geq 1/2+\delta$, $f\in\rmH^{r}(\Ss)$ satisfying $\lVert f\rVert_{\rmH^r}\leq 1/\delta$, and $\beta\in\rmH^{r-1}(\Ss)^2$ we have
 	\begin{equation}\label{eq:D(f)_Res_Hr-1}
 		\lVert(\lambda-\DD(f))[\beta]\rVert_{\rmH^{r-1}}\geq C\lVert \beta\rVert_{\rmH^{r-1}}.
	\end{equation}
	Moreover, $\lambda-\DD(f)\in\mcL\big(\rmH^{r-1}(\Ss)^2\big)$ is an isomorphism for all $\lambda\in\RR$ with $|\lambda|>1/2$ and~${f\in\rmH^r(\Ss)}$.
 	\end{Theorem}

	\begin{proof}
	Let $\delta\in (0,1)$ be fixed. Recalling Theorem~\ref{Thm:D(f)_Res_2} and  using the embedding ${\rmH^r(\Ss)\hookrightarrow \rmC^1(\Ss)}$, 
	we find a constant~${C=C(\delta)>0}$ such that for all~$\zeta\in\Ss$, $f\in\rmH^r(\Ss)$ with $\lVert f\rVert_{\rmH^r}\leq 1/\delta$, and~$\lambda\in\RR$ satisfying $|\lambda|\geq 1/2+\delta$, we have~${\lVert(\lambda-\DD(\tau_\zeta f))^{-1}\rVert_{\mcL(\rmL^2(\Ss)^2)}\leq C}$. Here, $\tau_\zeta f := f(\zeta+\cdot)$ is the left-shift operator. Noticing that the standard $\rmH^{r-1}$-norm, defined by means of the Fourier transform, is equivalent to the norm $\big[\beta\mapsto \lVert\beta\rVert_2 +[\beta]_{\rmW^{r-1,2}}\big]$, where
	\begin{equation*}
		[\beta]_{\rmW^{r-1,2}}^2:=\int_{-\pi}^\pi\frac{\lVert\tau_\zeta \beta-\beta\rVert_2^2}{|\zeta|^{2r-1}}\,\rmd \zeta,
	\end{equation*}
	we now estimate
	\begin{equation}\label{eq:beta_sem}
	\begin{aligned}\relax
		[\beta]_{\rmW^{r-1,2}}^2 &\leq C\int_{-\pi}^\pi \frac{\big\lVert (\lambda-\DD(\tau_\zeta f))[\tau_\zeta \beta-\beta] \big\rVert_2^2}{|\zeta|^{2r-1}}\,\rmd \zeta\\
		&\leq C\int_{-\pi}^\pi \frac{\big\lVert \tau_\zeta\big((\lambda-\DD(f))[\beta]\big)-(\lambda-\DD(f))[\beta] \big\rVert_2^2+\big\lVert\big(\DD(\tau_\zeta f)
		-\DD(f)\big)[\beta]\big\rVert_2^2}{|\zeta|^{2r-1}}\,\rmd \zeta\\
		&\leq C\big[(\lambda-\DD(f))[\beta]\big]_{\rmW^{r-1,2}}^2+C\int_{-\pi}^\pi \frac{\big\lVert\big(\DD(\tau_\zeta f)-\DD(f)\big)[\beta]\big\rVert_2^2}{|\zeta|^{2r-1}}\,\rmd \zeta.
	\end{aligned}
	\end{equation}
	Concerning the second term in the last line of \eqref{eq:beta_sem}, we fix an arbitrary $r'\in(3/2,r)$ and show  below that there exists a constant $C=C(\delta,r')>0$ such that for all~$\zeta\in\Ss$, $f\in\rmH^r(\Ss)$ with $\lVert f\rVert_{\rmH^r}\leq 1/\delta$, and $\beta \in\rmH^{r-1}(\Ss)^2$ we have
	\begin{equation}\label{eq:D(yf)-D(f)}
		\big\lVert\big(\DD(\tau_\zeta f)-\DD(f)\big)[\beta]\big\rVert_2\leq C \lVert \tau_\zeta f'-f'\rVert_2 \lVert\beta\rVert_{\rmH^{r'-1}}.
	\end{equation}
	 Recalling \eqref{eq:B_by_Bnmpq} and \eqref{eq:D(f)}, it suffices to consider terms of the form
	\begin{equation}\label{eq:two_op}
		\big(B_{n,m}^{p,q}(\tau_\zeta f)-B_{n,m}^{p,q}(f)\big)[\beta_i]\qquad \text{and}\qquad B_{n,m}^{p,q}(\tau_\zeta f)[(\tau_\zeta f)'\beta_i]-B_{n,m}^{p,q}(f)[f'\beta_i],
	\end{equation}	
	with $n,\,m,\,p,\,q\in\NN_0$ with $p\leq n+q+1$ and $i=1,\, 2$. Since
	\begin{equation*}
		B_{n,m}^{p,q}(\tau_\zeta f)[(\tau_\zeta f)'\beta_i]-B_{n,m}^{p,q}(f)[f'\beta_i]=\big(B_{n,m}^{p,q}(\tau_\zeta f)-B_{n,m}^{p,q}(f)\big)[(\tau_\zeta f)'\beta_i]+B_{n,m}^{p,q}(f)[(\tau_\zeta f'-f')\beta_i],
	\end{equation*}
	the last term on the right side being estimated in view of~\eqref{eq:Bnmpq_L2_L2} by
	\begin{equation*}
		\lVert B_{n,m}^{p,q}(f)[(\tau_\zeta f'-f')\beta_i]\rVert_2 \leq C \lVert (\tau_\zeta f'-f')\beta_i\rVert_2 \leq C \lVert (\tau_\zeta f'-f')\rVert_2 \lVert \beta \rVert_{\rmH^{r'-1}},
	\end{equation*}
	it remains to estimate  the first term  in \eqref{eq:two_op} according to~\eqref{eq:D(yf)-D(f)}. 
	To this end, we infer from the formula \eqref{eq:diffB} that the operator $B_{n,m}^{p,q}(\tau_\zeta f)-B_{n,m}^{p,q}(f)$ can be written as a linear combination of terms of the form
	\begin{equation*}
		T_1:=B_{\bar n, \bar m}^{\bar p,\bar q}(\tau_\zeta f,\dots,\tau_\zeta f \vert \tau_\zeta f,\dots, \tau_\zeta f)[\underbrace{f,\dots,f}_{i-1~\text{times}},\tau_\zeta f -f,\tau_\zeta f,\dots,\tau_\zeta f,\cdot]
	\end{equation*}
	and 
	\begin{equation*}
	\begin{aligned}
		T_2&:=\big(B_{\bar n, \bar m}^{\bar p,\bar q}(\underbrace{f,\dots,f}_{j~\text{times}},\tau_\zeta f,\dots, \tau_\zeta f \vert \underbrace{f,\dots,f}_{\ell-1~\text{times}}\tau_\zeta f,\ldots\tau_\zeta f)\\
		&\quad\hspace{1em}-B_{\bar n, \bar m}^{\bar p,\bar q}(\underbrace{f,\dots,f}_{j~\text{times}},\tau_\zeta f,\dots, \tau_\zeta f \vert \underbrace{f,\dots,f}_{\ell~\text{times}}\tau_\zeta f,\ldots\tau_\zeta f)\big)[f,\ldots,f,\cdot]
	\end{aligned}
	\end{equation*}
	where  $\bar n,\,\bar m,\,\bar p,\,\bar q\in\NN_0$ with $\bar p\leq \bar n+\bar q+1$,  $1\leq i\leq \bar q$, $0\leq j\leq \bar m-1$, and $1\leq \ell\leq \bar n$. 
		
	Applying Lemma~\ref{Lem:Bnmpq_L2_Hr-1}~(ii) to estimate $T_1$ (with $r$ replaced by $r'$), and Lemma~\ref{Lem:Bnmpq_L2_Hr-1}~(iii) for $T_2$ (again with $r$ replaced by $r'$), we obtain
	\begin{equation*}
		\big\lVert B_{n,m}^{p,q}(\tau_\zeta f) - B_{n,m}^{p,q}(f) \big\rVert_{\mcL(\rmH^{r'-1}(\Ss), \rmL^2(\Ss))} \leq C \lVert \tau_\zeta f' - f' \rVert_2,
	\end{equation*}
	from which \eqref{eq:D(yf)-D(f)} readily follows.
		
	Plugging~\eqref{eq:D(yf)-D(f)} into \eqref{eq:beta_sem} immediately yields
	\begin{equation*}
		[\beta]_{\rmW^{r-1,2}}^2\leq C\big(\lVert\beta\rVert_{\rmH^{r'-1}}^2+\big[(\lambda-\DD(f))[\beta]\big]_{\rmW^{r-1,2}}^2\big).
	\end{equation*}
	Using the interpolation property~\eqref{eq:interpolation}, Young's inequality, and Theorem~\ref{Thm:D(f)_Res_2}, we get
	\begin{equation}
	\begin{aligned}
		\lVert\beta\rVert_{\rmH^{r-1}}^2&\leq C\big(\lVert\beta\rVert_2^2+\lVert\beta\rVert_{\rmH^{r'-1}}^2+\big[(\lambda-\DD(f))[\beta]\big]_{\rmW^{r-1,2}}^2\big)\\
		&\leq \frac{1}{2}\lVert\beta\rVert_{\rmH^{r-1}}^2+C\lVert (\lambda-\DD(f))[\beta] \rVert_{\rmH^{r-1}}^2,
	\end{aligned}
	\end{equation}
	which proves \eqref{eq:D(f)_Res_Hr-1}. The remaining isomorphism property follows from~\eqref{eq:D(f)_Res_Hr-1} by the same continuity argument used in the proof of Theorem~\ref{Thm:D(f)_Res_2}.
	\end{proof}

	\subsection*{Invertibility in \texorpdfstring{$\mcL\big(\rmH^{2}(\Ss)^2\big)$}{L(H2)}}	 	We first establish in Lemma~\ref{Lem:Bnmpq_H1_H1} and Corollary \ref{Cor:Bnmpq_H2} mapping properties for the operators $B_{n,m}^{p,q}$ in higher order Sobolev spaces, which enable us to conclude in particular that~$\DD(f)\in \mcL\big(\rmH^{2}(\Ss)^2\big)$ provided that $f\in \rmH^{3}(\Ss)$. For convenience, we introduce the following shorthand notation: given $\bfz=(z_1,\ldots,z_n) \in X^n$, where~$X$ is a Banach space, we set
	\begin{equation*} 
		\bfz_i :=(z_1,\dots,z_{i-1},z_{i+1},\dots,z_n)\in X^{n-1},\qquad 1\leq i\leq n,\quad n\in\NN.
	\end{equation*}

	\begin{Lemma}\label{Lem:Bnmpq_H1_H1}
	Let $n,\,m,\,p,\,q\in\NN_0$ with $p\leq n+q+1$, $(\bfa,\bfb,\bfc)\in\rmH^2(\Ss)^{m+n+q}$, and $\varphi\in\rmH^1(\Ss)$ be given. Then, $B_{n,m}^{p,q}(\bfa\vert\bfb)[\bfc,\cdot]\in \mcL\big(\rmH^1(\Ss)\big)$ with 
	\begin{equation}\label{eq:B'}
 	\begin{aligned}
 		&\big(B_{n,m}^{p,q}(\bfa\vert\bfb)[\bfc,\varphi]\big)'\\
 		&=B_{n,m}^{p,q}(\bfa\vert\bfb)[\bfc,\varphi']+\sum_{j=1}^{q}B_{n,m}^{p,q}(\bfa\vert\bfb)[c_1,\dots,c_{j-1},c_j',c_{j+1},\dots,c_q,\varphi]\\
 		&\quad+\sum_{j=1}^n\big(B_{n-1,m}^{p,q+1}(\bfa\vert\bfb_j)[\bfc,b_j',\varphi]-B_{n+1,m}^{p+2,q+1}(\bfa\vert\bfb,b_j)[\bfc,b_j',\varphi]\big)\\
 		&\quad-2\sum_{j=1}^m \big(B_{n+1,m+1}^{p,q+1}(\bfa,a_j\vert\bfb,a_j)[\bfc,a_j',\varphi]-B_{n+3,m+1}^{p+2,q+1}(\bfa,a_j\vert\bfb,a_j,a_j,a_j)[\bfc,a_j',\varphi]\big).
 	\end{aligned}
 	\end{equation}
	Moreover,  there exists a constant~${C>0}$ that depends only on $n,\,m,\,p,\,q$, and~${\lVert(\bfa,\bfb)\rVert_{\rmH^2}}$ such that  
	\begin{equation}\label{eq:Bnmpq_H1_H1}
		\big\lVert B_{n,m}^{p,q}(\bfa\vert\bfb)[\bfc,\varphi]\big\rVert_{\rmH^1}\leq C\lVert\varphi\rVert_{\rmH^1}\prod_{i=1}^q \lVert c_i\rVert_{\rmH^2}.
	\end{equation}	
	In addition,~${B_{n,m}^{p,q}\in\rmC^{1-}\big(\rmH^2(\Ss)^{m+n},\mcL^q_{\textup{sym}}\big(\rmH^{2}(\Ss),\mcL\big(\rmH^{1}(\Ss)\big)\big)\big)}$. 
	\end{Lemma}

 	\begin{proof}[Proof of Lemma~\ref{Lem:Bnmpq_H1_H1}]
	Assume first that $(\bfa,\bfb,\bfc)\in\rmC^\infty(\Ss)^{m+n+q}$ and $\varphi\in\rmC^\infty(\Ss)$. 
	Then, by rewriting the principal value integral as a Riemann integral over the interval $[0, \pi]$ if $p=0$ (for $p\geq 1$ this step is not needed as the integral defining 
	$B_{n,m}^{p,q}(\bfa\vert\bfb)[\bfc,\varphi]$ is not singular), the theorem on the differentiation of parameter-dependent integrals ensures that $ B_{n,m}^{p,q}(\bfa\vert\bfb)[\bfc,\varphi]$ is continuously differentiable, and its derivative is given by~\eqref{eq:B'}. A standard density argument together with the local Lipschitz
	 continuity properties established in Lemma~\ref{Lem:Bnmpq_L2_L2} and Lemma~\ref{Lem:Bnmpq_L2_Hr-1}~(ii) show 
	 that~${B_{n,m}^{p,q}(\bfa\vert\bfb)[\bfc,\varphi]}$ is an element of~${\rmH^1(\Ss)}$ for all~$\varphi\in\rmH^1(\Ss)$ and~${(\bfa,\bfb,\bfc)\in\rmH^2(\Ss)^{m+n+q}}$.
 		
 	The estimate~\eqref{eq:Bnmpq_H1_H1} is obtained by applying~\eqref{eq:Bnmpq_L2_L2} to ${B_{n,m}^{p,q}(\bfa\vert\bfb)[\bfc,\varphi]}$ and to the first term on the right side of~\eqref{eq:B'}, and by applying~\eqref{eq:Bnmpq_L2_Hr-1_c} to the remaining terms on the right side of~\eqref{eq:B'}. Furthermore, the local Lipschitz continuity follows from the corresponding local Lipschitz continuity properties stated in Lemma~\ref{Lem:Bnmpq_L2_L2} and Lemma~\ref{Lem:Bnmpq_L2_Hr-1}~(ii).
 	\end{proof}

	As a straightforward consequence of Lemma~\ref{Lem:Bnmpq_H1_H1}, in particular of the formula~\eqref{eq:B'}, we conclude via induction the following result.
	\begin{Corollary}\label{Cor:Bnmpq_H2}
 	Given~${n,\,m,\,p,\,q\in\NN_0}$ with~${p\leq n+q+1}$ and~${k\in\NN}$, we have
 	\begin{equation*}
 		B_{n,m}^{p,q}\in\rmC^{1-}\big(\rmH^{k+1}(\Ss)^{m+n},\mcL^q_{\textup{sym}}\big(\rmH^{k+1}(\Ss),\mcL\big(\rmH^{k}(\Ss)\big)\big)\big).
 	\end{equation*}
	\end{Corollary}

	In view of Corollary~\ref{Cor:Bnmpq_H2}, \eqref{eq:B_by_Bnmpq}, and \eqref{eq:D(f)} we indeed have that~$\DD(f)\in \mcL\big(\rmH^2(\Ss)\big)^2$ for~$f\in \rmH^3(\Ss)$.
 	
	\begin{Theorem}\label{Thm:D(f)_Res_H2}
	Given $\lambda\in\RR$ with $|\lambda|>1/2$ and~${f\in\rmH^3(\Ss)}$, the operator $\lambda-\DD(f)\in\mcL\big(\rmH^2(\Ss)^2\big)$ is an isomorphism.
 	\end{Theorem}

	\begin{proof}
	Arguing  as in Theorem~\ref{Thm:D(f)_Res_2} and Theorem~\ref{Thm:D(f)_Res_Hr-1}, it suffices to show that given $\delta\in(0,1)$, there exists a constant $C=C(\delta)>0$ such that $|\lambda|\geq 1/2+\delta$ and $\lVert f\rVert_{\rmH^3}\leq 1/\delta$  we have
	\begin{equation}\label{eq:Resh2}
		C\lVert (\lambda-\DD(f))[\beta]\rVert_{\rmH^2}\geq \lVert \beta\rVert_{\rmH^2}, \qquad \beta\in\rmH^2(\Ss)^2.
	\end{equation}
	To this end we express, using~\eqref{eq:B_by_Bnmpq}, \eqref{eq:D(f)},  and~\eqref{eq:B'}, the difference $(\DD(f)[\beta])''-\DD(f)[\beta'']=: T_{\text{lot}}[\beta]$ as a linear combination of terms of the following form
	\begin{equation*}
	\begin{aligned}
		&B_{n,m}^{p,q}(f,\dots,f\vert f,\dots,f)[f'',\dots,f,(f')^{k}\beta_i],
		&&B_{n,m}^{p,q}(f,\dots,f\vert f,\dots,f)[f',f,\dots,f,((f')^k\beta_i)']\\
		&B_{n,m}^{p,q}(f,\dots,f\vert f,\dots,f)[f',f',\dots,f,(f')^{k}\beta_i],
		&&B_{n,m}^{p,q}(f)[f{'''}\beta_i+2f{''}\beta_i'],
	\end{aligned}
	\end{equation*}
	where $k\in\{0,1\}$, $i\in\{1,2\}$, and $n,\,m,\,p,\,q\in\NN_0$ satisfy $p\leq n+q+1$. Using Lemma~\ref{Lem:Bnmpq_L2_L2} to estimate the last three terms above and   Lemma~\ref{Lem:Bnmpq_L2_Hr-1}~(ii) for the first term, we find a constant $C>0$ that depends only on $n,\,m,\,p,\,q$ and $\lVert f\rVert_{\rmH^3}$ such that
	\begin{equation*}
		\lVert T_{\text{lot}}[\beta]\rVert_2\leq C\lVert \beta\rVert_{\rmH^1},\qquad\beta\in\rmH^2(\Ss)^2.
	\end{equation*}
	This estimate, the interpolation property~\eqref{eq:interpolation} together with Young's inequality, and the observation that $\lVert(\lambda-\DD(f))^{-1}\rVert_{\mcL(\rmL^2(\Ss))}\leq C$ by Theorem~\ref{Thm:D(f)_Res_2} show that for  $|\lambda|\geq 1/2+\delta$ and $\lVert f\rVert_{\rmH^3}\leq 1/\delta$ we have
	\begin{equation*}
	\begin{aligned}
		\lVert \beta\rVert_{\rmH^2}&\leq C(\lVert \beta\rVert_2+\lVert \beta''\rVert_2)\leq C\big(\lVert \beta\rVert_2+\lVert (\lambda-\DD(f))[\beta'']\rVert_2\big)\\
		&\leq C\big(\lVert \beta\rVert_2 +\lVert T_{\text{lot}}[\beta]\rVert_2+\lVert ((\lambda-\DD(f))[\beta])''\rVert_2\big)\\
		&\leq C\big(\lVert \beta\rVert_{\rmH^1}+\lVert ((\lambda-\DD(f))[\beta])''\rVert_2\big)\\
		&\leq \frac{1}{2}\lVert \beta\rVert_{\rmH^2}+ C\big(\lVert \beta\rVert_2+\lVert ((\lambda-\DD(f))[\beta])''\rVert_2\big)\\
		&\leq \frac{1}{2}\lVert \beta\rVert_{\rmH^2} +C\lVert (\lambda-\DD(f))[\beta]\rVert_{\rmH^2},
	\end{aligned}
	\end{equation*}
	which proves \eqref{eq:Resh2}.
	\end{proof}

 	\section{Local well-posedness of the Stokes flow~\texorpdfstring{\eqref{eq:STOKES}}{(1.1)}}\label{Sec:5}
	In this section, we establish the local well-posedness result stated in Theorem~\ref{Thm:Main}. 
	To this end, we first assume that the function $f = f(t)$, which  parametrizes the interface between the fluids, belongs to~$\rmH^3(\Ss)$ 
	at a fixed time $t > 0$ (we do not make the time dependence explicit at this stage), and we prove, using results from Section~\ref{Sec:2}--Section~\ref{Sec:4}, 
	that the velocity field and pressure are uniquely determined by $f$ as solutions to the transmission boundary value problem \eqref{eq:Stokes}$_{1-5}$. 
	This requires a unique choice for the far-field  conditions in \eqref{eq:Stokes}$_5$, which depend on $f$ as well. 
	Moreover, we provide an explicit formula for the trace of the velocity field on the interface; see Theorem~\ref{Thm:Ftp}. 
	This formula, together with the kinematic boundary condition \eqref{eq:Stokes}$_{6}$ and the results from Section~\ref{Sec:4}, allows us to recast \eqref{eq:STOKES} as an evolution problem for~$f$ in any subcritical space~$\rmH^r(\Ss)$, $r \in (3/2, 2)$, which is then shown to be of parabolic type; see~\eqref{eq:ev_pb} and Proposition~\ref{Prop:Psi_Gen}. 
	We emphasize that, due to the low regularity assumed for $f$, the evolution problem \eqref{eq:ev_pb} is not only nonlocal, but must also be treated as fully nonlinear. This is evident for example when considering the term $\phi(f)$, defined in~\eqref{eq:defphi} below, which depends fully nonlinearly on $f$ and appears in the definition of the nonlinearities in the evolution problem~\eqref{eq:ev_pb}, as seen through \eqref{eq:defV}--\eqref{eq:Db=V} and \eqref{eq:v_on_Gamma}. We conclude the section with the proof of Theorem~\ref{Thm:Main}, which relies on the abstract theory for fully nonlinear evolution problems presented in \cite[Chapter 8]{Lunardi.1995}.
 
	\subsection*{The solution of the transmission boundary value problem \texorpdfstring{\eqref{eq:Stokes}$_{1-5}$}{(1.1)}}\label{SS:5.1}

	In this section, we adopt the notation introduced at the beginning of Section~\ref{Sec:2} and present first some additional notation needed to state our main result concerning the solvability of \eqref{eq:Stokes}$_{1-5}$; see Theorem~\ref{Thm:Ftp}.

	Given $f\in \rmH^3(\Ss)$, we set
 	\begin{equation}\label{eq:defphi}
		\phi:= \phi(f):=(\phi_1(f),\phi_2(f)):=\big((\omega(f))^{-1}- 1,f'(\omega(f))^{-1}\big)\in\rmH^2(\Ss)^2
	\end{equation}	
	and define $\mcV:=\mcV(f)$ by
	\begin{equation}\label{eq:defV}
 		\mcV(f):=\frac{1}{4}\big(-\sigma\mcV_1(f)-\Theta\mcV_3(f),\,-\sigma\mcV_2(f)+\Theta\mcV_4(f)+\Theta\ln (4)\langle f\rangle\big)^\top,
 	\end{equation}
 	where
 	\begin{equation}\label{eq:def_Psi_i}
 	\begin{aligned}
 		\mcV_1(f)&:=\big(B_1-2B_4\big)(f)[\phi_1(f)-f' \phi_2(f)]+\big(2B_2+B_3\big)(f)[f' \phi_1(f)]+B_3(f)[\phi_2(f)],\\
 		\mcV_2(f)&:=B_1(f)[\phi_2(f)-f' \phi_1(f)]+B_3(f)[\phi_1(f)-f' \phi_2(f)]+2B_4(f)[f' \phi_1(f)+\phi_2(f)],\\
 		\mcV_3(f)&:=\big(B_0(f)+B_6(f)\big)[ff']+B_5(f)[f],\\
 		\mcV_4(f)&:=\big(B_0(f)-B_6(f)\big)[f]+B_5(f)[ff'].
 	\end{aligned}
 	\end{equation}   
	Note that  $\mcV(f)\in\rmH^2(\Ss)^2$ by \eqref{eq:B_by_Bnmpq}, \eqref{eq:defphi}, Corollary~\ref{Cor:Bnmpq_H2}, and Lemma~\ref{Lem:B0_H1_H2}.
 	
	Moreover, we set
	\begin{equation}\label{eq:a_mu}
 		a_\mu :=\frac{\mu^+ -\mu^-}{\mu^+ +\mu^-}\in(-1,1)
 	\end{equation}
	and infer from  Theorem~\ref{Thm:D(f)_Res_H2} that there exists a unique solution  $\beta=\beta(f)\in\rmH^2(\Ss)^2$ to
 	\begin{equation}\label{eq:Db=V}
 		\big(1+2 a_\mu\DD(f)\big)[\beta]= \mcV(f).
 	\end{equation}

	\begin{Theorem}\label{Thm:Ftp}
	Given $f\in\rmH^3(\Ss)$, the transmission boundary value problem 	
	\begin{equation}\label{eq:Ftp}
		\left.
		\arraycolsep=1.4pt
		\begin{array}{rclll}
			\mu^\pm\Delta v^\pm-\nabla q^\pm&=&0&\mbox{in $\Omega^\pm$,}\\
			\vdiv v^\pm&=&0&\mbox{in $\Omega^\pm$,}\\{}
			[v]&=&0&\mbox{on $\Gamma$,}\\{}
			[T_\mu (v,q)]\tnu&=&(\Theta x_2-\sigma\tkappa)\tnu&\mbox{on $\Gamma$,}\\
			(v^\pm,q^\pm)(x)&\to&\Big(\pm\frac{c_{1,\Gamma}}{\mu^\pm},\pm\frac{c_{2,\Gamma}}{\mu^\pm},\pm c_{3,\Gamma}\Big)&
			\begin{minipage}{3.05cm} 
				for $x_2\to\pm\infty$ uni-\\[-0.5ex]
				formly in $x_1\in \Ss$,
			\end{minipage}
		\end{array}\right\}
	\end{equation} 
	has a solution~${(v,q)\in X_f}$  if and only if the constants $c_{i,\Gamma}$, $i=1,\,2,\,3$, are given by 
 	\begin{equation}\label{eq:Ftp_c}
 		c_{1,\Gamma}=-\frac{\sigma}{2}\bigg\langle \frac{f'}{(1+f'^2)^{ 1/2}}\bigg\rangle+ a_\mu\langle\beta_1(f)-f'\beta_2(f)\rangle,\qquad c_{2,\Gamma}= 0,\qquad c_{3,\Gamma} =-\frac{\Theta}{2}\langle f\rangle,
	\end{equation} 	 
	where $\beta(f)$ is defined in \eqref{eq:Db=V}. For this choice of the constants, the solution is also unique and 
 	\begin{equation}\label{eq:v_on_Gamma}
 		\{v\}^\pm\circ\Xi=\frac{1}{\mu^\pm}\bigg[\mcV(f)+2a_\mu\bigg(\pm\frac{1}{2}-\DD(f) \bigg)[\beta(f)]\bigg]=\frac{2}{\mu^++\mu^-} \beta(f).
 	\end{equation}
 Moreover, the constant $c_{1,\Gamma}$ satisfies
 	\begin{equation}\label{eq:c_by_curl}
 		c_{1,\Gamma}=-\frac{\mu^+\mu^-}{2\pi(\mu^++\mu^-)}\PV\int_{\Ss\times\RR}\curl v\,\rmd{x}:=-\frac{\mu^+\mu^-}{2\pi(\mu^++\mu^-)}\lim_{n\to\infty}\int_{\{|x_2|<n\}}\curl v\,\rmd{x}
 	\end{equation}
 	and 
	\begin{equation}\label{eq:T_lim}
		T_{\mu^\pm} (v^\pm,q^\pm)(x)\to\pm\frac{\Theta}{2}\langle f\rangle  I_2\qquad\text{for $x_2\to\pm\infty$ uniformly in $x_1\in \Ss$.}
	\end{equation}	
	\end{Theorem}
	
	\begin{proof}
	We devise the proof into several steps.\medskip
	
	\noindent{\em Uniqueness.}  Let $c_\Gamma:=(c_{1,\Gamma}, c_{2,\Gamma},  c_{3,\Gamma})\in\RR^3$ be arbitrary and assume that ${(v,q)\in X_f}$ is a solution to the homogeneous problem associated to \eqref{eq:Ftp} (that is with the right side of \eqref{eq:Ftp}$_4$ replaced by $0$). Since this problem is equivalent to problem~\eqref{eq:Ftp_0}, Proposition~\ref{Prop:Ftp_0} ensures that $c_\Gamma=0$ and~${(v,q)=(0,0)}$, which establishes the uniqueness claim.\medskip
	
	\noindent{\em Identification of $c_{2,\Gamma}$.} If ${(v,q)\in X_f}$ is a solution to \eqref{eq:Ftp}, we may integrate  equation~\eqref{eq:Ftp}$_2$ over~$\Omega^\pm_n$ with $n\geq \lVert f\rVert_\infty$, recalling \eqref{eq:Omega_n}, to deduce from Stokes' theorem and~\eqref{eq:Ftp}$_3$ that
	\begin{equation*}
		\int_{-\pi}^\pi v_2^+(x_1,n)\,\rmd x_1=\int_{-\pi}^\pi v_2^-(x_1,-n)\,\rmd x_1.	
	\end{equation*}
	Taking the limit $n\to\infty$, it follows from~\eqref{eq:Ftp}$_5$ that indeed $c_{2,\Gamma}=0$.\medskip

	\noindent{\em Existence.} We first associate to \eqref{eq:Ftp} the transmission boundary value problem
	\begin{equation}\label{eq:Ftp_s}
		\left.
		\arraycolsep=1.4pt
		\begin{array}{rclll}
			\Delta v_s^\pm-\nabla q_s^\pm&=&0&\mbox{in $\Omega^\pm$,}\\
			\vdiv v_s^\pm&=&0&\mbox{in $\Omega^\pm$,}\\{}
			[v_s]&=&0&\mbox{on $\Gamma$,}\\{}
			[T_1 (v_s,q_s)]\tnu&=&(\Theta x_2-\sigma\tkappa)\tnu&\mbox{on $\Gamma$,}\\
			(v_s^\pm,q_s^\pm)(x)&\to&\pm(c_{1,s},0, c_{3,s})&
			\begin{minipage}{3cm} 
				for $x_2\to\pm\infty$ uni-\\[-0.5ex]
				formly in $x_1\in \Ss$,
			\end{minipage}
		\end{array}\right\}
 	\end{equation}
 	with 
 	\begin{equation}\label{eq:cs}
		c_{1,s}=-\frac{\sigma}{2}\bigg\langle \frac{f'}{(1+f'^2)^{1/2}}\bigg\rangle \qquad\text{and}\qquad c_{3,s}=-\frac{\Theta}{2}\langle f\rangle.
	\end{equation}
 	Noticing that $\big(\Theta x_2-\sigma\tkappa\big)\tnu=\big((\omega(f))^{-1}G(f)\big)\circ \Xi^{-1}$, where
	\begin{equation}\label{eq:defG}
		G(f):=\Theta(- ff', f)-\sigma\big(\phi(f)\big)'\in  \rmH^1(\Ss)^2 \qquad\text{and}\qquad  \langle G_1(f)\rangle=0,
	\end{equation}
	it follows from \cite[Theorem~2.2]{Bohme.2024} (with $\mu=1$ therein) that  there exists a unique solution~${(v_s,q_s)\in X_f}$ to~\eqref{eq:Ftp_s}. 
	Moreover, by~\cite[Eq. (C.19), Remark 2.1, and Remark~C.5]{Bohme.2024}, the trace of $v_s$ on~$\Gamma$ is given by the relation
	\begin{equation}\label{eq:vs_on_Gamma}
 		\{v_s\}^\pm\circ\Xi =\mcV(f),
 	\end{equation}
 	with $\mcV(f)$  defined in~\eqref{eq:defV}.
 	 Moreover, we infer from (the proof of) \cite[Lemma~C.2, Lemma~C.4 and  Lemma~C.6]{Bohme.2024} that 
	\begin{equation}\label{eq:Ts_lim}
		T_1 (v_s^\pm,q_s^\pm)(x)\to\pm\frac{\Theta }{2}\langle f\rangle  I_2\qquad\text{for $x_2\to\pm\infty$ uniformly in $x_1\in \Ss$.}
	\end{equation}	 
 	
	Let now $(v_d,q_d)\in X_f$ denote the unique solution to~\eqref{eq:Ftp_b} determined in  Theorem~\ref{Thm:Ftp_b} with $\beta$ therein replaced by $2a_\mu\beta(f)\in\rmH^2(\Ss)^2$ (and $\beta(f)$ defined in~\eqref{eq:Db=V}) and set
	\begin{equation}\label{eq:sol}
		v^\pm:=\frac{1}{\mu^\pm}(v_s^\pm+v_d^\pm)\qquad\text{and}\qquad q^\pm:=q_s^\pm+q_d^\pm.
	\end{equation}  
	Then, it is straightforward to check that $(v,q)\in X_f$ satisfies  the relations \eqref{eq:Ftp}$_{1-2}$ and~\eqref{eq:Ftp}$_{4-5}$. Moreover, since by Lemma~\ref{Lem:vq_bd}, the one-sided traces of $v$ on $\Gamma$ satisfy \eqref{eq:v_on_Gamma}, the equation \eqref{eq:Ftp}$_{3}$ is equivalent to $\beta(f)$ solving the equation~\eqref{eq:Db=V}, and therefore~$(v,q)\in X_f$ defined in \eqref{eq:sol} is indeed a solution to \eqref{eq:Ftp}. Recalling \eqref{eq:stress_inf}, the relation \eqref{eq:T_lim} is a direct consequence of \eqref{eq:sol},~\eqref{eq:Ts_lim}, and~\eqref{eq:Tb_lim}.\medskip

\noindent{\em Identification of $c_{1,\Gamma}$.} Let ${(v,q)\in X_f}$ be the unique solution to \eqref{eq:Ftp}. Integrating $\curl v$ over the domain $\{|x_2|<n\}$, with $n>\lVert f\rVert_\infty$, Stokes' theorem together with \eqref{eq:Ftp}$_{3}$ and \eqref{eq:Ftp}$_{5}$  yields
	\begin{equation*}
	\begin{aligned}
		\int_{\{|x_2|<n\}}\curl v\,\rmd{x}&=\int_{-\pi}^\pi v_1^-(x_1,-n)-v_1^+(x_1,n)\,\rmd{x_1}+\int_\Gamma\big[(-v_2,v_1)^\top\big]\cdot \tnu\,\rmd\sigma\\
		&\underset{n\to\infty}\longrightarrow -2\pi c_{1,\Gamma}\frac{\mu^++\mu^-}{\mu^+\mu^-},
	\end{aligned}
	\end{equation*}
	and \eqref{eq:c_by_curl} follows.
 	\end{proof}

	\subsection*{An equivalent formulation for~\texorpdfstring{\eqref{eq:STOKES}}{(1.1)}}\label{SS:5.2}
 	With Theorem~\ref{Thm:Ftp} at hand, we can now reformulate problem \eqref{eq:STOKES}  
 	 as an  evolution problem for~$f$. Indeed, if $(f,v^\pm,q^\pm,c_\Gamma)$ is a solution to \eqref{eq:STOKES} as defined in Theorem~\ref{Thm:Main}~(i), the kinematic boundary condition \eqref{eq:Stokes}$_6$ together with \eqref{eq:v_on_Gamma} implies that~$f$ solves the evolution problem 
	\begin{equation}\label{eq:ev_pb}
		\frac{\rmd f}{\rmd t}(t)=\Psi(f(t)),\quad t\geq 0,\qquad f(0)=f_0,
	\end{equation}
	where
	\begin{equation}\label{eq:def_Psi}
		\Psi(f):=\frac{2}{\mu^++\mu^-}(-f',1)^\top\cdot\beta(f),
	\end{equation}
	with $\beta(f):=(\beta_1(f),\beta_2(f))^\top$ defined in \eqref{eq:Db=V}.

	To formulate problem~\eqref{eq:ev_pb} within a suitable functional-analytic framework, we fix $r\in (3/2,2)$ and infer from \cite[Lemma~3.1]{Bohme.2024}, for the mapping $\phi$ introduced in~\eqref{eq:defphi}, that
	\begin{equation}\label{eq:phi_smooth}
		\big[f \mapsto \phi(f)\big] \in \rmC^\infty\big(\rmH^r(\Ss),\rmH^{r-1}(\Ss)^2\big).
	\end{equation}
	This property, together with the relations \eqref{eq:B_by_Bnmpq} and \eqref{eq:reg_B}, shows that the mapping $\mcV$ defined in~\eqref{eq:defV} satisfies
	\begin{equation}\label{eq:V_smooth}
		\big[f \mapsto \mcV_i(f)\big] \in \rmC^{\infty}\big(\rmH^r(\Ss), \rmH^{r-1}(\Ss)\big),\qquad 1\leq i\leq 4.
	\end{equation}
	Moreover, $\phi$ and $\mcV$ both map bounded sets in $\rmH^r(\Ss)$ to bounded sets in $\rmH^{r-1}(\Ss)^2$. Indeed, for~$\phi$ this follows by a straightforward computation, while the corresponding property for $\mcV$ relies additionally on Lemma~\ref{Lem:Bnmpq_L2_Hr-1}~(i) and \cite[Lemma~A.7]{Bohme.2024}.

	Recalling~\eqref{eq:D(f)_smooth} and using the smoothness of the mapping which associates to an invertible bounded operator in $\mcL\big(\rmH^{r-1}(\Ss)^2\big)$ its inverse, we conclude together with \eqref{eq:V_smooth} and Theorem~\ref{Thm:D(f)_Res_Hr-1}  that the mapping $\beta$, defined in \eqref{eq:Db=V}, satisfies
	\begin{equation}\label{eq:beta_smooth}
		\big[f\mapsto \beta(f)\big]\in\rmC^{\infty}\big(\rmH^r(\Ss),\rmH^{r-1}(\Ss)^2\big).
	\end{equation}
	Moreover, $\beta$ inherits the property to map bounded sets in $\rmH^r(\Ss)$ to bounded sets in $\rmH^{r-1}(\Ss)^2$ from~$\mcV$ in view of the estimate \eqref{eq:D(f)_Res_Hr-1}.
	 We now directly infer from \eqref{eq:beta_smooth} that
	\begin{equation}\label{eq:Psi_smooth}
		\big[f\mapsto \Psi(f)\big]\in\rmC^{\infty}\big(\rmH^r(\Ss),\rmH^{r-1}(\Ss)\big)
	\end{equation}
	and that $\Psi$ maps bounded sets in $\rmH^r(\Ss)$ to bounded sets in $\rmH^{r-1}(\Ss)$.

	\subsection*{The Fr\'echet derivative of \texorpdfstring{$\Phi$}{Phi}}\label{SS:5.3}	
	To show the local well-posedness of the evolution problem~\eqref{eq:ev_pb}, it is convenient to apply the abstract parabolic theory from \cite[Chapter~8]{Lunardi.1995}. In addition to \eqref{eq:Psi_smooth}, the following result is therefore required.

	\begin{Proposition}\label{Prop:Psi_Gen}
	Given $r\in(3/2,2)$ and  $f_0\in\rmH^{r}(\Ss)$, the Fr\'{e}chet derivative $\partial\Psi(f_0)$ generates a strongly continuous analytic semigroup in $\mcL\big(\rmH^{r-1}(\Ss)\big)$.
	\end{Proposition}
	
	The proof of Proposition~\ref{Prop:Psi_Gen} is postponed to the next section, as it requires some preparation, part of which is carried out here by computing the Fréchet derivative $\partial\Psi(f_0)$ and identifying its highest-order terms. In what follows, we fix $f_0\in\rmH^r(\Ss)$, $r\in (3/2,2)$, and $r'\in(3/2,r)$. Then, it is straightforward to infer from \eqref{eq:def_Psi} that
	\begin{equation}\label{eq:DPsi}
		\partial\Psi(f_0)[f]=\frac{2}{\mu^+ +\mu^-}\big((-f_0',1)^\top\cdot\partial\beta(f_0)[f] -f'\beta_1(f_0)\big),\qquad f\in\rmH^r(\Ss).
	\end{equation}
	Moreover, differentiating the relation \eqref{eq:Db=V}, we get
	\begin{equation}\label{eq:D_Db=V}
		\big(1+2a_\mu\DD(f_0)\big)\big[\partial\beta(f_0)[f]\big] =\partial \mcV(f_0)[f]-2a_\mu\partial\DD(f_0)[f][\beta(f_0)],\qquad f\in\rmH^r(\Ss).
	\end{equation}
	
	In order to handle these derivatives in an efficient way, we infer  from~\eqref{eq:reg_B} that
	\begin{equation}\label{eq:Bnmpq_lo}
		\big\lVert B_{n,m}^{p,q}(f_0)[\varphi]\big\rVert_{\rmH^{r-1}}\leq C\lVert\varphi\rVert_{\rmH^{r'-1}},\qquad 1\leq p\leq n+q+1,\quad \varphi\in\rmH^{r-1}(\Ss).
	\end{equation}
	Moreover, in view of~\cite[Lemma~A.2]{Bohme.2024}  we also have
	\begin{equation}\label{eq:Bnmq=Cnm}
		B_{n,m}^{0,q}(f_0)[\varphi]=C_{n+q,m}^0(f_0)[\varphi]+R[\varphi],\qquad \Vert R[\varphi]\rVert_{\rmH^{r-1}}\leq C\Vert \varphi\rVert_{\rmH^{r'-1}},\quad \varphi\in\rmH^{r-1}(\Ss).
	\end{equation}
	 The constant $C$ in \eqref{eq:Bnmpq_lo} and \eqref{eq:Bnmq=Cnm}  depends only on $\lVert f_0\rVert_{\rmH^r}$ and on the indices $n,\, m,\, p,\, q$. 
	 Additionally,  we define for $n,\,m,\,p,\,q\in\NN_0$ with $p\leq n+q+ 2$ the operators
	\begin{equation}\label{eq:Bnmpq1}
	\begin{aligned}
		B_{n,m}^{0,q,1}&:\rmH^r(\Ss)\to \mcL\big(\rmH^r(\Ss),\mcL\big(\rmH^{r-1}(\Ss)\big)\big),\\
		B_{n,m}^{p,q,1}&:\rmH^r(\Ss)\to \mcL\big(\rmH^r(\Ss),\mcL\big(\rmH^{r-1}(\Ss),\rmH^r(\Ss)\big)\big),\qquad 1\leq p\leq n+q+2,
	\end{aligned}
	\end{equation}
	by setting for  $f\in\rmH^r(\Ss)$
	\begin{equation*}
		B_{n,m}^{p,q,1}(f_0)[f]:= B_{n,m}^{p,q+1}(f_0,\dots,f_0\vert f_0,\dots,f_0)[f_0,\dots,f_0,f,\cdot]\in \mcL\big(\rmH^{r-1}(\Ss)\big).
	\end{equation*}
	For some fixed $\varphi_0\in\rmH^{r-1}(\Ss)$ we infer from~\cite[Lemma~A.2]{Bohme.2024} that
	\begin{equation}\label{eq:Bnmq=Cnm1}
		B_{n,m}^{0,q,1}(f_0)[f][\varphi_0] =C_{n+q+1,m}(f_0,\ldots,f_0)[f_0,\ldots,f_0,f,\varphi_0]+\bar R[f],
	\end{equation}
	where, for a positive constant $C$ depending only on $\lVert f_0\rVert_{\rmH^r}$, $\lVert\varphi_0\rVert_{\rmH^{r-1}}$, and  $n,\, m,\, q$ we have 
	\begin{equation}\label{eq:Bnmq=Cnm2}
		\Vert \bar R[f]\rVert_{\rmH^{r-1}}\leq C\Vert f\rVert_{\rmH^{r'}},\qquad f\in\rmH^{r}(\Ss).
	\end{equation}
	
	Furthermore, in the arguments that follow we will use \cite[Lemma~A.1~(iv)]{Bohme.2024}, which  states:
	\begin{Lemma}\label{Lem:Cnm_com_Hr-1_Hr-1}
	Given $n\in\NN$, $m\in\NN_0$, and $\bfa\in\rmH^r(\Ss)^m$, there exists a constant~$ C>0$ that depends only on $n,\,m,\,r,\,r'$, and $\lVert\bfa\rVert_{\rmH^r}$ such that for all~${\bfb\in\rmH^r(\Ss)^n}$ and~${\varphi\in\rmH^{r-1}(\Ss)}$ we have
	\begin{equation}\label{eq:Cnm_com_Hr-1_Hr-1}
	 	\big\lVert C_{n,m}(\bfa)[\bfb,\varphi]-\varphi C_{n-1,m}(\bfa)[b_2,\dots,b_n,b_1']\big\rVert_{\rmH^{r-1}}\leq C\lVert b_1\rVert_{\rmH^{r'}}\lVert\varphi\rVert_{\rmH^{r-1}}\prod_{i=2}^{n} \lVert b_i\rVert_{\rmH^r}.
	\end{equation}
	\end{Lemma}
	
	From  \eqref{eq:Cnm} and~\eqref{eq:Cnm0} we next deduce the identity	
	\begin{equation}\label{eq:Cnm-1}
		C_{n,m}^{0}(f_0)+C_{n+2,m}^{0}(f_0)=C_{n,m-1}^{0}(f_0),\qquad m\geq 1,
	\end{equation}
	and, together with \eqref{eq:B_by_Bnmpq}, \eqref{eq:Bnmpq_lo}, \eqref{eq:Bnmq=Cnm},  we find
	\begin{equation}\label{eq:App_B}
	\begin{aligned}
	 	B_1(f_0)[\varphi]&=\big(C^0_{0,2}+C^0_{2,2}\big)(f_0)[\varphi]+R_{1}[\varphi],\\
	 	B_2(f_0)[\varphi]&=\big(C^0_{1,2}+C^0_{3,2}\big)(f_0)[\varphi] +R_{2}[\varphi],\\
	 	B_3(f_0)[\varphi]&=\big(C^0_{1,2}-C_{3,2}^0\big)(f_0)[\varphi]+R_{3}[\varphi],\\
	 	B_4(f_0)[\varphi]&=C_{2,2}^0(f_0)[\varphi]+ R_{4}[\varphi],
	\end{aligned}
	\end{equation}	 
	where, with a positive constant $C$ depending only on $\lVert f_0\rVert_{\rmH^r}$,
	\begin{equation}\label{eq:App_B_rest}
		\lVert R_{i}[\varphi]\rVert_{\rmH^{r-1}}\leq C\lVert \varphi\rVert_{\rmH^{r'-1}},\qquad \varphi\in\rmH^{r-1}(\Ss),\quad 1\leq i\leq 4.
	\end{equation}	
	Recalling \eqref{eq:D(f)}, it follows from~\eqref{eq:App_B} that 
	\begin{equation}\label{eq:D_by_C}
		\DD(f_0)[\beta]=-
		\begin{pmatrix}
			C_{1,2}^0 & C_{2,2}^0\\[1ex]
			C_{2,2}^0 & C_{3,2}^0
		\end{pmatrix}
		\begin{bmatrix}
		\beta_1 \\[1ex]
		\beta_2
		\end{bmatrix}
		+
		\begin{pmatrix}
			C_{0,2}^0 & C_{1,2}^0\\[1ex]
			C_{1,2}^0 & C_{2,2}^0
		\end{pmatrix}
		\begin{bmatrix}
		f_0'\beta_1 \\[1ex]
		f_0'\beta_2
		\end{bmatrix}
		+R_5[\beta],
	\end{equation}		
	where  by \eqref{eq:App_B_rest}, for a positive  constant $C$ depending only on $\lVert f_0\rVert_{\rmH^r}$, 
	\begin{equation}\label{eq:D_by_C_rest}
		\lVert R_5[\beta]\rVert_{\rmH^{r-1}}\leq C\lVert \beta\rVert_{\rmH^{r'-1}}, \qquad \beta\in\rmH^{r-1}(\Ss)^2.
	\end{equation}
	
	Using~\cite[Lemma~A.9]{Bohme.2024} and \eqref{eq:Bnmpq1}, we may express the Fr\'echet derivative of $\partial B_{n,m}^{p,q}(f_0)$  in the compact form 
	\begin{equation}\label{eq:Frechet_Bnmpq}
	\begin{aligned}
		\partial B_{n,m}^{p,q}(f_0)[f]&=n\big(B_{n-1,m}^{p,q,1}(f_0) - B_{n+1,m}^{p+2,q,1}(f_0)\big)[f]\\
		&\quad +2m\big(B_{n+3,m+1}^{p+2,q,1}(f_0)-B_{n+1,m+1}^{p,q,1}(f_0)\big)[f]\\
		&\quad+qB_{n,m}^{p,q-1,1}(f_0)[f],
	\end{aligned}
	\end{equation}
	where terms with negative indices have to be neglected. Recalling~\eqref{eq:B_by_Bnmpq}, we conclude from~\eqref{eq:Frechet_Bnmpq}  and~\eqref{eq:Bnmpq1}--\eqref{eq:Cnm-1} that
	\begin{equation}\label{eq:DB}
	\begin{aligned}
		\partial B_1(f_0)[f][\varphi_0]&=-2\varphi_0\big(C^0_{1,3}+C^0_{3,3}\big)(f_0)[f']+R_{ 6}[f],\\
		\partial B_2(f_0)[f][\varphi_0]&= \varphi_0\big(C^0_{0,3}-C_{4,3}^0\big)(f_0)[f']+R_{7}[f],\\
		\partial B_3(f_0)[f][\varphi_0]&= \varphi_0\big(C^0_{0,3}-6C_{2,3}^0+C_{4,3}^0\big)(f_0)[f']+R_{ 8}[f],\\
		\partial B_4(f_0)[f][\varphi_0]&= 2\varphi_0\big(C^0_{1,3}-C_{3,3}^0\big)(f_0)[f']+R_{9}[f],  
	\end{aligned}
	\end{equation}	 
	for $\varphi_0\in\rmH^{r-1}(\Ss)$ and, for a positive constant $C$  depending only on $\lVert f_0\rVert_{\rmH^r}$ and $\lVert\varphi_0\rVert_{\rmH^{r-1}}$, 
	\begin{equation}\label{eq:DB_rest}
		\lVert R_{i}[f]\rVert_{\rmH^{r-1}}\leq   C\lVert f\rVert_{\rmH^{r'}},\qquad f\in\rmH^r(\Ss),\quad 6\leq i\leq 9.
	\end{equation}
	
	By~\cite[Lemma 3.1]{Bohme.2024}, the Fr\'echet derivative of the coordinate $\phi_i$, $i=1,\, 2$, of the mapping $\phi$, introduced in~\eqref{eq:defphi}, is given by
 	\begin{equation}\label{eq:deriv_phi}
		\partial\phi_i(f_0)=a_i(f_0)\frac{\rmd}{\rmd \xi}\in\mcL\big(\rmH^r(\Ss),\rmH^{r-1}(\Ss)\big),\qquad i=1,\, 2,
	\end{equation}
	with $a_i(f_0)\in \rmH^{r-1}(\Ss)$, $i=1,\, 2$, defined by 
	\begin{equation}\label{eq:def_ai}
		a_1(f_0):=-\frac{f_0'}{(1+f_0'^2)^{3/2}}\qquad\text{and}\qquad a_2(f_0):=\frac{1}{(1+f_0'^2)^{3/2}}.
	\end{equation}
 We set for the sake of brevity 
	\begin{equation}\label{eq:defa_iphi_i}
		a_i:=a_i(f_0)\qquad\text{and}\qquad \phi_i:=\phi_i(f_0),\qquad i=1,\, 2.
	\end{equation}
	Using \eqref{eq:Frechet_Bnmpq} to differentiate \eqref{eq:def_Psi_i} and applying \eqref{eq:App_B}--\eqref{eq:App_B_rest} and \eqref{eq:DB}--\eqref{eq:DB_rest} gives
	\begin{equation}\label{eq:DPsi_i}
	\begin{aligned}
		\partial \mcV_1(f_0)[f]&=\big(C_{0,2}^{0}-C_{2,2}^{0}\big)(f_0)[(a_1-\phi_2-f_0' a_2)f']+C_{1,2}^{0}(f_0)[(3(\phi_1+f_0' a_1)+a_2)f']\\
		&\quad+C_{3,2}^{0}(f_0)[(\phi_1+f_0' a_1-a_2)f']\\
	 	&\quad+ \phi_1\big(3f_0' C_{0,3}^{0}-6C_{1,3}^{0}-6 f_0' C_{2,3}^{0}+2C_{3,3}^{0}-f_0' C_{4,3}^{0}\big)(f_0)[f']\\
	 	&\quad+\phi_2\big(C_{0,3}^{0}+6f_0' C_{1,3}^{0}-6C_{2,3}^{0}-2f_0' C_{3,3}^{0}+C_{4,3}^{0}\big)(f_0)[f']+ R_{ 10}[f],\\
		\partial \mcV_2(f_0)[f]&=\big(C_{1,2}^{0}-C_{3,2}^{0}\big)(f_0)[(a_1-\phi_2-f_0' a_2)f'] -C_{0,2}^{0}(f_0)[(\phi_1+f_0' a_1-a_2)f']\\
		&\quad+C_{2,2}^{0}(f_0)[(\phi_1+f_0' a_1+3a_2)f']\\
		&\quad+ \phi_1\big(C_{0,3}^{0}+6f_0' C_{1,3}^{0}-6 C_{2,3}^{0}-2f_0' C_{3,3}^{0}+C_{4,3}^{0}\big)(f_0)[f']\\
		&\quad-\phi_2\big(f_0' C_{0,3}^{0}-2 C_{1,3}^{0}-6f_0' C_{2,3}^{0} +6 C_{3,3}^{0}+f_0' C_{4,3}^{0}\big)(f_0)[f']+R_{11}[f],
	\end{aligned}
	\end{equation}
	where, for a positive constant $C$ depending only on $\lVert f_0\rVert_{\rmH^r}$, 
	\begin{equation}\label{eq:DPsi_i_rest}
		\lVert R_{i}[f]\rVert_{\rmH^{r-1}}\leq C\lVert f\rVert_{\rmH^{r'}}, \qquad f\in\rmH^r(\Ss),\quad i=10,\,11.
	\end{equation}	

	Lastly,  we set
	\begin{equation}\label{eq:beta_0}
		\beta_0:=(\beta_{0,1},\beta_{0,2})^\top:=\beta(f_0)\in \rmH^{r-1}(\Ss)^2
	\end{equation}
	and use \eqref{eq:Frechet_Bnmpq} to differentiate \eqref{eq:D(f)}. Applying \eqref{eq:App_B}--\eqref{eq:App_B_rest} and~\eqref{eq:DB}--\eqref{eq:DB_rest}, we arrive at
	 \begin{equation}\label{eq:DDD_matrix}
	\begin{aligned}
		\partial\DD(f_0)[f][\beta_0]&= 
		\begin{pmatrix}
			C_{0,2}^0(f_0) & C_{1,2}^0(f_0)\\[1ex]
			C_{1,2}^0(f_0) & C_{2,2}^0(f_0)
		\end{pmatrix}
		\begin{bmatrix}
			f'\beta_{0,1} \\[1ex]
			f'\beta_{0,2}
		\end{bmatrix}\\
		&\quad	-
		\begin{pmatrix}
			 \beta_{0,1}\big(C_{0,3}^0 -3C_{2,3}^0\big)(f_0)[f']+2 \beta_{0,2}\big(C_{1,3}^0 -C_{3,3}^0\big)(f_0)[f']\\[1ex]
			 2 \beta_{0,1}\big(C_{1,3}^0 -C_{3,3}^0\big)(f_0)[f']+ \beta_{0,2}\big(3C_{2,3}^0 -C_{4,3}^0\big)(f_0)[f']
		\end{pmatrix}\\
&\quad +	\begin{pmatrix}
			 -4f_0' \beta_{0,1}C_{1,3}^0 (f_0)[f']+f_0' \beta_{0,2}\big(C_{0,3}^0 -3C_{2,3}^0\big)(f_0)[f']\\[1ex]
			 f_0' \beta_{0,1}\big(C_{0,3}^0 -3C_{2,3}^0\big)(f_0)[f']+2f_0' \beta_{0,2}\big(C_{1,3}^0 -C_{3,3}^0\big)(f_0)[f']
		\end{pmatrix}
		+R_{12}[f],
	\end{aligned}
	\end{equation}
	where, for some constant $C$ depending only on  $\lVert f_0\rVert_{\rmH^r}$ and $\lVert\beta_0\rVert_{\rmH^{r-1}}$,
	\begin{equation}\label{eq:DDD_rest}
		\lVert R_{12}[f]\rVert_{\rmH^{r-1}}\leq C\lVert f\rVert_{\rmH^{r'}}, \qquad f\in\rmH^r(\Ss).
	\end{equation}	
	
	\subsection*{Localization and the proof of Theorem~\ref{Thm:Main}}\label{SS:Loc}	
	The first part of this section is devoted to establishing Proposition~\ref{Prop:Psi_Gen}. 
	This is achieved by  adopting the strategy  employed  in \cite{MP2022}; see also~\cite{E94, ES95, Matioc.2019}. 
	The main step consists of locally approximating the Fr\'echet derivative $\partial\Psi(f_0)$ by certain Fourier multipliers, which themselves generate strongly continuous analytic semigroups; see Proposition~\ref{Prop:loc}. 
	We conclude the section by outlining the proof of Theorem~\ref{Thm:Main}.
		
	To start, for each $\ve\in(0,1)$ we fix a set of smooth functions $\{\pi_{j}^{\ve}: 1\leq j\leq N\}\subset \rmC^\infty(\Ss,[0,1])$, where the integer~$N=N(\ve)$ is sufficiently large, such that
	\begin{equation}\label{eq:pi_jp}
	\begin{aligned}
		\bullet & \,\supp \pi_j^\ve= I_j^\ve+2\pi \ZZ\text{ with } I_j^\ve:= [\xi_j^\ve-\ve,\xi_j^\ve+\ve] \text{ and } \xi_j^\ve:=j\ve;\\
		\bullet & \, \sum_{j=1}^{N}\pi_j^\ve =1 \text{ in } \rmC^\infty(\Ss). 
	\end{aligned}
	\end{equation}	 
	We call $\{\pi_{j}^{\ve}: 1\leq j\leq N\}$ an $\ve$\emph{-partition of unity}. To a given $\ve$-partition of unity, we associate a further set $\{\chi_j^\ve: 1\leq j \leq N\}\subset\rmC^\infty(\Ss,[0,1])$ with
	\begin{equation}\label{eq:chi_jp}
	\begin{aligned}
		\bullet & \,\supp \chi_j^\ve =J_j^\ve+2\pi \ZZ \text{ with } J_j^\ve=[\xi_j^\ve-2\ve,\xi_j^\ve+2\ve] ;\\
		\bullet & \,\chi_j^\ve =1 \text{ on } \supp \pi_j^\ve.
	\end{aligned}
	\end{equation}
	We associate to such an $\ve$-partition of unity a new norm on $\rmH^s(\Ss)$, $s\geq 0$, via the mapping
	\begin{equation*}	 
		\bigg[f\mapsto\sum_{j=1}^{N}\lVert\pi_j^\ve f\rVert_{\rmH^s}\bigg]:\rmH^s(\Ss)\to\RR,
	\end{equation*}
	which is equivalent to the standard norm in the sense that for some $c=c(\ve,s)\in(0,1)$ we have
	\begin{equation}\label{eq:norm_pi_jp}
		c\lVert f\rVert_{\rmH^s}\leq \sum_{j=1}^{N}\lVert\pi_j^\ve f\rVert_{\rmH^s}\leq c^{-1}\lVert f\rVert_{\rmH^s},\qquad f\in \rmH^s(\Ss).
	\end{equation}
	
	Similarly to the non-periodic case~\cite{MP2022}, let ${\Phi:[0,1]\to\mcL\big(\rmH^r(\Ss),\rmH^{r-1}(\Ss)\big)}$ be the continuous path given by
	\begin{equation}\label{eq:Phi_tau}
		\Phi(\tau)[f]:=\frac{2}{\mu^+ +\mu^-}\big(-\tau f' \beta_{0,1}+(-\tau f_0',1)^\top\cdot\mcB(\tau)[f] \big),\qquad \tau\in [0,1],\quad f\in\rmH^r(\Ss),
	\end{equation}	 
	 where $\beta_0=(\beta_{0,1},\beta_{0,2})^\top$  was introduced in~\eqref{eq:beta_0} and $\mcB(\tau)[f]\in \rmH^{r-1}(\Ss)^2$, $\tau\in [0,1]$, is the unique solution to
	\begin{equation}\label{eq:B_tau}
		\big(1+2\tau a_\mu\DD(f_0)\big)\big[\mcB(\tau)[f]\big]=\scV(\tau)[f]-2\tau a_{\mu}\partial\DD(f_0)[f][\beta_0].
	\end{equation}	
	Here, $\scV:[0,1]\to \mcL\big(\rmH^{r}(\Ss),\rmH^{r-1}(\Ss)^2\big)$ denotes the continuous mapping
	\begin{equation}\label{eq:V_tau}
		\scV(\tau):=\frac{1}{4}\big(-\sigma\partial\mcV_1(\tau f_0)-\tau\Theta\partial\mcV_3(f_0),\,-\sigma\partial\mcV_2(\tau f_0)+\tau\Theta\partial\mcV_4(f_0)+\tau\Theta\ln(4)\langle\cdot\rangle\big)^\top
	\end{equation}
	for $\tau\in [0,1]$. Note that $\mcB:[0,1]\to\mcL\big(\rmH^r(\Ss),\rmH^{r-1}(\Ss)^2\big)$ is continuous  as well and, using Theorem~\ref{Thm:D(f)_Res_Hr-1},~\eqref{eq:D(f)_smooth}, and~\eqref{eq:V_smooth}, we may find a constant $C>0$ such that
	\begin{equation}\label{eq:mcB_Hr-1}
		\lVert \mcB(\tau)[f]\rVert_{\rmH^{r-1}}\leq C\lVert f\rVert_{\rmH^r},\qquad f\in\rmH^r(\Ss),\quad \tau\in[0,1].
	\end{equation}
	Moreover, we observe that  $\Phi(1)=\partial\Psi(f_0)$ and that
	\begin{equation*}
		\scV(0)=\bigg(0,\,-\frac{\sigma}{4}H\circ\frac{\rmd}{\rmd\xi}\bigg)^\top,
	\end{equation*}
	where $H=B_{0,0}^{0,0}$ is the periodic Hilbert transform, c.f. \eqref{eq:HT}. It follows that 
	\begin{equation}\label{eq:Phi0}
		\Phi(0)=-\frac{\sigma}{2(\mu^+ +\mu^-)}\bigg(-\frac{\rmd^2}{\rmd\xi^2}\bigg)^{1/2},
	\end{equation}
	since the symbol of the periodic Hilbert transform is given by $(-i\,\mathrm{sign}(k))_{k\in\ZZ}$.  The invertibility of the operator $\lambda-\Phi(0)$ for $\lambda>0$ will later be used to conclude the invertibility of $\lambda-\Phi(1)$ for $\Rel\lambda$ large enough.
	
	In Proposition~\ref{Prop:loc}, we approximate not only $\partial\Psi(f_0)=\Phi(1)$ by Fourier multipliers but also the entire path~$\Phi(\tau)$, $\tau\in[0,1]$.
		
	\begin{Proposition}\label{Prop:loc}
	Given $\gamma>0$, there exist $\ve\in(0,1)$, an $\ve$-partition of unity ${\{\pi_{j}^{\ve} : 1\leq j\leq N\}}$, a constant $K=K(\ve)>0$, and bounded operators
	\begin{equation*}
		\Aa_{j,\tau}\in\mcL\big(\rmH^r(\Ss), \rmH^{r-1}(\Ss) \big), \qquad 1\leq j\leq N, \quad \tau\in[0,1],
	\end{equation*}
	such that
	\begin{equation}\label{eq:Phi_AA_approx}
		\big\lVert \pi_j^\ve \Phi(\tau)[f] - \Aa_{j,\tau}[\pi_j^\ve f] \big\rVert_{\rmH^{r-1}}\leq \gamma \lVert \pi_j^\ve f \rVert_{\rmH^r} + K \lVert f\rVert_{\rmH^{r'}}
	\end{equation}
	for all $1\leq j\leq N$, $f\in \rmH^r(\Ss)$, and $\tau\in[0,1]$. The operators $\Aa_{j,\tau}$ are defined by
	\begin{equation}\label{eq:A_jt}
		\Aa_{j,\tau} :=-\alpha_{\tau}(\xi_j^\ve)\bigg( -\frac{\rmd^2}{\rmd\xi^2}\bigg)^{1/2}+\vt_\tau(\xi_j^\ve)\frac{\rmd}{\rmd\xi},\qquad 1\leq j\leq N,\quad \tau\in[0,1],
	\end{equation}
	where the functions $\alpha_\tau$ and $\vt_\tau$ are given, recalling \eqref{eq:nutauomega} and \eqref{eq:beta_0}, by
	\begin{equation}\label{eq:alphabeta}
		\alpha_{\tau}:=\frac{\sigma}{2(\mu^+ +\mu^-)}(\omega(\tau f_0))^{-1} \qquad\text{and}\qquad \vt_\tau := -\frac{2\tau}{\mu^+ + \mu^-} \beta_{0,1}.
	\end{equation}
	\end{Proposition}
	
	The proof of this proposition uses the following commutator-type result for the operators $C_{n,m}$.
	\begin{Lemma}\label{Lem:Cnm_com}
	Given $n,\,m\in\NN_0$  and $a,\,f\in\rmC^1(\Ss)$, there exists a constant $C>0$ that depends only on $n,\,m$, and $\lVert (a,f)\rVert_{\rmC^1}$ such that for all $\varphi\in\rmL^2(\Ss)$ we have
	\begin{equation}\label{eq:Cnm_com}
		\big\lVert a C_{n,m}^{0}(f)[\varphi]-C_{n,m}^{0}(f)[a\varphi]\big\rVert_{\rmH^1}\leq C\lVert\varphi\rVert_2.
	\end{equation}
	\end{Lemma}

	\begin{proof}
		See \cite[Lemma~B.1]{Bohme.2024}
	\end{proof}

	The main ingredient in the proof of Proposition~\ref{Prop:loc} is the next lemma, which describes how to ``freeze the kernel'' of the operators $C_{n,m}$.
	\begin{Lemma}\label{Lem:Cnm_approx_a}
	Let $n,\,m\in\NN_0$, $r'\in(3/2,r)$, $f\in\rmH^r(\Ss)$, $a,\,b\in\rmH^{r-1}(\Ss)$, and~${\eta>0}$ be given. Then, for any sufficiently small $\ve\in(0,1)$, there exists a constant~$K>0$ that depends on~$\ve,\,n,\,m,\,\lVert f\rVert_{\rmH^r}$, and $\lVert (a,b)\rVert_{\rmH^{r-1}}$ such that for all $1\leq j\leq N$ and $\varphi\in\rmH^{r-1}(\Ss)$ we have 
	\begin{equation}\label{eq:Cnm_approx_a}
		\bigg\lVert \pi_j^\ve a C_{n,m}^0 (f)[b\varphi]- \frac{abf'^n}{\big(1+f'^2\big)^m}(\xi_j^\ve) H[\pi_j^\ve \varphi]\bigg\rVert_{\rmH^{r-1}}\leq \eta \lVert \pi_j^\ve \varphi\rVert_{\rmH^{r-1}}+K\lVert\varphi\rVert_{\rmH^{r'-1}}.
	\end{equation}
	\end{Lemma}

	\begin{proof}
		See \cite[Lemma~B.2]{Bohme.2024}
	\end{proof}

	We now prove Proposition~\ref{Prop:loc}, following closely the proof of \cite[Theorem~6.2]{MP2022}.

	\begin{proof}[Proof of Proposition~\ref{Prop:loc}]
	Let $\gamma>0$ be fixed, and let $\ve\in(0,1)$ (to be chosen later on), together with an associated $\ve$-partition of unity ${\{\pi_{j}^{\ve} : 1\leq j\leq N\}}$ and the corresponding family ${\{\chi_{j}^{\ve} : 1\leq j\leq N\}}$, satisfying \eqref{eq:pi_jp} and  \eqref{eq:chi_jp}, respectively. 
	We also  point out that, given  $s\in(1/2,1)$, there exists a constant $C=C(s)>0$ such that 
	\begin{equation}\label{eq:Hr_BAlg}
		\lVert fg\rVert_{\rmH^s}\leq C(\lVert f\rVert_\infty \lVert g\rVert_{\rmH^s}+\lVert g\rVert_\infty \lVert f\rVert_{\rmH^s}),\qquad f,\,g\in\rmH^s(\Ss).
	\end{equation}
 In what follows, the symbol $C$ is used for constants that are independent of~$\ve$, while $K$ denotes constants that depend on $\ve$.
  In view of  definition \eqref{eq:Phi_tau}, we need to approximate both operators~${\big[f\mapsto \mcB_2(\tau)[f]-\tau f_0'\mcB_1(\tau)[f]\big]}$ and~${\big[f\mapsto \tau f'\beta_{0,1}\big]}$. We start with the second one. Using the identity~${\chi_j^\ve \pi_j^\ve =\pi_j^\ve}$ together with~\eqref{eq:Hr_BAlg}, we  have
	\begin{equation}\label{eq:loc_e1}
	\begin{aligned}
		\lVert \pi_j^\ve\beta_{0,1}f'-\beta_{0,1}(\xi_j^\ve)(\pi_j^\ve f)'\rVert_{\rmH^{r-1}}
		&\leq C\lVert \chi_j^\ve(\beta_{0,1}-\beta_{0,1}(\xi_j^\ve))\rVert_\infty \lVert (\pi_j^\ve f)'\rVert_{\rmH^{r-1}}+K\lVert f\rVert_{\rmH^{r'}}\\
		& \leq \frac{\gamma(\mu^+ +\mu^-)}{6}\lVert\pi_j^\ve f\rVert_{\rmH^r}+K\lVert f\rVert_{\rmH^{r'}}
	\end{aligned}
	\end{equation}
	for all~${1\leq j\leq N}$ and~${f\in\rmH^r(\Ss)}$, provided that~$\ve$ is small enough. The estimate in the last line of~\eqref{eq:loc_e1} follows from~\eqref{eq:chi_jp} in view of~${\beta_{0,1}\in\rmH^{r-1}(\Ss)\hookrightarrow\rmC^{r-3/2}(\Ss)}$.
	
	Next, we show that
	\begin{equation}\label{eq:loc_e2}
		\lVert \pi_j^\ve \mcB(\tau)[f]\rVert_{\rmH^{r-1}}\leq C_\mcB \lVert \pi_j^\ve f\rVert_{\rmH^r}+K\lVert f\rVert_{\rmH^{r'}},
	\end{equation}
	with $C_\mcB$ independent of $\ve\in(0,1)$,~${\tau\in[0,1]}$,~${1\leq j\leq N}$, and~${f\in\rmH^r(\Ss)}$. To this end, we infer from~\eqref{eq:B_tau}, after multiplying this relation by $\pi_j^\ve$, that
	\begin{equation}\label{eq:loc_e3}
	\begin{aligned}
		\big(1+2\tau a_\mu  \DD(f_0) \big) \big[ \pi_j^\ve \mcB(\tau)[f] \big]&= \pi_j^\ve \scV(\tau)[f] - 2\tau a_\mu  \pi_j^\ve \partial \DD(f_0)[f][\beta_0] \\
		&\quad + 2\tau a_\mu  \big(\DD(f_0) \big[ \pi_j^\ve \mcB(\tau)[f] \big]-\pi_j^\ve \DD(f_0) \big[ \mcB(\tau)[f] \big]\big).
	\end{aligned}
	\end{equation}
	The terms on the first line of the right  side  of \eqref{eq:loc_e3} can be estimated according to 
	\begin{equation}\label{eq:loc_e4}
		\lVert \pi_j^\ve \scV(\tau)[f] \rVert_{\rmH^{r-1}} + \lVert \pi_j^\ve \partial \DD(f_0)[f][\beta_0] \rVert_{\rmH^{r-1}} \leq C \lVert \pi_j^\ve f \Vert_{\rmH^r} + K \lVert f\rVert_{\rmH^{r'}}.
	\end{equation}
	Indeed, recalling~\eqref{eq:def_Psi_i} and~\eqref{eq:V_tau}, the estimate~\eqref{eq:loc_e4} follows by observing from \eqref{eq:B_by_Bnmpq} and \eqref{eq:reg_B} that~${\partial\mcV_i(f_0)\in\mcL\big(\rmH^{r'}(\Ss)\big)}$ for $i=3,\,4$, so that in particular
	\begin{equation}\label{eq:loc_e8}
		\lVert \pi_j^\ve \partial\mcV_3(f_0)[f]\rVert_{\rmH^{r-1}} + \lVert\pi_j^\ve \partial\mcV_4(f_0)[f]\rVert_{\rmH^{r-1}}+\lVert\langle f \rangle \rVert_{\rmH^{r-1}}\leq K \lVert f\rVert_{\rmH^{r'}},
	\end{equation}
	the other terms, corresponding to $\partial \mcV_i(f_0)$, $i=1,\,2$, and to $\partial \DD(f_0)[f][\beta_0]$ being estimated using~\eqref{eq:DPsi_i}--\eqref{eq:DPsi_i_rest}, \eqref{eq:DDD_matrix}--\eqref{eq:DDD_rest}, Lemma~\ref{Lem:Bnmpq_L2_Hr-1}~(iv), and Lemma~\ref{Lem:Cnm_com}. Moreover, combining \eqref{eq:D_by_C}, Lemma~\ref{Lem:Cnm_com} and~\eqref{eq:mcB_Hr-1} (with $r=r'$) we get
	\begin{equation}\label{eq:loc_e5}
		\big\lVert \DD(f_0)\big[\pi_j^\ve \mcB(\tau)[f]\big]-\pi_j^\ve \DD(f_0)\big[\mcB(\tau)[f]\big]\big\rVert_{\rmH^{r-1}}\leq K\lVert\mcB(\tau)[f]\rVert_{\rmH^{r'-1}}\leq K\lVert f\rVert_{\rmH^{r'}},
	\end{equation}
	 the desired claim \eqref{eq:loc_e2} following now from Theorem~\ref{Thm:D(f)_Res_Hr-1} and \eqref{eq:loc_e3}--\eqref{eq:loc_e5}. 
	 
	 With the estimate~\eqref{eq:loc_e2} at hand, we next approximate the operator~${\big[f\mapsto \mcB(\tau)[f]\big]}$. To begin, we define Fourier multipliers~${\BB_{j,\tau}\in\mcL\big(\rmH^r(\Ss),\rmH^{r-1}(\Ss)^2\big)}$,~${1\leq j\leq N}$ and~${\tau\in[0,1]}$, by
	\begin{equation*}
		\BB_{j,\tau}:=-\frac{\sigma}{4}
		\begin{pmatrix}
			a_1(\tau f_0)(\xi_j^\ve)H\circ(\rmd/\rmd\xi)\\
			a_2(\tau f_0)(\xi_j^\ve)H\circ(\rmd/\rmd\xi)
		\end{pmatrix}.
	\end{equation*}
	Given $\tilde{\gamma}>0$, we show that if ~$\ve\in(0,1)$ is small enough, then
	\begin{equation}\label{eq:loc_e6}
		\lVert \pi_j^\ve \mcB(\tau)[f]-\BB_{j,\tau}[\pi_j^\ve f]\rVert_{\rmH^{r-1}}\leq \tilde{\gamma}\lVert \pi_j^\ve f\rVert_{\rmH^r}+K\lVert f\rVert_{\rmH^{r'}}
	\end{equation}
	for all~$\tau\in[0,1]$,~${1\leq j\leq N}$, and~${f\in\rmH^r(\Ss)}$. To this end we multiply \eqref{eq:B_tau} by $\pi_j^\ve$ and arrive at
	\begin{equation}\label{eq:loc_e7}
		\pi_j^\ve \mcB(\tau)[f]= \pi_j^\ve  \scV(\tau)[f]-2\tau a_\mu\pi_j^\ve\big(\partial\DD(f_0)[f][\beta_0]+\DD(f_0)\big[\mcB(\tau)[f]\big]\big).
	\end{equation}
	Recalling \eqref{eq:V_tau} and \eqref{eq:DPsi_i}--\eqref{eq:DPsi_i_rest}, the estimate~\eqref{eq:loc_e8} and repeated use of Lemma~\ref{Lem:Cnm_approx_a} in the context of \eqref{eq:DPsi_i} yield
	\begin{equation}\label{eq:loc_e9}
		\lVert \pi_j^\ve  \scV(\tau)[f]-\BB_{j,\tau}[\pi_j^\ve f]\rVert_{\rmH^{r-1}}\leq \frac{\tilde{\gamma}}{3 }\lVert \pi_j^\ve f\rVert_{\rmH^r}+K\lVert f\rVert_{\rmH^{r'}}
	\end{equation}
	for all~$\tau\in[0,1]$,~${1\leq j\leq N}$, and~${f\in\rmH^r(\Ss)}$, provided that~$\ve$ is small enough. Moreover, using the relations \eqref{eq:DDD_matrix}--\eqref{eq:DDD_rest}, we obtain, by repeatedly applying Lemma~\ref{Lem:Cnm_approx_a}, that
	\begin{equation}\label{eq:loc_e10}
		\lVert \pi_j^\ve \partial\DD(f_0)[f][\beta_0]\rVert_{\rmH^{r-1}}\leq \frac{\tilde{\gamma}}{6|a_\mu|+1}\lVert \pi_j^\ve f\rVert_{\rmH^r}+K\lVert f\rVert_{\rmH^{r'}}
	\end{equation}			
	for all~${1\leq j\leq N}$ and~${f\in\rmH^r(\Ss)}$, provided that~$\ve$ is small enough. 
	Similarly, by repeatedly applying Lemma~\ref{Lem:Cnm_approx_a} to \eqref{eq:D_by_C} and using \eqref{eq:D_by_C_rest}, \eqref{eq:mcB_Hr-1} (with $r=r'$), and \eqref{eq:loc_e2},  we get
	\begin{equation}\label{eq:loc_e11}
	\begin{aligned}
		\big\lVert \pi_j^\ve\DD(f_0)\big[\mcB(\tau)[f]\big]\big\rVert_{\rmH^{r-1}}& \leq\frac{\tilde{\gamma}}{(6|a_\mu|+1)C_\mcB}\lVert \pi_j^\ve \mcB(\tau)[f]\rVert_{\rmH^{r-1}}+K\lVert f\rVert_{\rmH^{r'}}\\
		&\leq\frac{\tilde{\gamma}}{6|a_\mu|+1}\lVert \pi_j^\ve f\rVert_{\rmH^r}+K\lVert f\rVert_{\rmH^{r'}}
	\end{aligned}
	\end{equation}
	for all~$\tau\in[0,1]$,~${1\leq j\leq N}$ and~${f\in\rmH^r(\Ss)}$, provided that~$\ve$ is small enough. 
	Recalling~\eqref{eq:loc_e7}, the estimate~\eqref{eq:loc_e6} is a straightforward consequence  of \eqref{eq:loc_e9}--\eqref{eq:loc_e11}, 
		
	To approximate the operators~${\big[f\mapsto \mcB_2(\tau)[f]-\tau f_0'\mcB_1(\tau)[f]\big]}$, we first infer from \eqref{eq:loc_e6}, that for~$\ve$ small enough, we have
	\begin{equation}\label{eq:loc_e12}
		\bigg\lVert \pi_j^\ve \mcB_2(\tau)[f]+\frac{  \sigma}{4}a_2(\tau f_0)(\xi_j^\ve)H[(\pi_j^\ve f)']\bigg\rVert_{\rmH^{r-1}}\leq \frac{\gamma(\mu^+ +\mu^-)}{6}\lVert \pi_j^\ve f\rVert_{\rmH^r}+K\lVert f\rVert_{\rmH^{r'}}
	\end{equation}
	for all~$\tau\in[0,1]$,~${1\leq j\leq N}$, and~${f\in\rmH^r(\Ss)}$. Moreover, using the identity $\pi_j^\ve =\chi_j^\ve \pi_j^\ve$, the regularity of $f_0\in\rmH^{r}(\Ss)\hookrightarrow\rmC^{r-1/2}(\Ss)$, and the estimates \eqref{eq:mcB_Hr-1} (with $r=r'$), \eqref{eq:loc_e2}, and \eqref{eq:loc_e6}, we obtain
	\begin{equation}\label{eq:loc_e13}
	\begin{aligned}
		&\bigg\lVert \pi_j^\ve f_0' \mcB_1(\tau)[f]+\frac{  \sigma}{4}f_0'(\xi_j^\ve)a_1(\tau f_0)(\xi_j^\ve)H[(\pi_j^\ve f)']\bigg\rVert_{\rmH^{r-1}}\\
		&\qquad \leq C\lVert \chi_j^\ve(f_0'-f_0'(\xi_j^\ve)) \rVert_\infty \lVert\pi_j^\ve\mcB_1(\tau)[f]\rVert_{\rmH^{r-1}}\\
		&\qquad\quad+C\bigg\lVert \pi_j^\ve \mcB_1(\tau)[f]+\frac{  \sigma}{4}a_1(\tau f_0)(\xi_j^\ve)H[(\pi_j^\ve f)']\bigg\rVert_{\rmH^{r-1}}+K\lVert f\rVert_{\rmH^{r'}}\\
		&\qquad \leq  \frac{\gamma(\mu^+ +\mu^-)}{6}\lVert \pi_j^\ve f\rVert_{\rmH^r}+K\lVert f\rVert_{\rmH^{r'}}
	\end{aligned}
	\end{equation}
	for all~$\tau\in[0,1]$,~${1\leq j\leq N}$, and~${f\in\rmH^r(\Ss)}$, provided that $\ve$ is small enough. 
		
	Recalling \eqref{eq:Phi_tau}, the desired claim \eqref{eq:Phi_AA_approx} follows from~\eqref{eq:loc_e1},~\eqref{eq:loc_e12}, and~\eqref{eq:loc_e13}.
	\end{proof}

	Since $f_0\in\rmH^{r}(\Ss)$, it follows from  \eqref{eq:def_ai} and  \eqref{eq:beta_0} that there exists a constant $\eta\in (0,1)$ such that the functions defined in \eqref{eq:alphabeta} satisfy
	\begin{equation*}
		\eta \leq \alpha_\tau \leq \eta^{-1}\qquad\text{and}\qquad |\vt_\tau|\leq \eta^{-1},\qquad \tau\in [0,1].
	\end{equation*}
	Next, we introduce the Fourier multiplier $\Aa_{\alpha,\vt}$ by
	\begin{equation*}
		\Aa_{\alpha,\vt}:= -\alpha \bigg(-\frac{\rmd^2}{\rmd\xi^2}\bigg)^{1/2}+\vt\frac{\rmd}{\rmd\xi}	\in\mcL\big(\rmH^r(\Ss),\rmH^{r-1}(\Ss)\big),\qquad \alpha\in [\eta,\eta^{-1}],\quad \vt\in[-\eta^{-1}, \eta^{-1}].
	\end{equation*}
	Using Fourier analysis techniques, one can readily verify that for all ${\alpha \in [\eta, \eta^{-1}]}$ and ${|\vt| \leq \eta^{-1}}$, we have
	\begin{equation}\label{eq:l-A_iso}
    	\text{$\lambda - \Aa_{\alpha, \vt} \in\mcL\big(\rmH^r(\Ss),\rmH^{r-1}(\Ss)\big)$ is an isomorphism for all $\Rel \lambda \geq 1$.}
	\end{equation}
	Moreover, there exists a constant ${\kappa_0 = \kappa_0(\eta) \geq 1}$ such that for all ${\alpha \in [\eta, \eta^{-1}]}$ and ${|\vt| \leq \eta^{-1}}$,
	\begin{equation}\label{eq:l-A}
    	\kappa_0  \lVert (\lambda - \Aa_{\alpha, \vt})[f] \rVert_{\rmH^{r-1}}\geq |\lambda| \, \lVert f \rVert_{\rmH^{r-1}} + \lVert f \rVert_{\rmH^r}, 
    \qquad f \in \rmH^r(\Ss), \quad \Rel \lambda \geq 1.
\end{equation}

	By combining \eqref{eq:l-A_iso}–\eqref{eq:l-A} with Proposition~\ref{Prop:loc} and the interpolation property~\eqref{eq:interpolation}, we are now in a position to prove Proposition~\ref{Prop:Psi_Gen}.
	\begin{proof}[Proof of Proposition~\ref{Prop:Psi_Gen}]
	We begin by fixing $r' \in (3/2, r)$ and letting $\kappa_0 \geq 1$ denote the constant from estimate~\eqref{eq:l-A}. 
	Applying Proposition~\ref{Prop:loc} with $\gamma = 1/(2\kappa_0)$ we find~$\ve \in (0,1)$, an $\ve$-partition of unity~${\{\pi_j^\ve : 1 \leq j \leq N\}}$, a constant ${K = K(\ve) > 0}$, and operators
	\begin{equation*}
		\Aa_{j,\tau} \in \mcL\big(\rmH^r(\Ss), \rmH^{r-1}(\Ss)\big), \qquad 1 \leq j \leq N, \quad \tau \in [0,1],
	\end{equation*}
	such that for all~$\tau\in[0,1]$,~${1\leq j\leq N}$, and~${f\in\rmH^r(\Ss)}$ we have
	\begin{equation*}
    	2\kappa_0  \big\lVert \pi_j^\ve  \Phi(\tau)[f] - \Aa_{j,\tau}[\pi_j^\ve f] \big\rVert_{\rmH^{r-1}} \leq \lVert \pi_j^\ve f \rVert_{\rmH^r} + 2\kappa_0 K \lVert f \rVert_{\rmH^{r'}}.
	\end{equation*}
	Additionally,  we infer from \eqref{eq:A_jt} and \eqref{eq:l-A} that
	\begin{equation*}
    	2\kappa_0  \big\lVert (\lambda - \Aa_{j,\tau})[\pi_j^\ve f] \big\rVert_{\rmH^{r-1}}\geq 2|\lambda| \, \lVert \pi_j^\ve f \rVert_{\rmH^{r-1}} + 2 \lVert \pi_j^\ve f \rVert_{\rmH^r}
	\end{equation*}
	for all $1 \leq j \leq N$, $\tau \in [0,1]$, $\Rel \lambda \geq 1$, and $f \in \rmH^r(\Ss)$.

	By applying the reverse triangle inequality together with the above estimates, we obtain
	\begin{equation*}
	\begin{aligned}
    	2\kappa_0  \big\lVert \pi_j^\ve (\lambda - \Phi(\tau))[f] \big\rVert_{\rmH^{r-1}}&\geq 2\kappa_0 \big\lVert (\lambda - \Aa_{j,\tau})[\pi_j^\ve f] \big\rVert_{\rmH^{r-1}}- 2\kappa_0 \big\lVert \pi_j^\ve \Phi(\tau)[f] - \Aa_{j,\tau}[\pi_j^\ve f] \big\rVert_{\rmH^{r-1}} \\[0.5em]
    	&\geq 2|\lambda| \, \lVert \pi_j^\ve f \rVert_{\rmH^{r-1}} + \lVert \pi_j^\ve f \rVert_{\rmH^r} - 2\kappa_0 K \lVert f \rVert_{\rmH^{r'}}.
	\end{aligned}
	\end{equation*}
	Summing over~$j$, and using the interpolation property~\eqref{eq:interpolation}, the equivalence of norms~\eqref{eq:norm_pi_jp}, and Young’s inequality, we find constants ${\kappa \geq 1}$ and ${\omega > 1}$ such that
	\begin{equation}\label{eq:l-Phi}
    	\kappa\big\lVert (\lambda - \Phi(\tau))[f] \big\rVert_{\rmH^{r-1}} \geq|\lambda| \, \lVert f \rVert_{\rmH^{r-1}} + \lVert f \rVert_{\rmH^r}
	\end{equation}
	for all $\tau \in [0,1]$, $\Rel \lambda \geq \omega$, and $f \in \rmH^r(\Ss)$.

	Moreover, by \eqref{eq:Phi0} and~\eqref{eq:l-A_iso}, the operator ${\omega - \Phi(0)}$ is an isomorphism. The method of continuity, cf.~\cite[Proposition~I.1.1.1]{Amann.1995} and~\eqref{eq:l-Phi} 
	imply now that ${\omega - \Phi(1) = \omega - \partial \Psi(f_0)}$ is an  isomorphism too. This property, together with~\eqref{eq:l-Phi} (with $\tau=1$), shows that $\partial \Psi(f_0)$ generates a strongly continuous analytic semigroup; see~\cite[Section~I.1.2]{Amann.1995}. Thus, the proof is complete.
	\end{proof}

	We are now in a position to prove Theorem~\ref{Thm:Main}.
	\begin{proof}[Proof of Theorem~\ref{Thm:Main}]
		Using the abstract parabolic theory from~\cite[Chapter~8]{Lunardi.1995}, the theorem follows from~\eqref{eq:Psi_smooth} and Proposition~\ref{Prop:Psi_Gen}. 
		The details are analogous to the proof in the non-periodic case; see~\cite[Theorem~1.1]{MP2022}, and we therefore do not present them here.
	\end{proof}

	\section{Stability analysis}\label{Sec:6}	
	In this section, we identify all equilibrium solutions to the two-phase Stokes problem~\eqref{eq:STOKES} and study their stability properties. 
	From Theorem~\ref{Thm:Main}~(ii), we deduce that any stationary solution~$f$ to~\eqref{eq:STOKES} belongs to $\rmC^\infty(\Ss)$. Moreover,  Lemma~\ref{Lem:equi_sol} below identifies the  equilibrium solutions to~\eqref{eq:STOKES} as the functions $f \in \rmC^\infty(\Ss)$ that solve the  nonlinear  differential equation~\eqref{eq:f_stab}. It is worth noting that equation~\eqref{eq:f_stab} characterizes the equilibrium solutions also in the context of the two-dimensional Muskat problem; see~\cite{Escher.2011, Ehrnstrom.2013, Sa23}. Interestingly, there exists a one-to-one correspondence between even solutions to~\eqref{eq:f_stab} with zero integral mean and odd solutions to the mathematical pendulum equation; see~\cite{Ehrnstrom.2013}.
	\begin{Lemma}\label{Lem:equi_sol}
	A function $f\in\rmC^\infty(\Ss)$ is an equilibrium solution  to~\eqref{eq:STOKES} if and only if
	\begin{equation}\label{eq:f_stab}
		\Theta (f-\langle f\rangle)-\sigma\kappa(f)=0.
	\end{equation}
	\end{Lemma}

	\begin{proof}
	Let $f\in\rmC^\infty(\Ss)$ be an equilibrium solution to~\eqref{eq:STOKES} and let $v$ and $q$ be the associated velocity and pressure fields. By \eqref{eq:Stokes}$_6$ we then have $v\cdot \wt \nu=0$ on $\Gamma$. With $u_n$, $n>\lVert f\rVert_\infty$, defined in the proof of Proposition \ref{Prop:Ftp_0},  we then compute, using \eqref{eq:Stokes}$_{3-4}$,
	\begin{equation*}
		\int_{\Ss\times\RR} \vdiv \big(u_n(x_2)T_\mu(v,q)(x)v(x)\big)\,\rmd{x}=-\int_\Gamma v\cdot\big[T_\mu(v,q)\big]\wt\nu\,\rmd{\sigma}=-\int_\Gamma (\Theta x_2-\sigma\tkappa) v\cdot \wt\nu\,\rmd{\sigma}=0.
	\end{equation*}
	Consequently, since $2\vdiv\big(T_\mu(v,q)v\big)=\big|\nabla v+(\nabla v)^\top\big|^2$ in $(\Ss\times\RR)\setminus\Gamma$, by \eqref{eq:Stokes}$_{1-2}$, we infer from~\eqref{eq:Stokes}$_{5}$, \eqref{eq:limits}, and~\eqref{eq:T_lim} that 
	\begin{equation*}
		\frac{1}{2}\int_{\{|x_2|<n\}}\big|\nabla v+(\nabla v)^\top\big|^2\,\rmd{x}\leq  -\int_{\{n<|x_2|<n+1\}} \nabla u_n(x_2)\cdot\big(T_\mu(v,q)(x)v(x)\big)\,\rmd{x}\underset{n\to\infty}\longrightarrow 0.
	\end{equation*}
	Thus, $\nabla v+(\nabla v)^\top=0$ in $\rmL^2(\Ss\times\RR, \RR^{2\times2})$ and, arguing similarly as in the final part of the proof of Proposition~\ref{Prop:Ftp_0},
	 we obtain that $\nabla v=0$ in $\rmL^2(\Ss\times\RR,\RR^{2\times2})$. Using  $(v,q)\in X_f$, \eqref{eq:Stokes}$_{1}$, and \eqref{eq:Stokes}$_{5}$, 
	 it immediately follows that $v=0$ in $\Ss\times\RR$ and $q^\pm=\mp\Theta\langle f\rangle/2$ in $\Omega^\pm$, 
and the relation~\eqref{eq:f_stab} is simply a reformulation of~\eqref{eq:Stokes}$_{4}$.

	For the converse implication, it suffices to show that if $f \in \rmC^\infty(\Ss)$ is a solution to~\eqref{eq:f_stab} with zero integral mean, then $(f, v^\pm, q^\pm,c_\Gamma) = (f, 0, 0,0)$ is an equilibrium solution to~\eqref{eq:STOKES}; see the discussion  preceding~\eqref{eq:sol_trans}.  
	This is, however, a straightforward consequence of the uniqueness result in Theorem~\ref{Thm:Ftp}, and the proof is complete.
	\end{proof}

	A complete description of the solution set to~\eqref{eq:f_stab} has been provided in~\cite{Ehrnstrom.2013,Sa23} in the context of the Muskat problem. It is shown there, in particular, that any solution to~\eqref{eq:f_stab} is a horizontal and/or vertical translation of a (distributional) solution~${f \in \rmhH_e^r(\Ss)}$, with $r \in (3/2,2)$, to 
	\begin{equation}\label{eq:Stab}
		\bigg( \frac{f'}{(1+f'^2)^{1/2}} \bigg)' + \lambda f = 0, \qquad \text{where}~ \lambda := -\frac{\Theta}{\sigma}=\frac{g[\rho]}{\sigma},
	\end{equation}
	which we view as an equation in  ${\rmhH_e^{r-2}(\Ss)}$. We define, for $s \in \RR$,
	\begin{equation*}
		\rmhH_e^s(\Ss) := \big\{f \in \rmH^s(\Ss) \,:\, \langle f,1\rangle=0 \text{ and }\big\langle f,e^{ik\cdot}\big\rangle = \big\langle f,e^{-ik\cdot}\big\rangle \text{ for all } k \in\ZZ\big\},
	\end{equation*}
	where $\langle\cdot ,\cdot \rangle $ is the canonical duality pairing between $\mcD'(\Ss)$ and $\mcD(\Ss)=\rmC^\infty(\Ss)$. 
	Indeed, a simple bootstrap argument shows that any solution ${f \in \rmhH_e^r(\Ss)}$, with $r \in (3/2,2)$, to~\eqref{eq:Stab} actually belongs to~$\rmC^\infty(\Ss)$. 
	Moreover, testing the equation for $\lambda \leq 0$ against $f$ yields $f=0$  as the unique solution in this case. 
	However, if $\lambda > 0$, which corresponds to the scenario where the denser fluid lies above the less dense one, equation~\eqref{eq:Stab} may also admit finger-shaped solutions. These are nontrivial periodic solutions to~\eqref{eq:STOKES}, where the interface between the fluids forms repeating upward and downward fingers.

	In Theorem~\ref{Thm:SolBif}, we summarize  results from~\cite{Ehrnstrom.2013}, adapted to our functional-analytic setting, concerning the solution set of~\eqref{eq:Stab} in ${\rmhH_e^{r}(\Ss)}$, regarded now as a bifurcation problem with bifurcation parameter~$\lambda$ and solution~$(\lambda,f)\in\RR\times{\rmhH_e^{r}(\Ss)}$. Theorem~\ref{Thm:SolBif} shows that the solutions to \eqref{eq:Stab} lie on global bifurcation branches $\mcC_\ell$,  $\ell\in\NN$. Each branch $\mcC_\ell$ consists of equilibrium solutions of minimal period $2\pi/\ell$ and intersects the trivial branch of solutions only once. 
	Along each branch, the flat equilibrium deforms into a cosine-shaped function and the slope at $\xi=\pi/(2\ell)$ blows-up monotonically as one approaches the end of the branch.

	\begin{Theorem}\label{Thm:SolBif}
	Let
	\begin{equation}\label{eq:lambdastar}
		\lambda_*:=\frac{1}{2\pi^2}B\bigg(\frac{3}{4},\frac{1}{2}\bigg)^2 \approx 0.2909,
	\end{equation}
	where $B$ is the beta function. Then, the following holds true:  
	\begin{enumerate}[label=\textup{(\roman*)}]
	\item 
		If $\lambda\leq \lambda_*$, then  equation \eqref{eq:Stab} only has the trivial solution $(\lambda,f)=(\lambda,0)$.
	\item 
		Let $\lambda>\lambda_*$.
		\begin{enumerate}[label=\textup{(\alph*)}]
		\item 				
			Equation~\eqref{eq:Stab} admits solutions $(\lambda,f) \in \RR \times \rmhH_e^r(\Ss)$ of minimal period~$2\pi$ if and only if~$\lambda_* < \lambda < 1$. More precisely, for each~$\lambda \in (\lambda_*,1)$, there exist exactly two even solutions~$(\lambda, \pm f_\lambda)$ to~\eqref{eq:Stab} of minimal period~$2\pi$. Moreover, we have~$|f_{\lambda_1}| \leq |f_{\lambda_2}|$ for~$\lambda_2 < \lambda_1$, and~$\lVert f_\lambda \rVert_\infty \to 0$ as~$\lambda \to 1$, while
			\begin{equation*}
				\lVert f_\lambda \rVert_\infty = |f_\lambda(0)| \to \sqrt{2/\lambda_*},\quad \lVert f_\lambda' \rVert_\infty = |f_\lambda'(\pi/2)| \to \infty \qquad \text{as}~\lambda \to \lambda_*.
			\end{equation*}
		\item 
			Equation~\eqref{eq:Stab} admits solutions $(\lambda,f) \in \RR \times \rmhH_e^{r}(\Ss)$ of minimal period~$2\pi/\ell$, $2\leq\ell\in\NN$, if and only if $\ell^2 \lambda_*<\lambda<\ell^2$. 
			To be precise, for every $\lambda\in(\ell^2\lambda_*,\ell^2)$, there exist exactly two even solutions $(\lambda,\pm f_\lambda)$ of minimal period $2\pi/\ell$ given by
			\begin{equation*}
				f_\lambda=\ell^{-1}f_{\lambda\ell^{-2}}(\ell\cdot),
			\end{equation*}
			where $f_{\lambda\ell^{-2}}(\ell\cdot)$ is the function identified at \textup{(a)}.
		\end{enumerate}
	\item 
		If we consider \eqref{eq:Stab} as an abstract bifurcation problem in $\RR\times\rmhH^{r}_e(\Ss)$, the global bifurcation branch $\mcC_\ell$ arising at $(\ell^2,0)$, $\ell\in\NN$, described in \textup{(ii)}, admits in a small neighborhood of~${(\ell^2,0)}$ a smooth  parametrization
		\begin{equation*}
			(\lambda_\ell,f_\ell)\in\rmC^\infty\big((-\delta_\ell,\delta_\ell),\RR\times \rmhH^{r}_e(\Ss)\big), \qquad \delta_\ell>0,
		\end{equation*}
		satisfying
		\begin{equation}\label{eq:lf_exp}
			\left\{
			\begin{array}{lll}
				\lambda_\ell(s):=\ell^2-\dfrac{3\ell^4}{8}s^2+\mcO(s^4)\quad\text{in}~\RR,\\[1ex]
				f_\ell(s):= s\cos(\ell\cdot)+\mcO(s^2)\quad\text{in}~\rmhH^{r}_e(\Ss),
			\end{array}\right.
			\qquad\text{for}~s\to 0.
		\end{equation}
	\end{enumerate}
	\end{Theorem}

	\begin{proof}
	The claims (i) and (ii) are established in~\cite{Ehrnstrom.2013}. The proof of~(iii) follows the strategy pursued in \cite{Escher.2011}, without differentiating~\eqref{eq:Stab}. To this end, recalling~\eqref{eq:nutauomega}, we define $F:\RR\times\rmhH^{r}_e(\Ss)\to\rmhH^{r-2}_e(\Ss)$ by
	\begin{equation}\label{eq:F_Bif}
		F(\lambda,f)=\big( f'\big(\omega(f))^{-1} \big)'+\lambda f,
	\end{equation}
	and reformulate~\eqref{eq:Stab} as the bifurcation problem
	\begin{equation}\label{eq:F=0}
		F(\lambda,f)=0.
	\end{equation}
	Due to~\eqref{eq:defphi} and \eqref{eq:phi_smooth}, we have  $F\in\rmC^\infty\big(\RR\times\rmhH^{r}_e(\Ss),\rmhH^{r-2}_e(\Ss)\big)$ and moreover $F(\lambda,0)=0$ for all~${\lambda\in\RR}$. Representing functions in $\rmhH^{r}_e(\Ss)$ by their cosine series $\sum_{k=1}^\infty a_k \cos(k\cdot)$, it is easy to compute  that the  Fr\'echet derivative~$\partial_f F(\lambda,0)$ is  the Fourier multiplier 
	\begin{equation*}
		\partial_f F(\lambda,0)\Bigg[\sum_{k=1}^\infty a_k\cos(k\cdot)\Bigg]= \sum_{k=1}^\infty(\lambda-k^2)a_k\cos(k\cdot).
	\end{equation*}
	It is now obvious that the assumptions of the theorem on bifurcation from simple eigenvalues due to Crandall and Rabinowitz \cite{Crandall.1971} are fulfilled in $(\ell^2,0)$, $\ell\in\NN$. Consequently, for each $\ell \in \NN$, there exists a smooth local bifurcation curve~${(\lambda_\ell,f_\ell):(-\delta_\ell,\delta_\ell) \to \RR \times \rmhH^{r}_e(\Ss)}$ such that, for $s \to 0$, 
	\begin{equation}\label{eq:lf_exp_CR}
		\left\{
		\begin{array}{lll}
			\lambda_\ell(s) := \ell^2 + \mcO(s) \quad \text{in}~\RR, \\[1ex]
			f_\ell(s) := s\cos(\ell\cdot) + \tau_\ell(s) \quad \text{in}~\rmhH^{r}_e(\Ss),
		\end{array}\right.
	\end{equation}
	with $\tau_\ell \in \rmC^\infty\big((-\delta_\ell,\delta_\ell),\rmhH^{r}_e(\Ss)\big)$ satisfying $\tau_\ell(0)=\tau_\ell'(0)=0$ and $\langle\tau_\ell(s),\cos(\ell\cdot)\rangle=0$ for all $|s|< \delta_\ell$, consisting only of equilibrium solutions to~\eqref{eq:STOKES}. Moreover, in a small neighborhood of~${(\ell^2,0)}$ in~$\RR \times \rmhH^{r}_e(\Ss)$, the solutions to~\eqref{eq:Stab} are either trivial or belong to the local bifurcation curve~$(\lambda_\ell,f_\ell)$. Noticing that $F(\lambda_\ell(s), -f_\ell(s))=0$ for~$|s|<\delta_\ell$, we deduce from~\eqref{eq:lf_exp_CR} and (ii), after making~$\delta_\ell$ smaller if necessary, that~$\lambda_\ell(-s)=\lambda_\ell(s)$ and $f_\ell(-s)=-f_\ell(s)$ for all $|s|<\delta_\ell$. 
	In particular, we have~$\lambda_{ \ell}'(0)=\lambda_{\ell}'''(0)=0$. In order to show that  $\lambda_\ell''(0)=-3\ell^4/4$, and prove in this way~\eqref{eq:lf_exp}, 
	we differentiate the identity $F(\lambda_{\ell}(s),f_{\ell}(s))=0$ three times with respect to $s$ and evaluate the resulting identity at $s=0$ to arrive at 
	\begin{equation}\label{eq:ddl_est}
		\lambda_\ell''(0)f_\ell'(0)-\frac{\rmd}{\rmd\xi}\bigg(\bigg(\frac{\rmd}{\rmd\xi} f_\ell'(0)\bigg)^3\bigg)+\frac{1}{3}\bigg(\frac{\rmd^2}{\rmd\xi^2}\big(f_\ell'''(0)\big)+\ell^2 f_\ell'''(0)\bigg)=0\qquad\text{in $\rmhH^{r-2}_e(\Ss)$}.
	\end{equation}
	Since $f_\ell'''(0)=\tau_\ell'''(0)$ with $\langle\tau_\ell'''(0),\cos(\ell\cdot)\rangle=0$ and  $f_\ell'(0)=\cos(\ell\cdot)$, testing \eqref{eq:ddl_est} by $\cos(\ell\cdot)$ leads to
	\begin{equation*}
		\lambda_\ell''(0)=\frac{\displaystyle\bigg\langle\frac{\rmd}{\rmd\xi}\bigg(\bigg(\frac{\rmd}{\rmd\xi} \cos(\ell\cdot)\bigg)^3\bigg), \cos(\ell\cdot) \bigg\rangle}{\langle\cos(\ell\cdot) , \cos(\ell\cdot) \rangle}=-\frac{3\ell^4}{4},
	\end{equation*}
	which completes the proof.
	\end{proof}

	We conclude this section with the proof of Theorem~\ref{Thm:Stability}.

	\begin{proof}[Proof of Theorem~\ref{Thm:Stability}]
	To establish claim (i),  we first infer from  
\eqref{eq:ev_pb} that~${\Psi(f)\in\rmhH^{r-1}(\Ss)}$ for each $f\in\rmhH^r(\Ss)$, since the integral mean of $\langle f \rangle$ is preserved; see \eqref{eq:pres}. 
Recalling \eqref{eq:Psi_smooth}, we thus have~${\Psi\in\rmC^\infty\big(\rmhH^{r}(\Ss),\rmhH^{r-1}(\Ss)\big)}$. 
Moreover, combining \eqref{eq:B_by_Bnmpq}, \eqref{eq:D(f)}, \eqref{eq:defphi}, \eqref{eq:def_Psi_i}, \eqref{eq:Db=V}, \eqref{eq:DPsi}, and~\eqref{eq:D_Db=V}
	we compute that the Fr\'echet derivative $\partial\Psi(0)\in\mcL(\rmhH^{r}(\Ss),\rmhH^{r-1}(\Ss)$ is the Fourier multiplier
	\begin{equation*}
		\partial\Psi(0):=\frac{\Theta}{2(\mu^+ +\mu^-)}B_0(0)-\frac{\sigma}{2(\mu^+ +\mu^-)}\bigg(-\frac{\rmd^2}{\rmd\xi^2}\bigg)^\frac{1}{2},
	\end{equation*}
	with $B_0(0)= H\circ S$, where $H$ is the Hilbert transform and $S\in\mcL\big(\rmhH^{r}(\Ss)\big)$ is the operator which associates to each function $f\in \rmhH^{r}(\Ss)$  its antiderivative
	\begin{equation}
		S[f](\xi):=\int_0^\xi f(s)\,\rmd s+\frac{1}{2\pi}\int_0^{2\pi} sf(s)\,\rmd s,\qquad \xi\in\Ss,
	\end{equation}
	cf. \cite[Lemma~3.4]{Bohme.2024}. In particular, $\partial\Psi(0)\in\mcL\big(\rmhH^{r}(\Ss),\rmhH^{r-1}(\Ss)\big)$ generates a strongly continuous analytic semigroup in~$\mcL\big(\rmhH^{r-1}(\Ss)\big)$, 
	cf. \eqref{eq:l-A_iso}--\eqref{eq:l-A},     and its spectrum is given by 
	\begin{equation}\label{eq:Psi_spec}
		\sigma(\partial\Psi(0))=\bigg\{-\frac{\Theta+\sigma k^2}{2(\mu^+ +\mu^-)k} :k\in\NN\bigg\},
	\end{equation} 
	consisting only of eigenvalues with finite multiplicity.
	
	Hence, assuming $\sigma+\Theta>0$, we have
	\begin{equation*}
		\sup\{\Rel\lambda:\lambda\in\sigma(\partial\Psi(0))\}\leq -\vartheta_0<0, 
	\end{equation*}
	and  we may apply the principle of linearized stability \cite[Theorem~9.1.2]{Lunardi.1995} in the context of \eqref{eq:ev_pb} (regarded as an evolution problem in $\rmhH^{r-1}(\Ss)$) to obtain (a).
	
	Conversely, if $\sigma+\Theta<0$, the spectrum of $\partial\Psi(0)\in\mcL\big(\rmhH^{r}(\Ss),\rmhH^{r-1}(\Ss)\big)$ contains a finite number of positive eigenvalues, more precisely, we have
	\begin{equation*}
		\left\{
		\begin{array}{lll}
			-\cfrac{\sigma +\Theta}{2(\mu^{+}+\mu^{-})}\in\sigma_+(\partial\Psi(0)):=\sigma(\partial\Psi(0))\cap \{\lambda\in\CC\,:\,\Rel\lambda>0\},\\[2ex]
			\inf\{\Rel\lambda\,:\, \lambda\in \sigma_+(\partial\Psi(0))\}>0.
		\end{array}
		\right.
	\end{equation*}
	Claim (b)  now follows directly from an application of~\cite[Theorem~9.1.3]{Lunardi.1995} in the context of~\eqref{eq:ev_pb} (regarded again as an evolution problem in $\rmhH^{r-1}(\Ss)$).
	
	In order to establish (ii), we first formulate \eqref{eq:ev_pb} as an evolution problem in $\rmhH_e^{r-1}(\Ss)$ depending on the parameter $\lambda$ introduced in~\eqref{eq:Stab} (by eliminating the parameter $\Theta$).  Recalling~\eqref{eq:defV} and \eqref{eq:Db=V}, for each  $f\in \rmhH_e^{r}(\Ss)$ and $\lambda\in\RR$ we denote by $\beta=\beta(f,\lambda)\in\rmH^{r-1}(\Ss)^2$ the unique solution to 
	\begin{equation}\label{eq:Db=Vi}
 		\big(1+2 a_\mu\DD(f)\big)[\beta]= -\frac{\sigma}{4}
 		\begin{pmatrix}
 			\mcV_1(f)-\lambda\mcV_3(f)\\
 			\mcV_2(f)+\lambda\mcV_4(f)+\lambda\ln 4\langle f\rangle
 		\end{pmatrix}.
 	\end{equation}
	Note that the solution $\beta(f)$ defined in \eqref{eq:Db=V} coincides with $\beta(f,\lambda)$ when $\lambda=-\Theta/\sigma$. Related to~\eqref{eq:ev_pb}, we consider the evolution problem 
	\begin{equation}\label{eq:ev_pbl}
		\frac{\rmd f}{\rmd t}(t)=\Psi(f(t),\lambda),\quad t\geq 0,\qquad f(0)=f_0,
	\end{equation}
	where
	\begin{equation}\label{eq:def_Psil}
		\Psi(f,\lambda):=\frac{2}{\mu^++\mu^-}(-f',1)^\top\cdot\beta(f,\lambda).
	\end{equation}
	The mappings defined  in \eqref{eq:Db=Vi} and \eqref{eq:def_Psil} depend linearly on $\lambda$, hence~${\Psi\in \rmC^\infty\big(\rmhH^{r}_\e(\Ss)\times\RR,\rmhH^{r-1}(\Ss)\big)}$. 
	To show that $\Psi(f,\lambda)$ is  also an even function, we introduce the operator $[\varphi\mapsto \check\varphi]\in\mcL\big(\rmL^2(\Ss)\big)$ by setting $\check\varphi(\xi):=\varphi(-\xi)$, $\xi\in\Ss$, and  observe that $(B_{n,m}^{p,q}(f)[\varphi])\check{\phantom{.}}=(-1)^{n+p+q+1}B_{n,m}^{p,q}(f)[\check\varphi]$. Since
	\begin{equation*}
		\phi_i(f)\check{\phantom{.}}=(-1)^{i+1}\phi_i(f), \quad i=1,\, 2,
	\end{equation*}
	respectively 
	\begin{equation*}
 		(B_i(f)[\varphi])\check{\phantom{.}}=-B_i(f)[\check\varphi], \quad i\in\{1,\, 4,\, 5\},\qquad\text{and}\qquad (B_i(f)[\varphi])\check{\phantom{.}}= B_i(f)[\check\varphi], \quad i\in\{0,\, 2,\, 3,\, 6\},
	\end{equation*} 
	by~\eqref{eq:B_by_Bnmpq}--\eqref{eq:B0} and \eqref{eq:defphi}, we infer from  \eqref{eq:def_Psi_i} that
	\begin{equation}\label{eq:Vi_check}
		\mcV_i(f)\check{\phantom{.}}=(-1)^i\mcV_i(f),\qquad 1\leq i\leq 4.
	\end{equation}
	Recalling~\eqref{eq:D(f)}, we deduce from~\eqref{eq:Db=Vi}, by using~\eqref{eq:Vi_check}, that  
	\begin{equation*}
 		\big(1+2 a_\mu\DD(f)\big)
 		\begin{bmatrix}
			-\check\beta_1\\
			\check\beta_2
		\end{bmatrix}
		= -\frac{\sigma}{4}
		\begin{pmatrix}
	 		\mcV_1(f)-\lambda\mcV_3(f)\\
 			\mcV_2(f)+\lambda\mcV_4(f)+\lambda\ln 4\langle f\rangle
 		\end{pmatrix},
 	\end{equation*}
	and Theorem~\ref{Thm:D(f)_Res_Hr-1} ensures that 
	\begin{equation}
		\begin{bmatrix}
			\beta_1(f,\lambda)\check{\phantom{.}}\\
			\beta_2(f,\lambda)\check{\phantom{.}}
		\end{bmatrix}
		=
		\begin{bmatrix}
			-\beta_1(f,\lambda)\\
			\beta_2(f,\lambda)
		\end{bmatrix},
		\qquad  f\in \rmhH_e^{r}(\Ss),\quad \lambda\in\RR.
	\end{equation}
	It is straightforward to infer from this relation together  with \eqref{eq:def_Psil} that 
	\begin{equation}\label{eq:Psi_reg}
		\Psi\in \rmC^\infty\big(\rmhH^{r}_\e(\Ss)\times\RR,\rmhH^{r-1}_e(\Ss)\big).
	\end{equation}
	Moreover, the Fr\'echet derivative $\partial_f\Psi(0,\lambda)$ generates a strongly continuous analytic semigroup in~$\mcL\big(\rmhH^{r-1}_e(\Ss)\big)$, being the Fourier multiplier
	\begin{equation*}
		\partial_f\Psi(0,\lambda)\bigg[\sum_{k=1}^\infty a_k\cos(k\cdot)\bigg]= \sum_{k=1}^\infty\frac{\sigma(\lambda- k^2)}{2(\mu^+ +\mu^-)k}a_k\cos(k\cdot).
	\end{equation*}
	
	We now study the stability properties of the equilibrium solution $f_\ell(s)$, $\ell\in\NN$ and~$|s|<\delta_\ell$, to~\eqref{eq:ev_pbl} (for $\lambda=\lambda_\ell(s)$). 
	To this end, we observe that if $\delta_\ell$ is sufficiently small, the perturbation result \cite[Theorem~I.1.3.1~(i)]{Amann.1995} ensures that $\partial_f\Psi(f_\ell(s),\lambda_\ell(s))$ generates a strongly continuous analytic semigroup in $\rmhH_e^{r-1}(\Ss)$ for all $|s|<\delta_\ell$. Moreover, since the embedding $\rmhH_e^{r}(\Ss) \hookrightarrow \rmhH_e^{r-1}(\Ss)$ is compact, we infer from \cite[Theorem~III.6.29]{Kato.1995} and the generator property of $\partial_f\Psi(f_\ell(s),\lambda_\ell(s))$ that the spectrum~$\sigma(\partial_f\Psi(f_\ell(s),\lambda_\ell(s)))$ consists entirely of isolated eigenvalues with finite algebraic multiplicities which do not accumulate at~$0$.

	Assume now that $\ell\geq 2$. Then, $\partial_f\Psi(0,\ell^2)$ has exactly $(\ell-1)$ positive simple eigenvalues. Since a finite set of simple eigenvalues of $\partial_f\Psi(f_\ell(s),\lambda_\ell(s))$ depends continuously on~$s$ as $s$ varies in $(-\delta_\ell,\delta_\ell)$, cf.~\cite[Chapter~IV.3.5]{Kato.1995}, we conclude that
	\begin{equation}\label{eq:pos_ev}
		\sigma_+(\partial_f\Psi(f_\ell(s),\lambda_\ell(s)))\neq\emptyset \qquad \text{and} \qquad \inf\{\Rel\lambda\,:\, \lambda\in \sigma_+(\partial_f\Psi(f_\ell(s),\lambda_\ell(s)))\}>0
	\end{equation}
for all $|s|<\delta_\ell$, provided that $\delta_\ell$ is sufficiently small. Therefore, we can apply \cite[Theorem~9.1.3]{Lunardi.1995} to deduce that $f_\ell(s)$ is an unstable equilibrium.
 
	Finally, we consider the more intricate case $\ell=1$, in which the eigenvalues of $\partial_f\Psi(0,1)$ are all negative, except for the simple eigenvalue zero. Let
	\begin{equation*}
		\gamma(\lambda):=\frac{\sigma(\lambda-1)}{2(\mu^+ + \mu^-)}
	\end{equation*}
	denote the eigenvalue of $\partial_f\Psi(0,\lambda)$ that satisfies $\gamma(\lambda)=0$ for $\lambda=1$. Assuming that $\delta_1$ is small, the operator $\partial_f\Psi(f_1(s),\lambda_1(s))$ has, for each $|s|<\delta_1$, a simple eigenvalue $z(s)$ in the vicinity of~$0$. 
	 According to the principle of exchange of stability due to Crandall and Rabinowitz, cf.~\cite[Theorem~1.16]{Crandall.1973}, it holds that
	\begin{equation*}
		\lim_{s\to 0} \frac{-s\lambda_1'(s)\gamma'(0)}{z(s)} = 1.
	\end{equation*}
	Since $s\lambda_1'(s)<0$ for all $0<|s|<\delta_1$ by~\eqref{eq:lf_exp} and $\gamma'(0)>0$, we deduce that $z(s)$ is a positive simple eigenvalue of $\partial_f\Psi(f_1(s),\lambda_1(s))$ for each $0<|s|<\delta_1$. Hence, \eqref{eq:pos_ev} also holds  true for~$\ell=1$ and~$0<|s|<\delta_1$, and therefore $f_1(s)$ is an unstable equilibrium solution to~\eqref{eq:ev_pbl} by the same abstract result~\cite[Theorem~9.1.3]{Lunardi.1995}.
	\end{proof}

 	\appendix 	
 	
	\section{Analysis of (singular) integral operators}\label{Sec:A}
	In this section, we prove mapping properties of the operators $B_{n,m}^{p,q}$, defined in~\eqref{eq:Bnmpq},  which were used in the analysis of the double-layer potential operator $\DD$ in Section~\ref{Sec:4}. For Lipschitz continuous mappings~$\bfa=(a_1,\dots,a_m):\RR\to\RR^{m},\,\bfb=(b_1,\dots,b_n):\RR\to\RR^{n}$, and~${\bfc=(c_1,\dots,c_q):\RR\to\RR^{q}}$, where~${m,\,n,\,q\in\NN_0}$ and $\ell\in\{1,2\}$, we define the integral operator
	\begin{equation}\label{eq:Anmlq}
	\begin{aligned}
		&A_{n,m}^{\ell,q}(\bfa\vert \bfb)[\bfc,\varphi](\xi)\\
		&:=\frac{1}{2\pi}\int_{-\pi}^{\pi}\left[\frac{\prod\limits_{i=1}^{n}\frac{\T{\xi,s}b_i}{\tss}\prod\limits_{i=1}^{q}\frac{\dg{\xi,s}{c_i}/2}{\tss}}{\prod\limits_{i=1}^{m}\bigg[1+\Big(\frac{\T{\xi,s}a_i}{\tss}\Big)^2\bigg]}\frac{1}{\tss^{\ell}}-\frac{\prod\limits_{i=1}^{n}\frac{\dg{\xi,s}{b_i}/2}{s/2}\prod\limits_{i=1}^{q}\frac{\dg{\xi,s}{c_i}/2}{s/2}}{\prod\limits_{i=1}^{m}\bigg[1+\Big(\frac{\dg{\xi,s}{a_i}/2}{s/2}\Big)^2\bigg]}\frac{1}{(s/2)^\ell}\right]\varphi(\xi-s)\rmd s,
	\end{aligned}
	\end{equation}
	where $\varphi\in\rmL^2(\Ss)$ and $\xi\in\RR$ (see \eqref{eq:notation1} and \eqref{eq:notation2}). 
	As noted in  \cite[Equation~(A.2)]{Bohme.2024}, we have
	\begin{equation}\label{eq:B=A+C}
		B_{n,m}^{0,q}(\bfa\vert \bfb)[\bfc,\varphi]=A_{n,m}^{1,q}(\bfa\vert \bfb)[\bfc,\varphi]+C_{n+q,m}(\bfa)[(\bfb,\bfc),\varphi],\qquad m,\,n,\,q\in\NN_0.
	\end{equation}
	Moreover, $A_{n,m}^{\ell,q}(\bfa\vert \bfb)[\bfc,\varphi]$ is $2\pi$-periodic if $\bfa,\, \bfb$, and $\bfc$ have this property. Let us also recall the following identity from \cite[Lemma~A.5]{Bohme.2024}
	\begin{equation}\label{eq:diffB}
	\begin{aligned}
		&B_{n,m}^{p,q}(\bfa\vert\bfb)[\bfc,\varphi]-B_{n,m}^{p,q}(\tilde\bfa\vert\tilde\bfb)[\tilde\bfc,\varphi]\\
		&= \sum_{i=1}^{q}B_{n,m}^{p,q}(\bfa\vert\bfb)[\tilde{c}_1,\dots,\tilde{c}_{i-1},c_i-\tilde{c}_i,c_{i+1},\dots,c_q,\varphi]\\
		&\qquad+\sum_{i=1}^{n}\big(B_{n,m}^{p,q}(\bfa\vert \tilde{b}_1,\dots,\tilde{b}_{i-1},b_i,\dots,b_n)-B_{n,m}^{p,q}(\bfa\vert \tilde{b}_1,\dots,\tilde{b}_i,b_{i+1},\dots,b_n)\big)[\tilde\bfc,\varphi]\\
		&\qquad+\sum_{i=1}^{m}\big(B_{n+2,m+1}^{p,q}(\tilde{a}_1,\dots,\tilde{a}_i,a_i,\dots,a_m\vert\tilde\bfb,\tilde{a}_i,\tilde{a}_i)\\
		&\hspace{5em}-B_{n+2,m+1}^{p,q}(\tilde{a}_1,\dots,\tilde{a}_i,a_i,\dots,a_m\vert\tilde\bfb,\tilde{a}_i,a_i)\big)[\tilde\bfc,\varphi]\\
		&\qquad+\sum_{i=1}^{m}\big(B_{n+2,m+1}^{p,q}(\tilde{a}_1,\dots,\tilde{a}_i,a_i,\dots,a_m\vert\tilde\bfb,\tilde{a}_i,a_i)\\
		&\hspace{5em}-B_{n+2,m+1}^{p,q}(\tilde{a}_1,\dots,\tilde{a}_i,a_i,\dots,a_m\vert\tilde\bfb,a_i,a_i)\big)[\tilde\bfc,\varphi],
	\end{aligned}
	\end{equation}	\medskip
	where $(\bfa,\bfb,\bfc),\,(\tilde{\bfa},\tilde{\bfb},\tilde{\bfc})\in\rmW^{1,\infty}(\Ss)^{m+n+q}$. 
		
	We begin by proving Lemma~\ref{Lem:Bnmpq_L2_Hr-1}~ (ii)--(iii), which is used several times in our analysis.
	
	\begin{proof}[Proof of Lemma~\ref{Lem:Bnmpq_L2_Hr-1}~{\rm (ii)--(iii)}]
	We start by recalling some elementary inequalities
	\begin{equation}\label{eq:tan_h}
	\begin{aligned}
		|\tanh(\xi)|\leq |\xi| \qquad &\text{and}\qquad |\tanh(\xi)-\xi|\leq |\xi\tanh^2(\xi)|,\\
		|\zeta|\leq |\tan(\zeta)| \qquad &\text{and}\qquad |\tan(\zeta)-\zeta|\leq |\zeta^2\tan(\zeta)|,
	\end{aligned}
	\end{equation}
	and
	\begin{equation}\label{eq:deriv}
		\bigg|\frac{\Tf{\xi,s}{d}}{\tss}\bigg|\leq\bigg|\frac{\dg{\xi,s}{d/2}}{\tss}\bigg| \leq \bigg|\frac{\dg{\xi,s}{d}}{s}\bigg|\bigg|\frac {s/2}{\tss}\bigg|\leq \lVert d'\rVert_\infty\bigg|\frac {s/2}{\tss}\bigg|\leq \lVert d'\rVert_\infty,\qquad d\in \rmW^{1,\infty}(\Ss),
	\end{equation}
	which hold for all $\xi\in\RR$, $\zeta\in(-\pi/2,\pi/2)$, and $0\neq s \in(-\pi,\pi)$. These inequalities immediately yield for all $0\neq s\in(-\pi,\pi)$ and $\xi\in\RR$ that
	\begin{equation}\label{eq:T-d}
	\begin{aligned}
		&\bigg|\frac{1}{\tss^\ell}-\frac{1}{(s/2)^\ell}\bigg|\leq 2|s|^{2-\ell},\qquad\ell\in\{1,\, 2\},\\
		&\bigg|\frac{\Tf{\xi,s}{d}}{\tss}-\frac{\dg{\xi,s}{d}}{s}\bigg|+\bigg|\frac{\dg{\xi,s}{d}/2}{\tss}-\frac{\dg{\xi,s}{d}}{s}\bigg|\leq C|\dg{\xi,s}{d}|\,|s|\leq C|s|^2,
	\end{aligned}
	\end{equation}
	with $C$ depending only on $\lVert d'\rVert_\infty$. We also note the following inequalities
	\begin{equation}\label{eq:d'2}
		\frac{|\Tf{\xi,s}{d}|}{|s/2|^{1/2}}\leq \frac{|\dg{\xi,s}{d}|}{|s|^{1/2}}\leq\lVert d'\rVert_2,\qquad d\in\rmH^1(\Ss),\quad 0\neq s\in(-\pi,\pi),\quad \xi\in\RR,  
	\end{equation}
	and
	\begin{equation}\label{eq:x-tanh_Lip}
		|\tanh(\xi)-\tanh(\zeta)|\leq|\xi-\zeta| \qquad\text{and}\qquad	|(\xi-\tanh(\xi))-(\zeta-\tanh(\zeta))|\leq(\xi^2+\zeta^2)|\xi-\zeta|
	\end{equation}
	for  $\xi,\,\zeta\in\RR$, which follow from \eqref{eq:tan_h} via the fundamental theorem of calculus.\medskip
		
	\noindent{\em Claim {\rm (ii)}.} If $p\geq 1$,  the estimate~\eqref{eq:Bnmpq_L2_Hr-1_c} is derived by using~\eqref{eq:deriv} and~\eqref{eq:d'2} as follows
	\begin{equation*}
	\begin{aligned}
		\lVert B_{n,m}^{p,q}(\bfa\vert\bfb)[\bfc,\varphi]\rVert_\infty&\leq C\lVert\varphi\rVert_\infty\bigg(\prod_{i=2}^q \lVert c_i'\rVert_\infty \bigg)\sup_{\xi\in\Ss}\int_{-\pi}^\pi \bigg|\frac{s/2}{\tss}\bigg|^{n+q+1-p}|s/2|^{p-3/2}\frac{|\dg{\xi,s}{c_1}|}{|s|^{1/2}}\,\rmd s\\
		&\leq C\lVert c_1'\rVert_2 \lVert\varphi\rVert_\infty \bigg(\prod_{i=2}^q \lVert c_i'\rVert_\infty \bigg)\int_{-\pi}^\pi |s/2|^{p-3/2}\,\rmd s\\
		&\leq C\lVert c_1'\rVert_2 \lVert\varphi\rVert_{\rmH^{r-1}} \prod_{i=2}^q \lVert c_i\rVert_{\rmH^r}.
	\end{aligned}
	\end{equation*}	
	If $p=0$, we use \eqref{eq:B=A+C} and then \cite[Lemma~A.1~(ii)]{Bohme.2024} to estimate the term $C_{n+q,m}(\bfa)[(\bfb,\bfc),\varphi]$ by the right side of~\eqref{eq:Bnmpq_L2_Hr-1_c}. We are thus left with the term $A_{n,m}^{1,q}(\bfa\vert\bfb)[\bfc,\varphi]$. To this end, we define the locally Lipschitz continuous function $F:\RR^{n+q+m}\to\RR$ by
	\begin{equation*}
		F(x,y,z)=\frac{1}{2\pi}\bigg(\prod\limits_{i=1}^{n}x_i\bigg)\bigg(\prod\limits_{i=1}^{q}y_i\bigg)\bigg(\prod\limits_{i=1}^{m}(1+z_i^2)^{-1}\bigg) \qquad \text{for}~(x,y,z)\in\RR^{n+q+m}.
	\end{equation*}
	Using a telescoping sum argument together with \eqref{eq:T-d}--\eqref{eq:d'2}, it is easy to see that that the kernel of~$A_{n,m}^{1,q}(\bfa\vert\bfb)[\bfc,\cdot]$ satisfies
	\begin{equation}\label{eq:FT-Fd}
		\bigg|F\bigg(\frac{\Tf{\xi,s}{\bfb}}{\tss},\frac{\dg{\xi,s}{(\bfc/2)}}{\tss},\frac{\Tf{\xi,s}{\bfa}}{\tss}\bigg)\frac{1}{\tss} - F\bigg(\frac{\dg{\xi,s}{\bfb}}{s},\frac{\dg{\xi,s}{\bfc}}{s},\frac{\dg{\xi,s}{\bfa}}{s}\bigg)\frac{1}{s/2}\bigg|\leq C|\dg{\xi,s}{c_1}|\prod_{\substack{i=2}}^q\lVert c_i'\rVert_\infty,
	\end{equation}
	where $C$ depends only on $\lVert(\bfa',\bfb')\rVert_\infty$, and therefore we have 
	\begin{equation*}
		\lVert A_{n,m}^{1,q}(\bfa\vert\bfb)[\bfc,\varphi]\rVert_\infty\leq C\lVert c_1'\rVert_2\lVert \varphi\rVert_\infty\bigg( \prod_{i=2}^q \lVert c_i'\rVert_\infty\bigg)  \int_{-\pi}^\pi  |s|^{1/2}\,\rmd s\leq C\lVert c_1'\rVert_2 \lVert\varphi\rVert_{\rmH^{r-1}} \prod_{i=2}^q \lVert c_i\rVert_{\rmH^r},
	\end{equation*}
	which proves~\eqref{eq:Bnmpq_L2_Hr-1_c}.
		
	Recalling \eqref{eq:diffB}, the local Lipschitz continuity property stated at (ii)   follows by showing that for~${q\geq 1}$ and $p\leq n+q+2$ we have
	\begin{equation*}
		\lVert (B_{n+1,m}^{p,q}(\bfa\vert(\bfb,d))-B_{n+1,m}^{p,q}(\bfa\vert(\bfb,\tilde{d}))[\bfc,\varphi]\rVert_2\leq C\lVert c_1'\rVert_2\lVert\varphi\rVert_{\rmH^{r-1}}\lVert d-\tilde{d}\rVert_{\rmH^r}\prod_{i=2}^q \lVert c_i \rVert_{\rmH^r}
	\end{equation*}		 
	for all $(\bfa,\bfb)\in\rmH^r(\Ss)^{m+n}$, $c_1\in\rmH^1(\Ss)$, $c_2,\ldots, c_q, d,\tilde{d}\in\rmH^r(\Ss)$, and $\varphi\in\rmH^{r-1}(\Ss)$, with a positive constant~${C}$ that depends only on $\lVert(\bfa,\bfb,d,\tilde{d})\rVert_{\rmH^r}$. To show this, we compute 
	\begin{equation}\label{eq:Bd-Btd_dec}
	\begin{aligned}
		&\big(B_{n+1,m}^{p,q}(\bfa\vert(\bfb,d))[\bfc,\varphi]-B_{n+1,m}^{p,q}(\bfa\vert(\bfb,\tilde{d}) )[\bfc,\varphi]\big)(\xi)\\
		&\qquad= B_{n,m}^{p,q+1}(\bfa\vert\bfb)[\bfc,d-\tilde{d},\varphi](\xi)-\int_{-\pi}^\pi K(\xi,s)\varphi(\xi-s)\,\rmd s,
	\end{aligned}
	\end{equation}
	where, for $\xi\in\RR$ and $0\neq s\in(-\pi,\pi)$, we set
	\begin{equation}\label{eq:defK}
		K(\xi,s):=F\bigg(\frac{\T{\xi,s}\bfb}{\tss},\frac{\dg{\xi,s}{\bfc}/2}{\tss},\frac{\T{\xi,s}\bfa}{\tss}\bigg)\frac{(\dg{\xi,s}d/2-\T{\xi,s}d)-(\dg{\xi,s}\tilde{d}/2-\T{\xi,s}\tilde{d})}{\tss^{2-p}}.
	\end{equation}
	The first term on the right side of \eqref{eq:Bd-Btd_dec} may be evaluated with the help of \eqref{eq:Bnmpq_L2_Hr-1_c}. Concerning the integral term, using \eqref{eq:tan_h}, \eqref{eq:deriv},~\eqref{eq:d'2}, and~\eqref{eq:x-tanh_Lip}, we have
	\begin{equation*}
		|K(\xi,s)|\leq C\lVert c_1'\rVert_2 \lVert (d-\tilde{d})'\rVert_\infty\bigg(\prod_{i=2}^q \lVert c_i'\rVert_\infty\bigg) |s|^{p+1/2},
	\end{equation*}			
	and therefore
	\begin{equation*}
		\bigg\lVert \int_{-\pi}^\pi K(\cdot,s)\varphi(\cdot-s)\,\rmd s\bigg\rVert_\infty \leq C\lVert c_1'\rVert_2 \lVert\varphi\rVert_{\infty} \lVert d-\tilde{d}\rVert_{\rmH^r} \prod_{i=2}^q \lVert c_i\rVert_{\rmH^r}.
	\end{equation*}
	Therewith, the proof of the local Lipschitz continuity property is complete.\medskip
		
	\noindent{\em Claim {\rm (iii)}.} We again start with the case $p\geq 1$. Using \eqref{eq:deriv},~\eqref{eq:d'2},~\eqref{eq:x-tanh_Lip},~ \eqref{eq:Bd-Btd_dec}, and~\eqref{eq:defK}, we find
	\begin{equation*}
	\begin{aligned}
		&\lVert B_{n+1,m}^{p,q}(\bfa\vert(\bfb,d))[\bfc,\varphi]-B_{n+1,m}^{p,q}(\bfa\vert(\bfb,\tilde{d}))[\bfc,\varphi]\rVert_\infty\\
		&\leq C\lVert\varphi\rVert_\infty \bigg(\prod_{i=1}^q\lVert c_i'\rVert_\infty \bigg)\sup_{\xi\in\Ss}\int_{-\pi}^\pi \big|\dg{\xi,s}{(d-\tilde{d})}\big|\,|s|^{p-2}\,\rmd s\\
		&\leq C\lVert (d-\tilde{d})'\rVert_2\lVert\varphi\rVert_\infty \bigg(\prod_{i=1}^q\lVert c_i'\rVert_\infty\bigg) \int_{-\pi}^\pi |s|^{p-3/2}\,\rmd s\\
		&\leq C\lVert (d-\tilde{d})'\rVert_2 \lVert\varphi\rVert_{\rmH^{r-1}} \prod_{i=1}^q \lVert c_i\rVert_{\rmH^r}.
	\end{aligned}
	\end{equation*}
	In the case $p=0$, after using~\eqref{eq:B=A+C} and~\cite[Lemma~A.1~(ii)]{Bohme.2024}, it remains to estimate the term
	\begin{equation*}
		A_{n+1,m}^{1,q}(\bfa\vert(\bfb,d))[\bfc,\varphi]-A_{n+1,m}^{1,q}(\bfa\vert(\bfb,\tilde{d}))[\bfc,\varphi],
	\end{equation*}
	for which, appealing to ~\eqref{eq:deriv}--\eqref{eq:FT-Fd},  we obtain
	\begin{equation*}
	\begin{aligned}
		\lVert A_{n+1,m}^{1,q}(\bfa\vert\bfb,d)[\bfc,\varphi]-A_{n+1,m}^{1,q}(\bfa\vert\bfb,\tilde{d})[\bfc,\varphi]\rVert_\infty
		&\leq C\lVert (d-\tilde{d})'\rVert_2\lVert\varphi\rVert_\infty \bigg(\prod_{i=1}^q\lVert c_i'\rVert_\infty\bigg) \int_{-\pi}^\pi |s|^{1/2}\,\rmd s\\
		&\leq C\lVert (d-\tilde{d})'\rVert_2 \lVert\varphi\rVert_{\rmH^{r-1}} \prod_{i=1}^q \lVert c_i\rVert_{\rmH^r},
	\end{aligned}
	\end{equation*}
	which proves \eqref{eq:Bnmpq_L2_Hr-1_b}.
	\end{proof}
 	
	The next lemma provides a mapping property for the operator $B_0$, introduced in~\eqref{eq:B0}, which is used in the proof of Theorem~\ref{Thm:Ftp}.
 	\begin{Lemma}\label{Lem:B0_H1_H2}
 	Given $f\in\rmH^2(\Ss)$, there exists a constant $C>0$ that depends only on $\lVert f\rVert_{\rmH^2}$ such that for all $\varphi\in\rmH^1(\Ss)$ we have
 	\begin{equation}\label{eq:B0_H1_H2}
 		\lVert B_0(f)[\varphi]\rVert_{\rmH^2}\leq C\lVert\varphi\rVert_{\rmH^1}.
 	\end{equation}
 	\end{Lemma}

 	\begin{proof}
	Given $f\in\rmH^2(\Ss)$ and $\varphi\in\rmH^1(\Ss)$, it follows from \cite[Lemma~A.7]{Bohme.2024} that $B_0(f)[\varphi]\in \rmH^1(\Ss)$ with 
 	\begin{equation*}
 		(B_0(f)[\varphi])'=f'B_2(f)[\varphi]+B_1(f)[\varphi].
 	\end{equation*}
 	 The claim is now a direct consequence of \eqref{eq:B_by_Bnmpq},  Lemma~\ref{Lem:Bnmpq_H1_H1}, and \cite[Lemma~A.7]{Bohme.2024}. 
 	\end{proof}

 	\section{The behavior of the pressure and velocity near the interface and in the far-field}\label{Sec:B}	

 	In this section, we show that the function $(v^\pm,q^\pm)$ defined in \eqref{eq:defwq} satisfies the boundary conditions \eqref{eq:Ftp_b}$_{3-4}$, the far-field condition \eqref{eq:Ftp_b}$_5$, and \eqref{eq:stress_inf} under the assumptions of Theorem~\ref{Thm:Ftp_b}, that is for~$f\in\rmH^3(\Ss)$ and $\beta\in\rmH^2(\Ss)^2$.
 	\begin{Lemma}\label{Lem:vq_bd}
 	We have $v^\pm\in\rmC^1(\overline{\Omega^\pm},\RR^2)$ and $q^\pm\in\rmC(\overline{\Omega^\pm})$. Furthermore, it holds that
 	\begin{equation}\label{eq:bd_terms}
 		\left.
 		\begin{array}{rcll}
 			\{v\}^\pm&=&\big(-\DD(f)[\beta]\pm\frac{1}{2}\beta\big)\circ\Xi^{-1}&\quad\textup{on } \Gamma,\\[1ex]{}
 			[T_1(v,q)]\tnu&=&0&\quad\textup{on } \Gamma
 		\end{array}
 		\right\}
 	\end{equation}
 	and 
 	\begin{equation}\label{eq:Tb_lim}
		T_1 (v^\pm,q^\pm)(x)\to0 \qquad\text{for $x_2\to\pm\infty$ uniformly in $x_1\in 	\Ss$.}
	\end{equation}
 	\end{Lemma}

 	\begin{proof}
 	We first recall for $\varphi\in\rmH^1(\Ss)$ the definition of the operators $Z_i(f)[\varphi]$, $1\leq i\leq 4$,  in~\eqref{eq:defZ} and the notation~\eqref{eq:deflimits}, and infer from \cite[Lemma~C.1]{Bohme.2024} that 
 	\begin{equation}\label{eq:reg_Zi}
 		Z_i(f)[\varphi]\in\rmC^\infty(\Omega^\pm)\cap\rmC(\overline{\Omega^\pm}),\qquad1\leq i\leq 4,
 	\end{equation} 
 	with
 	\begin{equation}\label{eq:Z_pm}
		\left
		\{\begin{pmatrix}
			Z_{1}(f)[\varphi]\\ 
			Z_{2}(f)[\varphi]\\ 
			Z_{3}(f)[\varphi]\\ 
			Z_{4}(f)[\varphi] 
		\end{pmatrix}\right\}^\pm\circ\Xi 
		= 
		\begin{pmatrix}
			B_{1}(f)[\varphi]\\ 
			B_{2}(f)[\varphi]\\ 
			B_{3}(f)[\varphi]\\ 
			B_{4}(f)[\varphi] 
		\end{pmatrix}
		\pm\frac{1}{2\omega^4}
		\begin{pmatrix}
			-2f' \omega^2\\ 
			2\omega^2 \\
			- 4f'^2 \\  
			f'-f'^3  
		\end{pmatrix} 
		\varphi,
	\end{equation}
	with $B_i(f)$, $1\leq i\leq 4$, defined in~\eqref{eq:B_by_Bnmpq}. Since $\beta\in\rmH^2(\Ss)^2$, we may now conclude from \eqref{eq:w_by_Z} and~\eqref{eq:q_by_Z}   that~$v^\pm\in\rmC(\overline{\Omega^\pm},\RR^2)$ and~$q^\pm\in\rmC(\overline{\Omega^\pm})$. Recalling the definition~\eqref{eq:D(f)} of the double layer potential $\DD(f)$, the boundary condition~\eqref{eq:bd_terms}$_1$ follows  easily from \eqref{eq:Z_pm}.
 	
	Given a function $\mcZ\in\rmC^1((\RR\times\Ss)\setminus\{0\})$, we define the vector field 
	\begin{equation*}
		\rot\, \mcZ:= (\rot^1 \mcZ, \rot^2 \mcZ):= (-\partial_2 \mcZ, \partial_1 \mcZ)\in\rmC((\RR\times\Ss)\setminus\{0\},\RR^2)
	\end{equation*}
	and compute for $1\leq j\leq 4$ and $\varphi\in\rmH^1(\Ss)$, using integration by parts, 
	\begin{equation}\label{eq:rot_z}
		\frac{1}{2\pi}\int_{-\pi}^\pi \sum_{i=1}^2(\rot^i z_j)(r)\nu_i\varphi\omega\,\rmd s = \frac{1}{2\pi}\int_{-\pi}^\pi (f'\partial_2 z_j(r)+\partial_1 z_j(r))\varphi\,\rmd s= Z_j(f)[\varphi'].
	\end{equation}
	Defining
	\begin{equation}\label{eq:mcZ}
		(\mcZ_0,\,\mcZ_1,\,\mcZ_2,\,\mcZ_3):=(z_1-z_4,\,z_2+z_3,\,z_4,\,z_2-z_3)\qquad \text{in $(\Ss\times\RR)\setminus\{0\}$},
	\end{equation}
 	we infer from \eqref{eq:WQexp} that
 	\begin{equation*}
 		\mcW_1=\frac{1}{4\pi}
 		\begin{pmatrix}
 			2\mcZ_0 & \mcZ_1\\
 			\mcZ_1 & 2\mcZ_2
 		\end{pmatrix}
 		\qquad \text{and}\qquad\mcW_2=\frac{1}{4\pi}
 		\begin{pmatrix}
 			\mcZ_1 & 2\mcZ_2\\
 			2\mcZ_2 & \mcZ_3
 		\end{pmatrix},
 	\end{equation*}
 	and  the identities \eqref{eq:defz}--\eqref{eq:grad_z_i} lead us to
 	\begin{equation}\label{eq:rotW}
 	\begin{aligned}
 		\rot^i \mcZ_0&=-2\pi\partial_2\mcW_1^{i,1},\\
 		\rot^i \mcZ_1 &=4\pi\partial_1 \mcW_1^{i,1}=-4\pi\partial_2 \mcW_2^{i,1}=-4\pi\partial_2 \mcW_1^{i,2},\\
 		\rot^i \mcZ_2 &=2\pi\partial_1 \mcW_1^{i,2}=2\pi\partial_1 \mcW_2^{i,1}=-2\pi\partial_2 \mcW_2^{i,2},\\
 		\rot^i \mcZ_3&=4\pi\partial_1\mcW_2^{i,2},\qquad i= 1,\,2.
 	\end{aligned}
 	\end{equation}
 	Moreover, for $1\leq i\leq 2$, we obtain by combining  \eqref{eq:grad_z_i},  \eqref{eq:WQexp}, and \eqref{eq:mcZ} the following identities 
 	\begin{equation}\label{eq:rotQ}
 		\rot^i(\mcZ_0+\mcZ_2)=2\pi\mcQ^{i,2} \qquad\text{and}\quad \rot^i(\mcZ_1+\mcZ_3)=-4\pi\mcQ^{i,1}.
 	\end{equation}
 	Given  $x\in(\Ss\times\RR)\setminus\Gamma$ and $1\leq j,\ell\leq 2$, the definitions \eqref{eq:stress} and \eqref{eq:defwq} yield, after interchanging in~\eqref{eq:defwq}$_1$ partial differentiation with respect to $x_1$ and $x_2$ and integration with respect to $s$, 
 	\begin{equation*}
 		(T_1(v,q))_{j\ell}(x)=\int_{-\pi}^{\pi}\sum_{i,k=1}^2\big(-\delta_{j\ell}\mcQ^{i,k}+\partial_\ell\mcW_j^{i,k}+\partial_j\mcW_\ell^{i,k}\big)(r)\nu_i(s)\beta_k(s)\omega(s)\,\rmd s,
 	\end{equation*}
	and together with \eqref{eq:rot_z}--\eqref{eq:rotQ} we arrive at
 	\begin{equation*}
 		T_1(v,q)=
 		\begin{pmatrix}
 			(2Z_2+Z_3)(f)[\beta_1']+(2Z_4-Z_1)(f)[\beta_2'] && (2Z_4-Z_1)(f)[\beta_1']-Z_3(f)[\beta_2']\\
 			(2Z_4-Z_1)(f)[\beta_1']-Z_3(f)[\beta_2'] && -Z_3(f)[\beta_1']-(Z_1+2Z_4)(f)[\beta_2']
 		\end{pmatrix}
 	\end{equation*}
 	in $(\Ss\times\RR)\setminus\Gamma$. Invoking \eqref{eq:Z_pm}, the boundary condition \eqref{eq:bd_terms}$_2$ can be easily verified, while \eqref{eq:Tb_lim} follows in view of Lemma~\ref{Lem:vq_ff} below.  
 	Since~${q^\pm\in\rmC(\overline{\Omega^\pm})}$ and, by similar arguments
	\begin{equation*}
		\partial_1v_2=Z_4(f)[\beta_1']+\frac{1}{2}(Z_2-Z_3)(f)[\beta_2']\qquad\text{in $(\Ss\times\RR)\setminus\Gamma$},
	\end{equation*}
	we may also conclude that $v^\pm\in\rmC^1(\overline{\Omega^\pm},\RR^2)$, and the proof is complete.
 	\end{proof}

 	Next, we consider the far-field behavior of $ v$ and $q$.

 	\begin{Lemma}\label{Lem:vq_ff}
 	For $x_2\to\pm\infty$ we have, uniformly in $x_1\in\Ss$,
 	\begin{equation}
 		v^\pm(x)\to \pm\frac{1}{2}
 		\begin{pmatrix}
 			\langle \beta_1-f'\beta_2\rangle\\
 			\langle \beta_2-f'\beta_1\rangle
 		\end{pmatrix}
 		\qquad\text{and }\qquad 
 		q^\pm(x)\to 0.
 	\end{equation}
 	\end{Lemma}

 	\begin{proof}
 	Given $\varphi\in\rmL^2(\Ss)$, straightforward computations show that for $x_2\to\pm\infty$ we have  
 	\begin{equation}\label{eq:Zi_lim}
 		Z_2(f)[\varphi](x)\to \pm\langle \varphi\rangle \qquad\text{and}\qquad Z_i(f)[\varphi](x)\to 0,\quad i\in\{1,\, 3,\, 4\},
 	\end{equation}
 	uniformly in $x_1\in\Ss$. The claim follows by using these limits in the context of  \eqref{eq:w_by_Z}--\eqref{eq:q_by_Z}.
 	\end{proof}\medskip






\bibliographystyle{siam}
\bibliography{references_final}
\end{document}